\documentclass[12pt]{amsart}  	

\usepackage[margin=1in]{geometry}  

\usepackage[all]{xy}						
\usepackage{amssymb}
\usepackage{amscd,latexsym,amsthm,amsfonts,amssymb,amsmath,amsxtra}
\usepackage[mathscr]{eucal}
\usepackage[colorlinks]{hyperref}
\usepackage{mathtools}

\usepackage{chngcntr}
\numberwithin{equation}{section}
\newcounter{keepeqno}

\hypersetup{linkcolor=black, citecolor=black}

\newcommand{\BA}{{\mathbb {A}}}

\newcommand{\BC}{{\mathbb {C}}}

\newcommand{\BF}{{\mathbb {F}}}
\newcommand{\BG}{{\mathbb {G}}}

\newcommand{\BN}{{\mathbb {N}}}
\newcommand{\BO}{{\mathbb {O}}}

\newcommand{\BR}{{\mathbb {R}}}

\newcommand{\BW}{{\mathbb {W}}}

\newcommand{\BZ}{{\mathbb {Z}}}

\newcommand{\CB}{{\mathcal {B}}}

\newcommand{\CCD}{{\mathcal {D}}}

\newcommand{\CF}{{\mathcal {F}}}

\newcommand{\CH}{{\mathcal {H}}}
\newcommand{\CI}{{\mathcal {I}}}

\newcommand{\CL}{{\mathcal {L}}}
\newcommand{\CM}{{\mathcal {M}}}

\newcommand{\CO}{{\mathcal {O}}}

\newcommand{\CS}{{\mathcal {S}}}

\newcommand{\CZ}{{\mathcal {Z}}}

\newcommand{\FC}{{\mathfrak {C}}}

\newcommand{\Ff}{{\mathfrak {f}}}

\newcommand{\Fo}{{\mathfrak {o}}}

\newcommand{\RB}{{\mathrm {B}}}
\newcommand{\RC}{{\mathrm {C}}}
\newcommand{\RD}{{\mathrm {D}}}
\newcommand{\RE}{{\mathrm {E}}}
\newcommand{\RF}{{\mathrm {F}}}
\newcommand{\RG}{{\mathrm {G}}}

\newcommand{\RI}{{\mathrm {I}}}

\newcommand{\RK}{{\mathrm {K}}}
\newcommand{\RL}{{\mathrm {L}}}

\newcommand{\RN}{{\mathrm {N}}}

\newcommand{\RR}{{\mathrm {R}}}
\newcommand{\RS}{{\mathrm {S}}}
\newcommand{\RT}{{\mathrm {T}}}
\newcommand{\RU}{{\mathrm {U}}}

\newcommand{\RZ}{{\mathrm {Z}}}

\newcommand{\cusp}{{\mathrm{cusp}}}

\newcommand{\GL}{{\mathrm{GL}}}

\newcommand{\Mat}{{\mathrm{Mat}}}

\newcommand{\rank}{{\mathrm{rank}}}

\newcommand{\PGL}{{\mathrm{PGL}}}

\renewcommand{\Re}{{\mathrm{Re}}}

\newcommand{\Res}{{\mathrm{Res}}}

\newcommand{\rb}{{\mathrm{b}}}

\newcommand{\Eis}{{\mathrm{Eis}}}

\newcommand{\SO}{{\mathrm{SO}}}

\newcommand{\SU}{{\mathrm{SU}}}

\newcommand{\st}{{\mathrm{st}}}

\newcommand{\tr}{{\mathrm{tr}}}

\newcommand{\ud}{\,\mathrm{d}}
\newcommand{\vol}{{\mathrm{vol}}}

\def\std{\rm std}

\newcommand{\sslash}{\mathbin{/\mkern-6mu/}}

\newtheorem{thm}{Theorem}[section]
\newtheorem{dfn}[thm]{Definition}
\newtheorem{rmk}[thm]{Remark}

\newtheorem{prp}[thm]{Proposition}
\newtheorem{lem}[thm]{Lemma}
\newtheorem{cor}[thm]{Corollary}

\makeatletter

\newcommand{\Rmnum}[1]{\expandafter\@slowromancap\romannumeral #1@}
\makeatother

\begin{document}

    \title[Poisson Summation Formula and Kuznetsov Trace Formula on $\GL_2$]{Beyond Endoscopy: Poisson Summation Formula and Kuznetsov Trace Formula on $\GL_2$}

    \date{\today}
    
    \author{Zhaolin Li}
    
    \address{School of Mathematics, University of Minnesota, 206 Church St. S.E., Minneapolis, MN 55455, USA.
}
    \email{li001870@umn.edu}
    \date{\today}

  \begin{abstract}
In the first part of this paper, we present a direct proof of a Poisson summation formula on the Whittaker space of $\GL_2$, which underlies the local Hankel transform computed in \cite{Jac03}. In the second part, we derive a Kuznetsov-type trace formula for incorporating non-standard test functions. This formulation reveals that the Poisson summation formula naturally yields the functional equation for the standard $L$-functions of $\GL_2$.
\end{abstract}
    \maketitle
    \tableofcontents

\section{Introduction}

In 1895, B. Riemann represented the Riemann zeta function as the Mellin transform of a theta series
\begin{align*}
	\pi^{-\frac{s}{s}}\Gamma(s)\zeta(s)=\int_{0}^{\infty}y^{\frac{s}{2}}\sum_{n=1}^{\infty}e^{-n^2\pi y }\ud^{\times }y
\end{align*}
and proved the functional equation of it by using the Poisson summation formula of the theta series, see \cite{Rie59}. This idea was further developed by K. Iwasawa and J. Tate and was reformulated in Tata's thesis \cite{Tat67} in the language of adeles using the theory of local and global harmonic analysis. R. Godement and H. Jacquet generalized this work to study the analytic properties of standard automorphic $L$-functions of $\GL_n$. The key input from the harmonic analysis to Godement-Jacquet theory is the classical Fourier transform on the affine space $\Mat_n$ of $n\times n $ matrices and the global Poisson summation formula that is responsible for the functional equation of standard $L$-functions. 

Around the year 2000, A. Braverman and D. Kazhdan (\cite{BK00}) proposed a framework to establish the Langlands conjecture for general $L$-functions $L(s,\pi,\rho)$ via local and global harmonic analysis on the group, which generalizes
the method of Tate's thesis and the work of Godement and Jacquet to the great generality, which is refined by B. Ngo in \cite{Ngo20}. We also refer the reader to \cite{Luo24} and \cite{JLX24}. Under this framework, if generalized Fourier transforms and Poisson summation formulae could be established, there is a hope to prove the Langlands conjecture on the meromorphic continuation and the global functional equation of automorphic $L$-functions. The difficult part is establishing the Poisson summation formulae. We also refer to \cite{BK99, BK00} and \cite{GL19, GL21, CG24} for some cases. D. Jiang and Z. Luo also established a version of Poisson summation formulae on $\BG_m$ with automorphic representations involved; see \cite{JL23}.

Another approach to the study of automorphic $L$-functions is through the lens of the functoriality conjecture. However, $L$-functions also play a central role in Langlands' broader vision of Beyond Endoscopy: they must be explicitly introduced to weight the trace formula, and their residues at poles are then used to detect the image of functoriality lifts. It is worth noting that the original idea involved weighting by the logarithmic derivatives of $L$-functions, a method now recognized as particularly challenging to implement in this context.  The Beyond Endoscopy has limited successes: In his thesis \cite{Ven04}, A. Venkatesh used the Kuznetsov trace formula successfully re-established functoriality from tori to $\GL_2$. S. Ali Altug's work \cite{Ag15, Ag17, Ag20} isolated the contribution of special representations for $\GL_2$ from the Arthur-Selberg trace formula and showed that the standard $L$-functions do not pick up functorial lifts in the cuspidal spectrum. A series paper by P. Hermann \cite{Her11, Her12} also used the Kuznetsov trace formula to give a new proof of the analytic continuation and functional equation of standard $L$-functions and Rankin-Selberg $L$-functions. As noted in \cite[Section 6]{Sak19}, Hermann's techniques bear a resemblance to Jacquet's formula in \cite{Jac03}.

Motivated by the work of \cite{FLN10}, one might consider employing non-standard test functions in the trace formula, which would naturally lead to the appearance of $L$-functions in the spectral expansion. By carefully analyzing the geometric side, one could potentially extract information about $L$-functions. A key ingredient in this approach is a Poisson summation formula on the Hitchin base. As highlighted in \cite{Sak23b}, the Kuznetsov quotient appears to serve as the foundation for all relative trace formulas. In this paper, we introduce Godement-Jacquet test functions into the Kuznetsov trace formula and establish a global Poisson summation formula corresponding to the local Hankel transforms computed by Jacquet in \cite{Jac03}. Specifically, we prove the commutativity of the following diagram:
\begin{align*}
	\xymatrix{
\CCD^-_{L\left(\std,\frac{1}{2}\right)}(\RN(\BA),\psi\backslash \GL_2(\BA)/\RN(\BA),\psi) \ar@{->}[rd]^{\mathrm{KTF}} \ar[rr]^{\CH_{\std,\BA}}&& \CCD^-_{L\left(\std^{\vee},\frac{1}{2}\right)}(\RN(\BA),\psi\backslash \GL_2(\BA)/\RN(\BA),\psi)\ar@{->}[ld]^{\mathrm{KTF}}\\
 &\BC & 
}
\end{align*}
to obtain the analytic continuation and functional equations of standard $L$-functions of $\GL_2$. In the above diagram,
\[
    	\CCD^-_{L\left(\std,\frac{1}{2}\right)}(\RN(\BA),\psi\backslash \RG(\BA)/\RN(\BA),\psi):=\otimes^{\prime}_{\nu}\CCD^-_{L\left(\std,\frac{1}{2}\right)}(\RN_{\nu},\psi_{\nu}\backslash \RG_{\nu}/\RN_{\nu},\psi_{\nu})
\]
is the restricted tesor product of local spaces with basic vectors given in Definition \ref{dfn_unramified}, and the local space $\displaystyle{\CCD^-_{L\left(\std,\frac{1}{2}\right)}(\RN_{\nu},\psi_{\nu}\backslash \RG_{\nu}/\RN_{\nu},\psi_{\nu})
}$ is defined to be the space of orbital integrals of Schwartz functions on $\Mat_2(k_{\nu})$ with a suitable normalization, see Definition \ref{dfn_localspace}. Using the coordintes of the maximal torus, $\mathrm{KTF}$ has the following explicit expression:
\[
	\mathrm{KTF}_{L\left(\std,\frac{1}{2}\right)}(f):=\sum_{t\in \RT(k)}\frac{f}{(\ud t)^{\frac{1}{2}}}(t)+\sum_{t_1\in k^{\times}}\CO_{u,1,\BA}(f)(t_1)+\sum_{t_2\in k^{\times}}\CO_{u,2,\BA}(f)(t_2),
	\]
a summation over the rational points of the maximal torus with boundary terms explicitly given, which are hided in the expression of local Hankel transforms. For unexplained notations, see Section \ref{sec_local} and Section \ref{sec_global}.

Notably, this paper relies solely on the local theory developed in \cite{GJ72} and does not invoke any global theory, including the global Poisson summation formula on $\Mat_2(\BA)$. This represents a novel approach to establishing the analytic continuation and functional equations of standard $L$-functions for $\GL_2$. One advantage of this method is that it avoids reliance on the highly non-trivial Poisson summation formula for general reductive groups. Moreover, this approach offers a more systematic framework for studying automorphic $L$-functions. Provided that the appropriate local theory, as proposed in \cite{Ngo20}, is available, along with the machinery of a general trace formula and Poisson summation formula on orbital integrals, the analytic continuation and functional equations can be derived. It is also worth noting that Z. Luo and B. C. Ngo have recently made significant progress in the development of the local theory proposed by Braverman, Kazhdan, and Ngo, as detailed in \cite{LN24}.

Motived by the work \cite{Sak19a} of Y. Sakellaridis, this paper presents two distinct proofs of the Poisson summation formula. The first proof is indirect, achieved through a descent from the Poisson summation formula on $\Mat_2(\BA)$. This part can be regarded as a verification from the perspective of this paper. However, the primary focus lies on the second proof, which provides a direct and self-contained argument, relying on a two-step application of the classical one-dimensional Poisson summation formula. The overarching principle guiding this approach is that the establishment of local Hankel transforms at the level of orbital integrals provides critical insights into the formulation of the boundary terms in the Poisson summation formula. By decomposing local Hankel transforms into compositions of classical Fourier transforms with appropriate correction factors, it becomes feasible to derive the Poisson summation formula without explicit reference to the upstairs space. Once the Poisson summation formula is established at the level of orbital integrals, it can be integrated into the trace formula to derive the functional equation of automorphic $L$-functions, a central objective of the latter part of this paper. 

In this context, we construct a Kuznetsov-type trace formula utilizing non-standard test functions coming from Schwartz functions on $\Mat_2$. As anticipated, the geometric side of this trace formula exhibits a close relationship with the Poisson summation formula developed herein. On the spectral side, the expansion has been previously established in \cite{Wu}. When examining the Eisenstein series component of the trace formula, it is expected to reduce to the lower-rank case, specifically to the setting of Tate's thesis \cite{Tat67}. To isolate a single cuspidal spectrum, we adopt a technique analogous to that employed in \cite{Sak19a}, interpreting the trace formula as a measure on the unramified Hecke algebra. This methodology enables us to establish the analytic continuation and functional equations of the standard $L$-functions. 

A noteworthy phenomenon observed in this analysis is that during the process of analytic continuation, the boundary terms arising from the Poisson summation formula manifest as residues of the Eisenstein series component within the spectral expansion. These boundary terms, originating from the geometric side, ultimately cancel out with the corresponding Eisenstein series contributions. For a comprehensive treatment of these results, the reader is directed to Section \ref{sec_cont} and Section \ref{sec_feq}.

The key elements of this paper are the Poisson summation formula on the level of orbital integrals and the trace formula. In the Beyond Endoscpoy program, by employing the non-standard test functions in the trace formula, we can encode the information of $L$-functions in the spectral expansion of the trace formula. The primary challenge on the geometric side lies in formulating a Poisson summation formula. The central idea of this approach is that we do not  necessarily require a Poisson summation formula on the geometric space itself; instead, a Poisson summation formula on the level of orbital integrals suffices for applying the trace formula. In this framework, the local Hankel transforms serve as the foundation of our Poisson summation formula. Just as we do in this paper, as long as we can decompose the Hankel transforms into a composition of classical one-dimensional Fourier transforms with certain correction factors, these classical Fourier transforms will guide us in constructing a global Poisson summation formula. Rather than focusing on the  Poisson summation formula in the upstairs space, the Poisson summation formula on the level of orbital integrals appears more attainable from this perspective. The remaining challenge is to determine how to decompose the local Hankel transforms, a problem that remains unresolved in general. However, we have made progress in specific cases, as evidenced in \cite{Sak19, Sak22, Sak23}.
\subsection{Organization of the Paper}
In Section \ref{sec_local}, we will study properties of local orbital integrals and recall the Hankel transforms established by Jacquet in \cite{Jac03}. Then, in Section \ref{sec_global}, we will write down the global Poisson summation formula based on the local Hankel transforms and give two proofs of it. In the following section, we will establish the Kuznetsov trace formula with Schwartz functions from $2\times 2$ matrices. This can be viewed as one example of a trace formula with non-standard test functions. It turns out that the Poisson summation formula that we constructed in the previous section is closely related to the geometric side of the Kuznetsov trace formula. Then, in section \ref{sec_cont}, we will study the analytic continuation of the geometric side and the Eisenstein series part from the spectrum side of the trace formula. For the geometric side, we will use the Poisson summation formula in Section \ref{sec_global}, and the Eisenstein series part is essentially reduced to Tate's thesis. In this way, we can establish the analytic continuation of the cuspidal part. The next section will focus on the functional equation. It also consists of two parts, the first of which is the functional equation of the geometric part corresponding to the Poisson summation formula. Then, we will compare the Eisenstein series parts according to the functional equation of the $\GL_1$ integrals and obtain the functional equation for the cuspidal part. Finally, in the last section, we will isolate the spectrum and obtain the analytic continuation and functional equation of standard $L$-functions for one single cuspidal representation.

\subsection{Notations and Convensions}

\begin{itemize}
    \item In Section \ref{sec_local}, we will use $F$ to denote a local field, and use $\psi$ to denote an additive character. In the later sections, we will use $k$ to denote a global field, $\BA$ its ring of adeles, and $\psi:k\backslash\BA\rightarrow \BC^{\times}$ a global additive adelic character.
    \item We will reserve $\RG$ to denote $\RG=\GL_2$ in this paper. Let $\RT$ be the universal Cartan of $\GL_2$, which can be identified with the maximal torus consisting of all diagonal matrices in $\GL_2$. Let $\RN$ be the unipotent subgroup of $\RG$ consisting of all strictly upper triangular matrices. We also view $\psi$ as an additive character on $\RN$ by identifying $\RN\cong\BG_a:\begin{pmatrix}
    	1 &  n \\0 & 1
    \end{pmatrix}\mapsto n$. Let $w=\begin{pmatrix}
    	0& 1 \\ 1 &0
    \end{pmatrix}$ be the longest Weyl element, and $\RB=\RT\RN$. By $\epsilon_i^{\vee}$ ($i=1,2$), we denote the character of the $i$-th coordinate of the torus of diagonal elements, written additively. Write $e^\alpha_1$ for the positive root, written multiplicatively. 
    \item When the meaning is clear from the context, for a variety $X$ over a local field $F$, we denote the set of $F$-rational points $X(F)$ simply by $X$. If $X$ is over a global field $k$, we will use $[X]:=X(k)\backslash X(\BA)$ to denote its adelic quotient.
    \item Let $\FC=\RN\backslash \RG \sslash \RN$, then $\RT$ can be embedded as an open subset of $\FC$ via $t\mapsto wt$.
    \item For a smooth function $\Psi$ on $\RG$, we will use $\CO_t^{\psi}(\Psi)$ to denote its orbital integrals
    \[
    \CO_t^{\psi}(\Psi):=\int_{\RN\times \RN}\Psi(n_1wtn_2)\psi^{-1}(n_1n_2)\ud n_1\ud n_2
    \]
    for $t\in \RT$.
    \item Since $\RG$ acts on $\RG$ both on the left and on the right, we will use $\RL$ and $\RR$ to denote its left and right action, and hence $\RG$ also acts on functions of $\RG$ by 
    \[
    \RL(g)\Phi(x):=\Phi(g^{-1}x),\;\RR(g)\Phi(x):=\Phi(xg),
    \]
    for $g,x\in \RG$, and $\Phi:\RG\rightarrow\BC$.
    \item For a smooth variety $X$ over a local field, we denote by $\CF(X)$ the space of Schwartz functions on the $F$-points of $X$. It means smooth rapid decay function in the Archimedean case, see \cite{AG08}, and smooth functions of compact support in the non-Archimedean case. Similarly, we will use $\CS(X)$ to denote the space of Schwartz measures and $\CCD(X)$ the space of Schwartz half-densities. Often $X(F)$ is a homogeneous $\RG(F)$-space and admits an invariant positive measure $\ud x$, then we have $\CS(X)=\CF(X)\ud x$ and $\CCD(X)=\CF(X)(\ud x)^{\frac{1}{2}}$.
    \item For the notation of the Fourier transforms, let $f$ be an half-density on $\RT$ and $\lambda^{\vee}$ be any cocharacter of $\RT$, $s\in\BC$, we define
	\[
\CF^{\psi}_{\lambda^{\vee},s}f(\xi):=\int_{\BG_m}|x|^s\psi(x)f(\lambda^{\vee}(x)^{-1}\xi)\ud^{\times}x,
	\]
whenever this integral makes sense.
\item Let $f$ be a one-variable function. We also use the usual Fourier transform on the Affine line
\[
\BF_{\psi}(f)(x):=\int_{\BA^1} f(y)\psi^{-1}(xy)\ud y 
\]
when the integral converges. When $f$ is a multi-variable function, we will use $\BF_{\psi, i}(f)$ to denote this Fourier transform for the $i$-th variable with other variables fixed. We will also identify $\Mat_2\cong\BA^4:\begin{pmatrix}
	x_{11} & x_{12} \\ x_{21} & x_{22} 
\end{pmatrix}\mapsto (x_{11},x_{12},x_{21},x_{22})$ and use $\BF_{\psi,i}$ to denote the Fourier transform with respect to the $i$-th coordinate using this identification.
\end{itemize}

\subsection{Acknowdegement}
I want to thank my advisor, Prof. Dihua Jiang, for providing continued support and encouragement, providing practical and thought-provoking viewpoints. I am very grateful to Prof. Yiannis Sakellaridis for his suggestions during the number theory seminar at Johns Hopkins University. I would also like to thank Xinchen Miao and Guodong Xi for many helpful conversions and comments.

\section{Local Theory}\label{sec_local}
    \subsection{Speces of Orbital Integrals}\label{sec_local spaces}
    Let $F$ be a local field in this section. We will use $\Fo_F$ to denote the ring of integers of $F$, $\varpi$ a uniformizer. Let $\RG=\GL_2$. We will use $\Mat_2$ to denote the space of $2\times 2$ matrices. For the purpose of Hankel transforms responsible for the functional equations, following \cite{Sak19} it is easier to work with half-densities, but one may also feel free to transfer densities with functions and measures, see also \cite{Sak19}.
    \begin{dfn}
    	Denote $\CCD(\Mat_2)$ the pull-back of Shcwartz half-densities on $\Mat_2$ to $\GL_2$ via the natural embedding $g\mapsto g$. Note that a Schwartz half-density on $\Mat_2$ is of the form $\Phi(x)(\ud ^+x)^{\frac{1}{2}}$, where $\Phi(x)$ is a usual Schwartz function on $\Mat_2$ and $\ud ^+x$ is an additive Haar measure on $\Mat_2$. Then its pull-back to $\GL_2$ is $|\det g|\Phi(g)(\ud g)^{\frac{1}{2}}$, where $\displaystyle{\ud g=\frac{\ud^+ g}{|\det g|^2}}$ is the induced Haar measure on $\RG$. When there is no confusion, we will not distinguish the Schwartz half-density on $\Mat_2$ and its pull-back to $\RG$.
    \end{dfn}
\begin{rmk}
	Here we consider the natural action of $\RG\times \RG$ on both $\RG$ and $\Mat_2$ by
	\[
	x.(g_1,g_2):=g_1^{-1}xg_2,
	\]
	for $x\in \RG$ or $\Mat_2$ and $(g_1,g_2)\in \RG$. Later, we will consider another embedding of $\RG$ into $\Mat_2$ via the inverse map and modify the $\RG$-action.
\end{rmk}

\begin{dfn}
	For a smooth half-density $\Psi(g)(\ud g)^{\frac{1}{2}}$ on $\RG$, we define its twisted push-forward to the open subset $\RT$ of $\FC$ as
\[
\delta^{\frac{1}{2}}(t)\CO_t^{\psi}(\Psi)(\ud t)^{\frac{1}{2}},
\]
where for $t=\begin{pmatrix}
	t_1 & 0 \\ 0  & t_2
\end{pmatrix}$, $\displaystyle{ \delta(t)=\left|\frac{t_1}{t_2}\right|}$, 
\[
\CO_t^{\psi}(\Psi)=\int_{\RN\times \RN}\Psi(n_1wtn_2)\psi^{-1}(n_1n_2)\ud n_1\ud n_2,
\]
and $\ud t$ is the Haar measure on the torus $\RT$ such that
\begin{align}\label{eq_int formula}
\int_{\RG}\Psi(g)\ud g=\int_{\RT} \delta(t) \CO_t^{\mathrm{triv}}(\Psi)\ud t,
\end{align}
whenever this integral converges.
\end{dfn}

\begin{rmk}
	In the following, we fix an additive character $\psi:F\rightarrow\BC$, and choose the Haar measure on $\Mat_2$ to be $\ud^+ x:=\ud x_{11}\ud x_{12}\ud_{21}\ud x_{22}$ for $x=\begin{pmatrix}
		x_{11} & x_{12} \\ x_{21} & x_{22}
	\end{pmatrix}$ with each $\ud x_{ij}$ being the self-dual Haar measure on the additive group $F$ with respect to $\psi$. Then we choose $\displaystyle{\ud g:=\frac{\ud^+ g }{ |\det g|^2}}$. We also choose the Haar measure on $\RN$ to be the self-dual Haar measure with respect to $\psi$. Once all of these measures are chosen, the Haar measure $\ud t$ on the torus $\RT$ is the one such that the integration formula (\ref{eq_int formula}) holds. Using these choices, when $\psi$ is unramified, we have $\displaystyle{\int_{\RT(\Fo_F)}1\ud t=1.}$
\end{rmk}

\begin{dfn}\label{dfn_localspace}
	The image of the twisted push-forward of $ \mathcal{D}(\Mat_2)$ is denoted by \[
	\mathcal{D}^{-}_{L(\std,\frac{1}{2})}(\RN,\psi\backslash \RG/\RN,\psi).\]
	More explicitly, for $\Phi\in\CF(\Mat_2)$, the twisted push-forward of $|\det g|\Phi(g)(\ud g)^{\frac{1}{2}}$ to $\RT$ is
\[
|\det t|\delta^{\frac{1}{2}}(t)\CO_t^{\psi}(\Phi)(\ud t)^{\frac{1}{2}}.
\]

\end{dfn}

\begin{dfn}\label{dfn_unramified}
	We define the basic vector $f^{\circ}$ to be the twisted push-forward of $1_{\Mat_2(\Fo_F)}$, the characteristic function of $\Mat_2(\Fo_F)$.
\end{dfn}

\begin{prp}\label{prp_unramified}
	The restriction of $f^{\circ}$ to $\RT(\Fo_F)$ is $(\ud t)^{\frac{1}{2}}$ in the case that $\psi$ is unramified.
\end{prp}

\begin{proof}
	Note that we have
	\[
	\begin{pmatrix}
		1 & n_1 \\ 0 & 1
	\end{pmatrix}\begin{pmatrix}
		0 & t_2 \\ t_1 & 0 \\
	\end{pmatrix}\begin{pmatrix}
		1 & n_2 \\0 & 1
	\end{pmatrix}=\begin{pmatrix}
		t_1n_1 & t_1n_1n_2+t_2 \\ t_1 & t_1n_2
	\end{pmatrix}.
	\]
	So in the case that $t_1,t_2\in\Fo_F^{\times}$,
	\[
	\begin{pmatrix}
		1 & n_1 \\ 0 & 1
	\end{pmatrix}\begin{pmatrix}
		 0 & t_2 \\ t_1 & 0\\
	\end{pmatrix}\begin{pmatrix}
		1 & n_2 \\ 0 & 1
	\end{pmatrix}\in\Mat_2(\Fo_F)
	\]
	if and only if $n_1,n_2\in \Fo_F$. So we have $\CO^{\psi}_t(1_{\Mat_2(\Fo_F)})=1$ under the above conditions when $t\in \RT(\Fo_F)$. Note that $|\det t|=\delta^{\frac{1}{2}}(t)=1$ as well in such a case, so $f^{\circ}|_{\RT(\Fo_F)}=(\ud t)^{\frac{1}{2}}$.
\end{proof}

\begin{prp}\label{prp_irr1}
	For fixed $t_1\neq 0$, $\displaystyle{\frac{f(t_1,\cdot)}{|\det t|\delta^{\frac{1}{2}}(t)} (\ud t)^{\frac{1}{2}}}$ is of Schwartz type when $t_2\rightarrow\infty$, and can be extended across $0$ smoothly. The extended function on $F$ is a usual Schwartz function on $F$, where $t=\begin{pmatrix}
		t_1 & 0 \\ 0 & t_2
	\end{pmatrix}$. Moreover, when $F$ is non-Archimedean and $t_1\in\Fo_F^{\times}$,
	\[
	\frac{f^{\circ}(t_1,\cdot)}{|\det t|\delta^{\frac{1}{2}}(t)} (\ud t)^{\frac{1}{2}}=1_{\Fo_F}(\cdot).
	\]
\end{prp}

\begin{proof}
	Assume $f$ is the twisted push-forward of $\Phi(x)(\ud^+ x)^{\frac{1}{2}}$, where $\Phi\in\CF(\Mat_2)$, then 
    	\[
    	\frac{f}{|\det t|\delta^{\frac{1}{2}}(t)(\ud t)^{\frac{1}{2}}}(t_1,t_2)=\int_{F\times F}\Phi\left(\begin{pmatrix}
    		1 & n_1 \\ 0 & 1
    	\end{pmatrix} \begin{pmatrix}
    		 0 & t_2 \\ t_1 & 0
    	\end{pmatrix} \begin{pmatrix}
    		1 & n_2 \\ 0 & 1
    	\end{pmatrix} \right)\psi^{-1}(n_1+n_2)\ud n_1\ud n_2.
    	\]
    	
    	Note that
    	\[
	\begin{pmatrix}
		1 & n_1 \\ 0 & 1
	\end{pmatrix}\begin{pmatrix}
		0 & t_2 \\ t_1 & 0\\
	\end{pmatrix}\begin{pmatrix}
		1 & n_2 \\ 0  & 1
	\end{pmatrix}=\begin{pmatrix}
		t_1n_1 & t_1n_1n_2+t_2 \\ t_1 & t_1n_2
	\end{pmatrix},
	\]
	for fixed $t_1\in F^{\times}$, since the function $\Phi$ is of Schwartz type on $\Mat_2$, by looking at the $(1,2)$-entry and the $(2,1)$-entry, this integral can be bounded by an integrable function that is independent of $t_2$, then by the dominant convergence theorem, we have
	\begin{align*}
	\int_{F\times F}\Phi\left(\begin{pmatrix}
    		1 & n_1 \\  0 & 1
    	\end{pmatrix} \begin{pmatrix}
    		0 & t_2 \\ t_1 & 0
    	\end{pmatrix} \begin{pmatrix}
    		1 & n_2 \\ 0 & 1
    	\end{pmatrix} \right)\psi^{-1}(n_1+n_2)\ud n_1\ud n_2\\
    	\rightarrow \int_{F\times F}\Phi\left(\begin{pmatrix}
    		1 & n_1 \\  0 & 1
    	\end{pmatrix} \begin{pmatrix}
    		0 & 0  \\ t_1 & 0
    	\end{pmatrix} \begin{pmatrix}
    		1 & n_2 \\ 0 & 1
    	\end{pmatrix} \right)\psi^{-1}(n_1+n_2)\ud n_1\ud n_2,\; t_2\rightarrow0.
	\end{align*}
		
	When $t_2\rightarrow\infty$, when $F$ is non-Archimedean, when $t_1n_1$ and $t_2n_2$ are both in some compact set, $t_2+t_1n_1n_2\rightarrow\infty$ as $t_2\rightarrow\infty$ since $t_1\in F^{\times}$ is fixed, which means this integral is zero. The second claim is clear.
	
When $F$ is Archimedean, since for any differential operator with polynomial coefficient $\RD$, we have $\RD\Phi$ is still a Schwartz function, then by dominant convergence theorem again we see $\displaystyle{\frac{f(t_1,\cdot)}{|\det t|\delta^{\frac{1}{2}}(t)} (\ud t)^{\frac{1}{2}}}$ is a smooth function near $0$. When $t_2$ is near $\infty$, since
\[
\int_{ |n_1|^2+|n_2|^2\gg 1}\left|\Phi \left(\begin{pmatrix}
		n_1 & \frac{n_1n_2}{t_1}+t_2 \\ t_1 & n_2
	\end{pmatrix} \right)\right|\ud n_1\ud n_2\ll 1.
\]	
and
\[
\int_{ |n_1|^2+|n_2|^2\ll 1}\left|\Phi \left(\begin{pmatrix}
		n_1 & \frac{n_1n_2}{t_1}+t_2 \\ t_1 & n_2
	\end{pmatrix} \right)\right|\ud n_1\ud n_2\ll \left|\frac{n_1n_2}{t_1}+t_2\right|^{-N}\ll |1+t_2|^{-N},
\]
then 
\[
\left|\frac{f(t_1,\cdot)}{|\det t|\delta^{\frac{1}{2}}(t)} (\ud t)^{\frac{1}{2}}\right|\ll|1+t_2|^{-N}. 
\]
Since again the derivatives of $\Phi$ is still a Schwartz function, we see all the derivatives of $\displaystyle{\frac{f(t_1,\cdot)}{|\det t|\delta^{\frac{1}{2}}(t)} (\ud t)^{\frac{1}{2}}}$ are of rapid decay at $\infty$, i.e., $\displaystyle{\frac{f(t_1,\cdot)}{|\det t|\delta^{\frac{1}{2}}(t)} (\ud t)^{\frac{1}{2}}}$ is a Schwatz function.
\end{proof}

Motivated by the Hankel transform in Section \ref{ssec_hankel} calculated by \cite{Jac03} and the global application, let us introduce the following intermediate half-densities and normalized intermediate functions:
\begin{dfn}\label{dfn_int}
	Let $f\in\CCD^-_{L\left(\std,\frac{1}{2}\right)}(\RN,\psi \backslash \RG/\RN,\psi)$, we define its intermediate half-density as 
	\[\varphi_f:=\psi(-e^{\alpha_1})\circ\CF^{\psi}_{-\epsilon_2^{\vee},\frac{1}{2}}(f),\]
	where $\psi(-e^{\alpha_1})$ is the multiplication $\begin{pmatrix}
		t_1 & 0 \\  0 & t_2
	\end{pmatrix}\mapsto \psi\left(-\frac{t_2}{t_1}\right)$.
	We define the normalized intermediate function
	\[
	\widetilde{\varphi_f}(t):=\frac{\varphi_f(t_1,t_2)}{\left|\frac{t_1}{t_2}\right|^{\frac{1}{2}}(\ud t)^{\frac{1}{2}}},
	\]
	where $t=\begin{pmatrix}
		t_1 & 0 \\ 0 & t_2
	\end{pmatrix}$.
\end{dfn}

\begin{lem}\label{lem_mid=F}
	If we write $f=|\det t|\delta^{\frac{1}{2}}(t)\CO_t(\Phi)$, then
	\[
	\widetilde{\varphi_f}(t_1,t_2)=|t_1|\psi\left( -\frac{t_2}{t_1}  \right)\BF_{\psi,2}(\CO^{\psi}_{(t_1,\cdot)}(\Phi))\left(-\frac{1}{t_2}\right),
	\]
	where $\BF_{\psi,2}$ is the Fourier transform with respect to the second variable, which makes sense according to Proposition \ref{prp_irr1}.
\end{lem}

\begin{proof}
	Write $f=|t_1|^{\frac{3}{2}}|t_2|^{\frac{1}{2}}\CO^{\psi}_{\Phi}(t_1,t_2)(\ud t)^{\frac{1}{2}}$ with $\Phi\in\CF(\Mat_{2})$. Then we have
	\begin{align*}
		\CF^{\psi}_{-\epsilon_2^{\vee},\frac{1}{2}}(f)(t_1,t_2)&=\int_{F^{\times}}f(t_1,t_2x)\psi(x)|x|^{\frac{1}{2}}\ud^{\times}x\\
		&=\int_F |t_1|^{\frac{3}{2}}|t_2x|^{\frac{1}{2}}\CO^{\psi}_{\Phi}(t_1,t_2x)(\ud t)^{\frac{1}{2}}\psi(x)|x|^{\frac{1}{2}}\ud^{\times}x .
	\end{align*}
	By making a change of variable $ x\mapsto  t_2^{-1} x $, we get
	\begin{align*}
		\frac{\CF^{\psi}_{-\epsilon_2^{\vee},\frac{1}{2}}(f)(t_1,t_2)}{(\ud t)^{\frac{1}{2}}}=|t_1|^{\frac{3}{2}}|t_2|^{-\frac{1}{2}}\int_F\CO^{\psi}_{\Phi}(t_1,x)\psi(t_2^{-1}x)\ud x.
	\end{align*}
\end{proof}

\begin{prp}\label{prp_mid}
	For fixed $t_2\in F^{\times}$, $\widetilde{\varphi_f}(\cdot,t_2)$ is of Schwartz type near $\infty$, and
	\[
	\widetilde{\varphi_f}(0,t_2):=\lim_{t_1\rightarrow 0}\widetilde{\varphi_f}(t_1,t_2)	
	\]
	exists. So for fixed $t_2\in F^{\times}$, $\widetilde{\varphi_f}(\cdot,t_2)$ is a Schwartz function on $F$. Moreover, when $t_2\in\Fo_F^{\times}$ and $\psi$ is unramified, 
	\[
	\widetilde{\varphi_{f^{\circ}}}(\cdot,t_2)=1_{\Fo_F}(\cdot).
	\]
  \end{prp}

\begin{proof}
	Let us expand $\CO_{\Phi}^{\psi}$, then $\BF_{\psi,2}(\CO^{\psi}_{(t_1,\cdot)}(\Phi))\left(-\frac{1}{t_2}\right)$ is	\begin{align*}
		\int_F\int_{F\times F}\Phi\left( \begin{pmatrix}
			t_1n_1 & x+t_1n_1n_2\\ t_1 & t_1n_2
		\end{pmatrix}  \right)\psi^{-1}(n_1+n_2)\ud n_1\ud n_2 \cdot \psi(t_2^{-1}x)\ud x.
	\end{align*}
	Let $x+t_1n_1n_2\mapsto x $, $t_1n_1\mapsto n_1$ and $t_1n_2\mapsto n_2$, we have the above is
	\begin{align*}
		\int_{F^3}\Phi\left(    \begin{pmatrix}
			n_1 & x \\ t_1 & n_2
		\end{pmatrix}  \right)\psi^{-1}\left(\frac{n_1+n_2}{t_1}\right)\psi\left(\frac{xt_1-n_1n_2}{t_1t_2}\right)|t_1|^{-2}\ud n_1\ud n_2\ud x.
	\end{align*}
	Therefore we have
\begin{align}\label{eq_inter}
	\frac{\CF^{\psi}_{-\epsilon_2^{\vee},\frac{1}{2}}(f)(t_1,t_2)}{(\ud t)^{\frac{1}{2}}}&=|t_1t_2|^{-\frac{1}{2}}\cdot \int_{F^3} \Phi\left(  \begin{pmatrix}
		n_1 & x   \\ t_1 & n_2 
	\end{pmatrix}  \right)  \psi\left( \frac{x}{t_2}  \right)\psi\left(  - \frac{n_1+n_2}{t_1} -\frac{n_1n_2}{t_1 t_2}  \right) \ud x\ud n_1\ud n_2\\
	&=|t_1t_2|^{-\frac{1}{2}}\int_F \BF_{\psi,1}\BF_{\psi,2}\Phi\left( \begin{pmatrix} \frac{1}{t_1}+\frac{n_2}{t_1t_2} & -\frac{1}{t_2} \\ t_1 &  n_2 \end{pmatrix}   \right)\psi\left(-\frac{n_2}{t_1}\right) \ud n_2.
\end{align}
Make a change of variable that $n_2=t_1t_2z-t_2$, then we have
\begin{align*}
	\frac{\CF^{\psi}_{-\epsilon_2^{\vee},\frac{1}{2}}(f)(t_1,t_2)}{(\ud t)^{\frac{1}{2}}}=|t_1t_2|^{\frac{1}{2}}\int_F\BF_{\psi,1}\BF_{\psi,2}\Phi\left( \begin{pmatrix}
		z & -\frac{1}{t_2} \\ t_1 & t_1t_2z-t_2
	\end{pmatrix}   \right)\psi\left(\frac{t_2}{t_1}-t_2 z\right)\ud z.
\end{align*}
Therefore, 
\begin{align*}
	\widetilde{\varphi_f}(t_1,t_2)=|t_2|\int_F\BF_{\psi,1}\BF_{\psi,2}\Phi\left( \begin{pmatrix}
		z & -\frac{1}{t_2} \\ t_1 & t_1t_2z-t_2
	\end{pmatrix}   \right)\psi\left(-t_2 z\right)\ud z.
\end{align*}
Then $\widetilde{\varphi_f}(\cdot,t_2)$ is a Schwartz function of $t_1$ for fixed $t_2\in F^{\times}$.

When $F$ is non-Archimedean, $t_2\in\Fo^{\times}$ and $\psi$ is unramified, we can choose
\[
\Phi_{11}=\Phi_{12}=\Phi_{21}=\Phi_{22}=1_{\Fo_F},
\]
then we have
\[
\widetilde{\varphi_{f^{\circ}}}(t_1,t_2)=1_{\Mat_2(\Fo_F)}(t_1)\int_{F}1_{\Mat_2(\Fo_F)}(z)1_{\Mat_2(\Fo_F)}(t_1z)\psi^{-1}(z)\ud z=1_{\Fo_F}(t_1).
\]

\end{proof}

Let $\Mat_2^{\vee}$ be the dual space of $\Mat_2$, which can be identified with $\Mat_2$ via the trace pairing. But we embed $G\hookrightarrow \Mat_2^{\vee}$ via the inverse map $g\mapsto g^{-1}$, and the action is given by $X(g_1,g_2):=g_2^{-1}Xg_1$. Then $G\hookrightarrow\Mat_2^{\vee}$ is $G\times G$-equivariant. 

\begin{dfn}
	Let $\mathcal{D}(\Mat_2^{\vee})$ be the pull-back of Schwartz half-densities of $\Mat^{\vee}_2$ to $G$ via the above embedding. In this case, for $\Phi\in\CS(\Mat_2^{\vee})$, the pull-back of $\Phi(x)(\ud x)^{\frac{1}{2}}$ is $|\det g|^{-1}\Phi(g^{-1})(\ud g)^{\frac{1}{2}}$. We use $\CCD^-_{L\left(\std^{\vee},\frac{1}{2}\right)}(\RN,\psi\backslash \RG/\RN,\psi)$ to be the image of the twisted push-forward of $\CCD(\Mat_2^{\vee})$ to the open subset $\RT$ of $\FC$. More explicitly, the twisted push-forward of $|\det g|^{-1}\Phi(g^{-1})(\ud g)^{\frac{1}{2}}$ is
	 \[
|\det t|^{-1}\delta^{\frac{1}{2}}(t) \int_{\RN\times \RN}\Phi((n_1wtn_2)^{-1})\psi^{-1}(n_1n_2)\ud n_1\ud n_2\cdot (\ud t)^{\frac{1}{2}}=|\det t|^{-1}\delta^{\frac{1}{2}}(t)\CO_{wt^{-1}w}^{\psi^{-1}}(\Phi) (\ud t)^{\frac{1}{2}}.
\]
\end{dfn}

\begin{dfn}
		We write $f^{\circ,\vee}$ for the basic vector that is defined to be the twisted push-forward of $1_{\Mat_2^{\vee}(\Fo_F)}(\ud x)^{\frac{1}{2}} \in\CS(\Mat_2^{\vee})$.
\end{dfn}

Similar to the above case, we have
\begin{prp}
	The restriction of $f^{\circ,\vee}$ to $\RT(\Fo_F)$ is $(\ud t)^{\frac{1}{2}}$ in the case that $\psi$ is unramified.
\end{prp}

We also have the exact proofs of the parallel version of Proposition \ref{prp_irr1} and Proposition \ref{prp_mid}.

\begin{prp}
	For fixed $t_2\neq 0$, $\displaystyle{\frac{f((\cdot)^{-1},t_2)}{|\det t|^{-1}\delta^{\frac{1}{2}}(t)}(\ud t)^{\frac{1}{2}}}$ can be extended to a Schwartz function on $F$, where $t=\begin{pmatrix}
		t_1 & 0 \\ 0 &  t_2
	\end{pmatrix}$.
\end{prp}

\begin{dfn}\label{dfn_int_dual}
	Let $f\in\CCD^-_{L\left(\std^{\vee},\frac{1}{2}\right)}(\RN,\psi \backslash \RG/ \RN,\psi)$, we define its intermediate half-density as 
	\[^{\vee}\varphi_f:=\psi^{-1}(-e^{\alpha_1})\circ\CF_{\epsilon_1^{\vee},\frac{1}{2}}(f),\]
	where $\psi^{-1}(-e^{\alpha_1})$ is still the multiplication $\begin{pmatrix}
		t_1 & 0 \\ 0 & t_2
	\end{pmatrix}\mapsto \psi^{-1}\left(-\frac{t_2}{t_1}\right)$.
	We define the normalized intermediate function.
	\[
	\widetilde{^{\vee}\varphi_f}:=\frac{^{\vee}\varphi_f(t_1,t_2)}{\left|\frac{t_1}{t_2}\right|^{\frac{1}{2}}(\ud t)^{\frac{1}{2}}},
	\]
	where $t=\begin{pmatrix}
		t_1 & 0\\ 0 & t_2
	\end{pmatrix}$.
\end{dfn}
\begin{lem}\label{lem_irr2dual}
	If we write $f=|\det t|^{-1}\delta^{\frac{1}{2}}(t)\CO^{\psi^{-1}}_{wt^{-1}w}(\Phi)$, then
	\[
	\widetilde{^{\vee}\varphi_f}(t_1,t_2)=|t_2|^{-1}\psi^{-1}\left( -\frac{t_2}{t_1}  \right)\BF_{\psi^{-1},2}(\CO^{\psi^{-1}}_{(t_2^{-1},\cdot)}(\Phi))\left(-t_1\right),
	\]
	where $\BF_{\psi,2}$ is the Fourier transform with respect to the second variable, which makes sense according to Proposition \ref{prp_irr1}.
\end{lem}

\begin{proof}
    The proof is the same as Lemma \ref{lem_mid=F}.
\end{proof}

\begin{prp}\label{prp_middual}
	If we write $f=|\det t|^{-1}\delta^{\frac{1}{2}}(t)\CO_{wt^{-1}w}^{\psi^{-1}}(\Phi)$, then for fixed $t_1\in F^{\times}$, $\widetilde{^{\vee}\varphi_f}(t_1,(\cdot)^{-1})$ is of Schwartz type on $F$. In particular,
	\[
	\widetilde{^{\vee}\varphi_f}(t_1,\infty):=\lim_{t_2\rightarrow \infty}\widetilde{^{\vee}\varphi_f}(t_1,t_2)
	\]
	exists.
\end{prp}

\begin{proof}
    The proof is the same as Proposition \ref{prp_mid}.
\end{proof}

\subsection{Hankel Tranforms}\label{ssec_hankel}
    
Consider the classical Fourier transform $\CF(\Mat_2)\rightarrow\CF(\Mat_2^{\vee})$ on the space of matrices:
\[
\Phi(x) \mapsto\widehat{\Phi}(x),
\]
where 
\[
\widehat{\Phi}(x):=\int_{\Mat_2}\Phi(y)\psi(\tr (yx))\ud^+ y.
\]
Then $\Phi\mapsto\widehat{\Phi}$ is $\RG\times \RG$-equivariant with respect to the $\RG\times \RG$ actions on $\Mat_2$ and $\Mat_2^{\vee}$ that we defined in the previous section. This Fourier transform induces a Fourier transform $\CF$ on $\CCD(\Mat_2)\rightarrow\CCD(\Mat_2^{\vee}): \Phi(x)(\ud x)^{\frac{1}{2}}\mapsto \widehat{\Phi}(x)(\ud x)^{\frac{1}{2}}$. This Fourier transform descends to the level of orbital integrals, and we have a beautiful formula to describe it on the level of orbital integrals.

\begin{thm}[\cite{Jac03,Sak19}]\label{thm_hankel}
	We have a commutative diagram
	\[
		\begin{array}[c]{ccc}
\CCD(\Mat_2)&\stackrel{\widehat{(\cdot)}}{\rightarrow}&\CCD(\Mat_2^{\vee})\\
\downarrow &&\downarrow \\
\CCD^-_{L\left(\std,\frac{1}{2}\right)}(\RN,\psi\backslash \RG/\RN,\psi)&\stackrel{\CH_{\std}}{\rightarrow}&\CCD^-_{L\left(\std^{\vee},\frac{1}{2}\right)}(\RN,\psi\backslash \RG/\RN,\psi)
\end{array},
	\]
	and $\CH_{\std}$ is given by the following formula
	\[
	\CH_{\std}=\CF^{\psi}_{-\epsilon_1^{\vee},\frac{1}{2}}\circ\psi (-e^{-\alpha_1})\circ\CF^{\psi}_{-\epsilon_2^{\vee},\frac{1}{2}}.
	\]
 \end{thm}

\begin{rmk}
	The calculation of this descent of the Fourier transform for general $\GL_n$ first appeared in \cite[Theorem 1]{Jac03}, and was written is the above form in \cite[Theorem 8.1]{Sak19}. For the convenience of later reference and further comparison with the notations in \cite{Jac03}, let us copy the proof of \cite[Theorem 8.1]{Sak19} here.
\end{rmk}

\begin{proof}
	Let $\Phi\in\CF(\Mat_2)$, recall that \cite{Jac03} used the notation
	\[
	\Omega(\Phi,\psi^{-1},t):=\int_{\RN\times \RN}\Phi({^tn_1}tn_2)\psi^{-1}(n_1n_2)\ud n_1\ud n_2
	\]
	to denote the orbital integral. Note that $^tn_1=wtw$, so we have $\Omega(\Phi,\psi^{-1}:t)=\CO_{\RL(w)\Phi}(t)$. And in \cite{Jac03}, the normalization is
	\[
	\widetilde{\Omega}(\Phi,\psi^{-1}:t):=|t_1|\Omega(\Phi,\psi^{-1}:t),
	\]
	so if we write $f(t)$ for the twisted push-forward of $|\det g|\Phi(g)(\ud g)^{\frac{1}{2}}$, then we have $f(t)=|t_1t_2|^{\frac{1}{2}}\widetilde{\Omega}(\RL(w)\Phi,\psi^{-1}:t)(\ud t)^{\frac{1}{2}}$.
	
	Since the Fourier transform of $|\det g|\Phi(g)(\ud g)^{\frac{1}{2}}$ is $\widehat{\Phi}(x)(\ud x)^{\frac{1}{2}}$ on $\Mat_2^{\vee}$, whose pull-back to $G$ is $|\det g|^{-1}\widehat{\Phi}(g^{-1})(\ud g)^{\frac{1}{2}}$. Then, its twisted push-forward is
	\[
	|\det t|^{-1}\delta^{1/2}(t)\int_{\RN\times \RN}\widehat{\Phi}(n_1 w \cdot wt^{-1}w \cdot n_2)\psi(n_1n_2)\ud n_1\ud n_2 \cdot(\ud t)^{\frac{1}{2}},
	\]
	which is 
	\[
	|\det t|^{-1}\delta^{\frac{1}{2}}(t)\int_{\RN\times \RN}\widehat{\Phi}(w{^tn_1} wt^{-1}w n_2)\psi(n_1n_2)\ud n_1\ud n_2\cdot (\ud t)^{\frac{1}{2}},
	\]
	and this is $|t_1t_2|^{-\frac{1}{2}}\widetilde{\Omega}(\RL(w)\widehat{\Phi},\psi, wt^{-1}w)(\ud t)^{\frac{1}{2}}$. Note that we have $\RL(w)\widehat{\Phi}=(\RL(w)\Phi)^{\vee}$ using the notation of \cite{Jac03}. Then we have
	\begin{align*}
		&\CF^{\psi}_{-\epsilon_1^{\vee},\frac{1}{2}}\circ\psi (-e^{-\alpha_1})\circ\CF^{\psi}_{-\epsilon_2^{\vee},\frac{1}{2}}(f)(t_1,t_2)\\
		&=\CF^{\psi}_{-\epsilon_1^{\vee},\frac{1}{2}}\left( (b_1,b_2)\mapsto  \psi\left(-\frac{b_2}{b_1}\right)\CF^{\psi}_{-\epsilon_2^{\vee},\frac{1}{2}}(f)(b_1,b_2)         \right)\\
		&=\int_{F^{\times}} |p_1|^{\frac{1}{2}}\psi(p_1) \psi\left(-\frac{t_2}{p_1t_1} \right)\CF^{\psi}_{-\epsilon_2^{\vee},\frac{1}{2}}(f)(p_1t_1,t_2) \ud^{\times }p_1 \\
		&=\int_{F^{\times}} |p_1|^{\frac{1}{2}}\psi(p_1) \psi\left(-\frac{t_2}{p_1t_1} \right) \int_{F^{\times}}|p_2|^{\frac{1}{2}}\psi(p_2)f(p_1t_1,p_2t_2)\ud^{\times }p_2\ud^{\times}p_1\\
		&=\int_{F^{\times}}\int_{F^{\times}}|p_1p_2|^{\frac{1}{2}}\psi\left(p_1+p_2-\frac{t_2}{p_1t_1}\right)f(p_1t_1,p_2t_2)\ud^{\times}p_2\ud^{\times}p_1.
	\end{align*}
	Make use of the relation that $f(t)=|t_1t_2|^{\frac{1}{2}}\widetilde{\Omega}((w,1)\Phi,\psi^{-1}:t)(\ud t)^{\frac{1}{2}}$, we have the above is
	\begin{align*}
		&\int_F\int_F |p_1p_2|^{\frac{1}{2}}\psi\left(p_1+p_2-\frac{t_2}{p_1t_1}\right)|p_1t_1p_2t_2|^{\frac{1}{2}}\widetilde{\Omega}\left((w,1)\Phi,\psi^{-1}: \begin{pmatrix}
			p_1t_1 & 0 \\ 0 & p_2t_2
		\end{pmatrix}\right)\\
		&\cdot (\ud p_1t_1)^{\frac{1}{2}}(\ud p_2t_2)^{\frac{1}{2}}\ud^{\times}p_2\ud^{\times}p_1\\
		&=\int_F\int_F|t_1t_2|^{\frac{1}{2}} \psi\left(p_1+p_2-\frac{t_2}{p_1t_1}\right)  \widetilde{\Omega}\left((w,1)\Phi,\psi^{-1}:\begin{pmatrix}
			p_1t_1 & 0\\0 & p_2t_2
		\end{pmatrix}\right)\ud p_2\ud p_1\cdot (\ud t_1)^{\frac{1}{2}}(\ud t_2)^{\frac{1}{2}}\\
		&=|t_1t_2|^{-\frac{1}{2}}\int_F\int_F\psi\left(\frac{p_1}{t_1}+\frac{p_2}{t_2}-\frac{t_2}{p_1}\right)\widetilde{\Omega}\left((w,1)\Phi,\psi^{-1}:\begin{pmatrix}
			p_1 & 0 \\ 0 & p_2
		\end{pmatrix}\right)\ud p_2\ud p_1 (\ud t_1)^{\frac{1}{2}}(\ud t_2)^{\frac{1}{2}}\\
		&=|t_1t_2|^{-\frac{1}{2}}\widetilde{\Omega}\left((\RL(w)\Phi)^{\vee},\psi: \begin{pmatrix}
			t_2^{-1} & 0 \\0 & t_1^{-1}
		\end{pmatrix}\right)(\ud t)^{\frac{1}{2}}\\
		&=\CH_{\std}(f),
	\end{align*}
	where the last but two equality is \cite[Theorem 1]{Jac03}.
\end{proof}

\begin{rmk}\label{rmk_han}
	Using the notation in Proposition \ref{dfn_int}, we have
	\[
	\CH_{\std}(f)=\CF^{\psi}_{-\epsilon_1^{\vee},\frac{1}{2}}(\varphi_f).
	\]
	More explicitly, we have
	\begin{align*}
	\frac{\CH_{\std}(f)}{(\ud t)^{\frac{1}{2}}}(t_1,t_2)&=|t_1t_2|^{-\frac{1}{2}} \int_{F}\widetilde{\varphi_{\nu}}(x,t_2)\psi\left(\frac{x}{t_1}\right)\ud x\\
	&=|t_1t_2|^{-\frac{1}{2}}\BF_{\psi,1}(\widetilde{\varphi_f}(\cdot,t_2))\left(-\frac{1}{t_1}\right).
	\end{align*}

 Also note that $\CH_{\std}^{-1}=\CF_{\epsilon_2^{\vee},\frac{1}{2}}^{\psi^{-1}}\circ \psi^{-1}(-e^{-\alpha_1})\circ\CF_{\epsilon_1^{\vee},\frac{1}{2}}^{\psi^{-1}}$, then for $f\in\CCD^-_{L\left(\std^{\vee},\frac{1}{2}\right)}(\RN,\psi\backslash \RG/\RN,\psi)$, we have
\[
	\CH_{\std}^{-1}(f)=\CF_{\epsilon_2^{\vee},\frac{1}{2}}^{\psi^{-1}}({^{\vee}\varphi_f}).
	\]
	and
	\begin{align*}
	\frac{\CH_{\std}(f)}{(\ud t)^{\frac{1}{2}}}(t_1,t_2)&=|t_1t_2|^{\frac{1}{2}} \int_{F}\widetilde{^{\vee}\varphi_{f}}(t_1,x^{-1})\psi^{-1}\left(xt_2\right)\ud x\\
	&=|t_1t_2|^{\frac{1}{2}}\BF_{\psi^{-1},2}(\widetilde{^{\vee}\varphi_f}(t_1,(\cdot)^{-1}))\left(-t_2\right).
	\end{align*}
\end{rmk}
    
    Let $\CH:=\CH(\RG,\RG(\Fo))$ be the spherical Hecke algebra, which acts on $\RG(\Fo)$-invariant functions on $\RG$ by convolution:
    \[
    h\star f(g):=\int_{\RG}h(xg^{-1})f(g)\ud g,\;h\in\CH,g\in\RG.
    \]
    Then to isolate global spectrum in Section \ref{sec_iso}, we need the following fundamental lemma for spherical Hecke algebra:

    \begin{prp}\label{prp_hankel basic}
    Let $\Phi$ be a Schwartz function on $\Mat_2$, then for $h\in\CH$, we have
    have
    	\[\widehat{h\star \Phi}= (h(\cdot)|\det\cdot |^2)\star \widehat{\Phi} \]
    	as functions on $\RG$.
    	In particular, if $\psi$ is unramified, we have
    	\[
    	\widehat{h\star 1_{\Mat_2(\Fo)}}= (h(\cdot)|\det\cdot |^2)\star 1_{\Mat_2(\Fo)}.
    	\]
    \end{prp}
    
    \begin{proof}
    	\begin{align*}
    		\widehat{h\star\Phi }(x)&=\int_{\Mat_2}h\star \Phi(y)\psi(\tr(yx^{-1}))\ud^+y\\
    		&=\int_{\Mat_2} \int_{\RG} h(yg^{-1})\Phi(g)\ud g\psi(\tr(yx^{-1}))\ud^+y\\
    		&=\int_{\Mat_2}\int_{\RG} h(g^{-1})\Phi(gy)\ud g\psi(\tr(yx^{-1}))\ud^+y\\
    		&=\int_{\RG} h(g^{-1})\int_{\Mat_2}\Phi(y)\psi(\tr(yx^{-1}g^{-1}))|\det g|^{-2}\ud^+y\\
    		&=\int_{\RG}h(g^{-1})\widehat{\Phi}(gx)|\det g|^{-2}\ud g\\
    		&=\int_{\RG}h(xg^{-1})|\det xg^{-1} |^2   \widehat{\Phi}(g)\ud g.
    	\end{align*}
    	In particular, if $\psi$ is unramified, we have $\displaystyle{\widehat{1_{\Mat_2}(\Fo)}=1_{\Mat_2}(\Fo)}$, then 
    	\[
    	\widehat{h\star 1_{\Mat_2(\Fo)}}= (h(\cdot)|\det\cdot |^2)\star 1_{\Mat_2(\Fo)}.
    	\]

    \end{proof}
    
    \subsection{Irregular Distributions}
   
   In this section, we will consider two types of irregular distributions on the space of orbital integrals, which will contribute to the boundary terms in the global Poisson summation formula. 
   
 Due to Proposition \ref{prp_irr1}, we have the following definition 
    \begin{dfn}
    	We define the first type irregular distribution on $\CCD^-_{L\left(\std,\frac{1}{2}\right)}(\RN,\psi\backslash \RG/\RN,\psi)$ as the following:
    	\[
    	\CO_{u,1}(f)(t_1):=\lim_{t_2\rightarrow 0}\frac{f}{|\det t|\delta^{\frac{1}{2}}(t)(\ud t)^{\frac{1}{2}}}(t_1,t_2),\;f\in\CCD^-_{L\left(\std,\frac{1}{2}\right)}(\RN,\psi\backslash \RG/\RN,\psi),\;t_1\in F^{\times},
    	\]
    	where $t=\begin{pmatrix}
    		t_1 & 0 \\ 0 & t_2
    	\end{pmatrix}$.
    \end{dfn}
    
    We have a parallel version for the dual side $\CCD^-_{L\left(\std^{\vee},\frac{1}{2}\right)}(\RN,\psi\backslash \RG/\RN,\psi)$:
 
    \begin{dfn}
    	We define the first type irregular distribution on $\CCD^-_{L\left(\std^{\vee},\frac{1}{2}\right)}(\RN,\psi\backslash \RG/\RN,\psi)$ as the following:
    	\[
    	\CO_{u,1}^{\vee}(f)(t_2):=\lim_{t_1\rightarrow \infty}\frac{f}{|\det t|^{-1}\delta^{\frac{1}{2}}(t)(\ud t)^{\frac{1}{2}}}(t_1,t_2),\;f\in\CCD^-_{L\left(\std^{\vee},\frac{1}{2}\right)}(\RN,\psi\backslash \RG/\RN,\psi),\;t_2\in F^{\times},
    	\]
    	where $t=\begin{pmatrix}
    		t_1  & 0 \\0 & t_2
    	\end{pmatrix}$.
    \end{dfn}
    
We need the following easy lemma in Section \ref{ssec_dirtpf}:
    \begin{lem}\label{lem_irr1+han}
    	If the half-density is of the form $\CH_{\std}(f)$ for some $f\in\CCD_{L\left(\std,\frac{1}{2}\right)}(\RN,\psi\backslash \RG/\RN,\psi)$, then we have
    	\[
    	\CO_{u,1}^{\vee}(\CH_{\std}(f))(t_2)=|t_2|\BF_{\psi,1}(\widetilde{\varphi_f}(\cdot,t_2))(0).
    	\]
    \end{lem}
    
    \begin{proof}
    	This follows from Remark \ref{rmk_han} directly.
    \end{proof}

Motived by Theorem \ref{thm_hankel} and Proposition \ref{prp_mid}, we have the following second type irregular distributions
\begin{dfn}\label{dfn_irr2}
	We define the second type irregular distribution on $\CCD^-_{L\left(\std,\frac{1}{2}\right)}(\RN,\psi\backslash \RG/\RN,\psi)$ in the following way:
    	\[
    	\CO_{u,2}(f)(t_2):=\widetilde{\varphi_f}(0,t_2), t_2\in F^{\times}.
    	\]
\end{dfn}

To relate it to orbital integrals of the upstairs function, it turns out that it comes from the \emph{central orbital integral}. Let $\Phi(x)$ be a Schwartz function on $\Mat_2$, and $z\in F^{\times}$, consider the orbital integral
\[
\phi_{\Phi}(z):=\int_{\RN}\Phi(z\RI_2\cdot n)\psi^{-1}(n)\ud n,
\]
where the notation indicates that it is related to the function $\phi$ used in \cite{Jac03}. We have
\begin{prp}\label{prp_irr_2_loc-glo}
	\[
	\phi_{\Phi}(z)=|z|^{-2}\widetilde{\varphi_f}(0,-z),
	\]
	where $f$ is the twisted push forward of $\Phi(x)(\ud x)^{\frac{1}{2}}$ and $\widetilde{\varphi_f}$ is the normalized intermediate function of $f$ defined in Definition \ref{dfn_int}.
	\end{prp}
\begin{proof}
	Using the notations in \cite{Jac03}, $\phi_{\Phi}(z)=\Omega((w,1)\Phi,\psi^{-1}:wz)$. According to \cite[Proposition 4]{Jac03}, $\phi_{\Phi}(z)$ is a smooth function with compact support on $F^{\times}$, and is given by
\[
\phi_{\Phi}(z)=|z|^{-3}\int_F \Omega\left(   ((w,1)\Phi)^{\vee},\psi:\begin{pmatrix}
	-z^{-1} & 0 \\0 & t_1
\end{pmatrix}    \right)\ud t_1.
\]
Note that $\displaystyle{\Omega\left(   ((w,1)\Phi)^{\vee},\psi:\begin{pmatrix}
	-z^{-1} & 0 \\ 0 & t_1
\end{pmatrix}    \right)= \frac{{\CH_{\std}(f)}}{(\ud t)^{\frac{1}{2}}}(t_1^{-1},-z)|t_1|^{-\frac{1}{2}}|z|^{\frac{3}{2}}}$, so the above integral is
\[
\int_F\frac{\CH_{\std}(f)}{(\ud t)^{\frac{1}{2}}}(t_1^{-1},-z)|t_1|^{-\frac{1}{2}}|z|^{-\frac{3}{2}}\ud t_1=\int_F \frac{\CH_{\std}(f)}{(\ud t)^{\frac{1}{2}}}(t_1,-z) |bz|^{-\frac{3}{2}}\ud t_1.
\]
According to Remark \ref{rmk_han}, we have
\begin{align*}
	\CH_{\std}(f)(t_1,- z)&=\CF_{\epsilon_2^{\vee},\frac{1}{2}}^{\psi^{-1}}(\varphi_f)(t_1,-z)\\
		&=\int_F \widetilde{\varphi_f}(x,-z)\psi(t_1^{-1} x)\ud x \cdot |t_1|^{-\frac{1}{2}}|z|^{-\frac{1}{2}}(\ud t_1)^{\frac{1}{2}}(\ud z)^{\frac{1}{2}}\\
	&=\BF_{\psi,1}(\widetilde{\varphi_f})(-t_1^{-1},-z)|t_1|^{-\frac{1}{2}}|z|^{-\frac{1}{2}}(\ud t_1)^{\frac{1}{2}}(\ud z)^{\frac{1}{2}}.
\end{align*}
Therefore
\begin{align*}
	&\int_F \frac{\CH_{\std}(f)}{(\ud t)^{\frac{1}{2}}}(t_1,-z)|t_1z|^{-\frac{3}{2}}\ud t_1\\
	&=\int_F \BF_{\psi,1}(\widetilde{\varphi_f})(-t_1^{-1},-z)|t_1|^{-2}|z|^{-2}\ud t_1\\
	&=|z|^{-2}\widetilde{\varphi_f}(0,-z)
\end{align*}
according to the Fourier inversion formula since $\widetilde{\varphi_f}(\cdot,z)$ is a Schwartz function on $F$ for fixed $z\in F^{\times}$ due to Proposition \ref{prp_mid}. 

\end{proof}

We also have the counterpart for the dual space $\CCD^-_{L\left(\std^{\vee,}\frac{1}{2}\right)}(\RN,\psi\backslash \RG/\RN,\psi)$. 
\begin{dfn}
	We define the second type irregular distribution on $\CCD^-_{L\left(\std^{\vee},\frac{1}{2}\right)}(\RN,\psi\backslash \RG/\RN,\psi)$ in the following way:
    	\[
    	\CO_{u,2}^{\vee}(f)(t_1):=\widetilde{^{\vee}\varphi_f}(t_1,\infty),\;f\in\CCD^-_{L\left(\std^{\vee},\frac{1}{2}\right)}(\RN,\psi\backslash \RG/\RN,\psi),t_1\in F^{\times}.
    	\]
\end{dfn}

Let $\Phi\in\CS(\Mat_2^{\vee})$, consider the corresponding integral
\[
^{\vee}\phi_{\Phi}(z):=\int_{\RN}\Phi(z^{-1}\RI_2\cdot n)\psi(n)\ud n,
\]
we have
\begin{prp}
	\[
	^{\vee}\phi_{\Phi}(z)=|z|^2\cdot \widetilde{^{\vee}\varphi_f}(-z^{-1},\infty),
	\]
	where $f$ is the twisted push forward of $\Phi(x)(\ud ^+x)^{\frac{1}{2}}$, and $\widetilde{^{\vee}\varphi_f}$ is the normalized intermediate function of $f$ defined in Definition \ref{dfn_int_dual}.
\end{prp}

\begin{proof}
	Using the notations in \cite{Jac03}, $\phi^{\vee}_{\Phi}(z)=\Omega((w,1)\Phi,\psi:wz)$. According to \cite[Proposition 4]{Jac03} again, $\phi^{\vee}_{\Phi}(z)$ is a smooth function with compact support on $F^{\times}$, and is given by
\[
\phi_{\Phi}^{\vee}(z)=|z|^{3}\int_F \Omega\left(   ((w,1)\Phi)^{\vee},\psi^{-1}:\begin{pmatrix}
	-z & \\ & t_2
\end{pmatrix}    \right)\ud t_2.
\]
Note that $\displaystyle{\Omega\left(   ((w,1)\Phi)^{\vee},\psi^{-1}:\begin{pmatrix}
	-z & 0\\  0& t_2
\end{pmatrix}    \right)= \frac{{\CH_{\std}^{-1}(f)}}{(\ud t)^{\frac{1}{2}}}(-z, t_2)|t_2|^{-\frac{1}{2}}|z|^{\frac{3}{2}}}$, so the above integral is
\begin{align}\label{eq_irr2}
\int_F\frac{\CH_{\std}^{-1}(f)}{(\ud t)^{\frac{1}{2}}}(-z, t_2)|t_2|^{-\frac{1}{2}}|z|^{\frac{3}{2}}\ud t_2=\int_F \frac{\CH_{\std}^{-1}(f)}{(\ud t)^{\frac{1}{2}}}(-z, t_2^{-1}) |t_2z^{-1}|^{-\frac{3}{2}}\ud t_2.
\end{align}
According to Remark \ref{rmk_han}, we have
\begin{align*}
	\CH_{\std}^{-1}(f)(-z, t_2^{-1})&=\CF_{\epsilon_2^{\vee},\frac{1}{2}}(^{\vee}\varphi_f)(-z, t_2^{-1})\\
		&=\int_F \widetilde{^{\vee}\varphi_f}(-z,x^{-1})\psi^{-1}(t_2^{-1}x)\ud x\cdot |t_2|^{-\frac{1}{2}}|z|^{\frac{1}{2}}(\ud t_2)^{\frac{1}{2}}(\ud z)^{\frac{1}{2}}\\
	&=\BF_{\psi^{-1},2}(^{\vee}\widetilde{\varphi_f}(-z,(\cdot)^{-1}))(-t_2^{-1})|t_2|^{-\frac{1}{2}}|z|^{\frac{1}{2}}(\ud t_2)^{\frac{1}{2}}(\ud z)^{\frac{1}{2}}.
\end{align*}
Therefore
\begin{align*}
	&\int_F \frac{\CH_{\std}^{-1}(f)}{(\ud t)^{\frac{1}{2}}}(-z, t_2^{-1})|t_2z^{-1}|^{-\frac{3}{2}}\ud t_2\\
	&=\int_F\BF_{\psi^{-1},2}(^{\vee}\widetilde{\varphi_f}(-z,(\cdot)^{-1}))(-t_2^{-1})|t_2|^{-2}|z|^{2}\ud t_2\\
	&=|z|^{2}  \widetilde{^{\vee}\varphi_f}(-z,\infty)
\end{align*}
according to the Fourier inversion formula again.
\end{proof}

Let us conclude this section with another lemma that will be used in Section \ref{ssec_dirtpf}:
\begin{lem}\label{lem_irr2+han}
		Let $f\in\CCD^-_{L\left(\std,\frac{1}{2}\right)}(\RN,\psi\backslash \RG/\RN,\psi)$, then we have
		\[
	\CO_{u,2}^{\vee}(\CH_{\std}(f))(z)=|z|^{3}\BF_{\psi,2}(\CO_{(-z,\cdot)}(\Phi))(0)
	\]
	if $f$ is the twisted push-forward of $\Phi(x)(\ud^+ x)^{\frac{1}{2}}$.
\end{lem}

\begin{proof}
	According to (\ref{eq_irr2}), in the case that $f$ is the twisted push-forward of $\Phi(x)(\ud^+ x)^{\frac{1}{2}}$, we have
	\begin{align*}
		\CO_{u,2}^{\vee}(\CH_{\std}(f))(z)=\int_F |-z|^{\frac{3}{2}}|t_2|^{\frac{1}{2}}\CO_{\Phi}(-z,t_2) |t_2|^{-\frac{1}{2}}|z|^{\frac{3}{2}}\ud t_2=|z|^{3}\BF_{\psi,2}(\CO_{(-z,\cdot)}(\Phi))(0).
	\end{align*}
	\end{proof}

    \section{Global Theory}\label{sec_global}  
    
\subsection{Orbital Integrals, Hankel Transforms, and the Poisson Summation Formula}
    
    In the following sections, let $k$ be a global field, and for each place $\nu$ of $k$, we set $k_{\nu}$ to be the completion of $k$ at $\nu$. Write $\Fo_{\nu}$ and $\varpi_{\nu}$ for the ring of integers and a local uniformizer if $\nu$ is a finite place. Let $\BA$ be the ring of adeles of $k$. We fix an adelic character $\psi=\otimes^{\prime}_{\nu}\psi_{\nu}:\BA/k\rightarrow \BC$. To emphasize the place $\nu$, we will use $\CCD(\Mat_{2,\nu})$, $\CCD(\Mat_{2,\nu}^{\vee})$, $\CCD^-_{L\left(\std,\frac{1}{2}\right)}(\RN_{\nu},\psi_{\nu}\backslash \RG_{\nu}/\RN_{\nu},\psi_{\nu})$, $\CCD^-_{L\left(\std^{\vee},\frac{1}{2}\right)}(\RN_{\nu},\psi_{\nu}\backslash \RG_{\nu}/\RN_{\nu},\psi_{\nu})$ to denote the local spaces that we defined in Section \ref{sec_local spaces} when $F=k_{\nu}$. We will use $f^{\circ}_{\nu}$ and $f^{\circ,\vee}_{\nu}$ to denote the basic vectors, and $\CH_{\std,\nu}$ to be the Hankel transform. We also use $\CO_{u,1,\nu}$, $\CO_{u,1,\nu}^{\vee}$, $\CO_{u,2,\nu}$ and $\CO_{u,2,\nu}^{\vee}$ to denote the irregular distributions. We will also use $\BF_{\psi_{\nu}},\BF_{\psi_{\nu},i}$ to denote the local Fourier transform with respect to $\psi_{\nu}$, the partial Fourier transforms with the $i$-th variable, and will use $\BF_{\psi}=\otimes^{\prime}_{\nu}\BF_{\psi_{\nu}},\BF_{\psi,i,\BA}=\otimes^{\prime}_{\nu}\BF_{\psi_{\nu},i}$ to denote their global versions.
   
    \begin{dfn}
    	We define the global space of orbital integrals 
    	\[
    	\CCD^-_{L\left(\std,\frac{1}{2}\right)}(\RN(\BA),\psi\backslash \RG(\BA)/\RN(\BA),\psi):=\otimes^{\prime}_{\nu}\CCD^-_{L\left(\std,\frac{1}{2}\right)}(\RN_{\nu},\psi_{\nu}\backslash \RG_{\nu}/\RN_{\nu},\psi_{\nu}),
    	\]
    	where the restricted tensor product is taken with respect to the basic vectors $f^{\circ}_{\nu}$. We define $\CCD^-_{L\left(\std^{\vee},\frac{1}{2}\right)}(\RN(\BA),\psi\backslash \RG(\BA)/\RN(\BA),\psi)$ in a similar way.
    \end{dfn}
    
    \begin{rmk}
    	Thanks to Proposition \ref{prp_unramified}, the restriction of $f_{\nu}^{\circ}$ to $\RT(\Fo_{\nu})$ is $(\ud t_{\nu})^{\frac{1}{2}}$ for almost all $\nu$, then we can identify $\CCD^-_{L\left(\std,\frac{1}{2}\right)}(\RN(\BA),\psi\backslash \RG(\BA)/\RN(\BA),\psi)$ as half-densities on $\RT(\BA)$.
    \end{rmk}
    
    According to Proposition \ref{prp_hankel basic}, we can define a global Hankel transform
    \[
\CH_{\std,\BA}:=\otimes^{\prime}\CH_{\std,\nu}:\CCD^-_{L\left(\std,\frac{1}{2}\right)}(\RN(\BA),\psi\backslash \RG(\BA)/\RN(\BA),\psi)\rightarrow\CCD^-_{L\left(\std^{\vee},\frac{1}{2}\right)}(\RN(\BA),\psi\backslash \RG(\BA)/\RN(\BA),\psi).
    \]
    
    Now, we state the global Poisson summation formula on the level of orbital integrals. Let $f\in \CCD^-_{L\left(\std,\frac{1}{2}\right)}(\RN(\BA),\psi\backslash \RG(\BA)/\RN(\BA),\psi)$, we first define
    \begin{dfn}\label{dfn_KTF}
	\[
	\mathrm{KTF}_{L\left(\std,\frac{1}{2}\right)}(f):=\sum_{t\in \RT(k)}\frac{f}{(\ud t)^{\frac{1}{2}}}(t)+\sum_{t_1\in k^{\times}}\CO_{u,1,\BA}(f)(t_1)+\sum_{t_2\in k^{\times}}\CO_{u,2,\BA}(f)(t_2),
	\]
	where $\CO_{u,1,\BA}:=\otimes^{\prime}_{\nu}\CO_{u,1,\nu}$ and $\CO_{u,2,\BA}:=\otimes^{\prime}_{\nu}\CO_{u,2,\nu}$.
\end{dfn}
Similarly for $f\in\CCD^-_{L\left(\std^{\vee},\frac{1}{2}\right)}(\RN(\BA),\psi\backslash \RG(\BA)/\RN(\BA),\psi)$, we define
\[
\mathrm{KTF}_{L\left(\std^{\vee},\frac{1}{2}\right)}(f):=\sum_{t\in \RT(k)}\frac{f}{(\ud t)^{\frac{1}{2}}}(t)+\sum_{t_2\in k^{\times}}\CO_{u,1,\BA}^{\vee}(f)(t_2)+\sum_{t_1\in k^{\times}}\CO_{u,2,\BA}^{\vee}(f)(t_1),
	\]
	where $\CO_{u,1,\BA}^{\vee}:=\otimes^{\prime}_{\nu}\CO^{\vee}_{u,1,\nu}$ and $\CO_{u,2,\BA}:=\otimes^{\prime}_{\nu}\CO^{\vee}_{u,2,\nu}$.

We need the following lemmas to guarantee the absolute convergence of the sums in the above Definition \ref{dfn_KTF}.

\begin{lem}\label{lem_es} \hfill\break
	
	For $f\in\CCD^-_{L\left(\std,\frac{1}{2}\right)}(\RN(\BA),\psi\backslash \RG(\BA)/\RN(\BA),\psi)$ (resp. $\CCD^-_{L\left(\std^{\vee},\frac{1}{2}\right)}(\RN(\BA),\psi\backslash \RG(\BA)/\RN(\BA),\psi)$), we have
	\begin{itemize}
		\item [(1)] $\displaystyle{\sum_{t\in \RT(k)}\frac{f}{(\ud t)^{\frac{1}{2}}}(t)}$ is absolutely convergent.
		\item [(2)] $\displaystyle{\sum_{t_1\in k^{\times}}\CO_{u,1,\BA}(f)(t_1)}$ (resp. $\displaystyle{\sum_{t_2\in k^{\times}}\CO^{\vee}_{u,1,\BA}(f)(t_2)}$) is absolutely convergent.
		\item [(3)] $\displaystyle{\sum_{t_2\in k^{\times}}\CO_{u,2,\BA}(f)(t_2)}$ (resp. $\displaystyle{\sum_{t_1\in k^{\times}}\CO^{\vee}_{u,2,\BA}(f)(t_1)}$) is absolutely convergent.
	\end{itemize}
\end{lem}

\begin{proof}
	We only prove the statements for $f\in \CCD^-_{L\left(\std,\frac{1}{2}\right)}(\RN(\BA),\psi\backslash \RG(\BA)/\RN(\BA),\psi)$.
	Let $\Phi$ be a Schwartz function on $\Mat_2(\BA)$ such that $f$ is the twisted push-forward of $\displaystyle{\Phi(x)(\ud x)^{\frac{1}{2}}}$. Let us consider $(1)$ and $(2)$ together.
	
	Note that we have
	\begin{align*}
		&\sum_{t\in \RT(k)} \left|\frac{f}{(\ud t)^{\frac{1}{2}}}(t)\right|+\sum_{t_1\in k^{\times}}\CO_{u,1,\BA}(f)(t_1)\\
		&=\sum_{t_1\in k^{\times},t_2\in k}\left|   \int_{\RN(\BA)\times\RN(\BA)}\Phi\left(n_1 w \begin{pmatrix}
			t_1 & 0 \\  0 & t_2
		\end{pmatrix}  n_2\right)\psi^{-1}(n_1n_2)\ud n_1\ud n_2      \right|\\
		&\leq \sum_{t_1\in k^{\times},t_2\in k} \int_{\RN(\BA)\times\RN(\BA)}\left|\Phi\left(    \begin{pmatrix}
			t_1n_1 & t_1n_1n_2+t_2\\ t_1 & t_1n_2
		\end{pmatrix}  \right)\right|\ud n_1\ud n_2.
		\end{align*}
		We may assume $\Phi$ is positive and applying the Poisson summation formula, we have
		\begin{align*}
			\sum_{t_2\in k}\Phi\left(    \begin{pmatrix}
			t_1n_1 & t_1n_1n_2+t_2\\ t_1 & t_1n_2
		\end{pmatrix}  \right) =\sum_{t_2\in k }\BF_{\psi,2}\Phi\left(    \begin{pmatrix}
			t_1n_1 & t_2 \\ t_1 & t_1n_2
		\end{pmatrix}  \right)\psi(t_1t_2n_1n_2)	,	\end{align*}
		which is bounded by
		\[
		\sum_{t_2\in k }\left| \BF_{\psi,2}\Phi\left(    \begin{pmatrix}
			t_1n_1 & t_2 \\ t_1 & t_1n_2
		\end{pmatrix}  \right)\right|.
		\]
		Therefore, $(1)$ and $(2)$ are proved.

As for $(3)$, according to Proposition \ref{prp_irr_2_loc-glo}, we have
\begin{align*}
	\sum_{t_2\in k^{\times}}|\CO_{u,2,\BA}(f)(t_2)|&=\sum_{z\in k^{\times}}\left|\int_{\RN(\BA)}\Phi(z\RI_2\cdot n)\psi^{-1}(n)\ud n\right|\\
	&\leq \sum_{z\in k^{\times}}\int_{\RN(\BA)}\left| \Phi\left(\begin{pmatrix}
		z & zn \\ & z
	\end{pmatrix} \right)  \right|\ud n,
\end{align*}
which converges absolutely.

\end{proof}

Then we have the following Poisson summation formula
\begin{thm}\label{thm_psf}
	For $f\in\CCD_{L\left(\std,\frac{1}{2}\right)}(\RN(\BA),\psi\backslash \RG(\BA)/\RN(\BA),\psi)$, we have
	\[
	\mathrm{KTF}_{L\left(\std,\frac{1}{2}\right)}(f)=\mathrm{KTF}_{L\left(\std^{\vee},\frac{1}{2}\right)}(\CH_{\std,\BA}(f)). 	
	\]
\end{thm}

\subsection{Poisson Summation Formula: An Indirect Proof}\label{ssec_inproof}
In this section, we give an indirect proof of the Poisson summation formula, which is based on the classical Poisson summation formula on the upstairs spaces $\Mat(\BA)$ and $\Mat^{\vee}(\BA)$. The more important thing in the view of this paper is the direct proof from the next section, and this section can be viewed as a double-check that Theorem \ref{thm_psf} coincides with the classical one. The reader may also skip this section and move to the next section.

\begin{proof}[proof of Theorem \ref{thm_psf}]
	By linearity, we may assume $f=\otimes^{\prime}_{\nu} f_{\nu}$, with each $f_{\nu}$ being the twisted push-forward of some $\Phi_{\nu}(x)\in\CS(\Mat_{2,\nu})$. Let $\Phi=\otimes^{\prime}_{\nu}\Phi_{\nu}\in\CS(\Mat_2(\BA))$. Consider the integral 
	\[
\int_{k\backslash\BA}\int_{k\backslash\BA}\sum_{\gamma\in \Mat_{2}(k)}\Phi\left(  \begin{pmatrix}
	1 & n_1 \\ 0& 1
\end{pmatrix} \gamma \begin{pmatrix}
	1 & n_2 \\0 & 1
\end{pmatrix} \right)\psi^{-1}(n_1+n_2)\ud n_1\ud n_2.
\]

A first observation is that we have the double cosets decomposition
\[
\Mat_{2}(k)=\RB(k)w\RB(k)\bigsqcup \RB(k)\bigsqcup \{g\in \Mat_2(k)\mid  g\neq0,\det g=0\}\bigsqcup\{0\}.
\]
For the open Brahut cell, we have
\begin{align*}
	&\int_{k\backslash\BA}\int_{k\backslash\BA}\sum_{\gamma\in \RB(k)w\RB(k)}\Phi\left(  \begin{pmatrix}
	1 & n_1 \\ 0& 1
\end{pmatrix} \gamma \begin{pmatrix}
	1 & n_2 \\0 & 1
\end{pmatrix} \right)\psi^{-1}(n_1+n_2)\ud n_1\ud n_2\\
&=\int_{k\backslash\BA}\int_{k\backslash\BA}\sum_{\gamma \in \RT(k)}  \sum_{\delta_1\in \RN(k)}\sum_{\delta_2\in \RN(k)}  \Phi\left(  \begin{pmatrix}
	1 & n_1 \\0 & 1
\end{pmatrix}\delta_1 w \gamma \delta_2\begin{pmatrix}
	1 & n_2 \\ 0& 1
\end{pmatrix} \right)\psi^{-1}(n_1+n_2)\ud n_1\ud n_2\\
&=\int_{\BA}\int_{\BA} \sum_{\gamma\in \RT(k)}\Phi\left(  \begin{pmatrix}
	1 & n_1 \\0 & 1
\end{pmatrix} w\gamma \begin{pmatrix}
	1 & n_2 \\0 & 1
\end{pmatrix} \right)\psi^{-1}(n_1+n_2)\ud n_1\ud n_2\\
&=\sum_{\gamma\in \RT(k)}\prod_{\nu}\int_{k_{\nu}}\int_{k_{\nu}} \Phi_{\nu}\left(  \begin{pmatrix}
	1 & n_{1} \\ 0 & 1
\end{pmatrix} w\gamma \begin{pmatrix}
	1 & n_{2} \\  0& 1
\end{pmatrix} \right)\psi_{\nu}^{-1}(n_1+n_2)\ud n_1\ud n_2   \\
&=\sum_{\gamma\in \RT(k)} \frac{f}{(\ud t)^{\frac{1}{2}}}(\gamma).
\end{align*}
 
For the sum over $\RB(k)$, we have
\begin{align*}
	&\int_{k\backslash\BA}\int_{k\backslash\BA}\sum_{\gamma\in \RB(k)}\Phi\left(  \begin{pmatrix}
	1 & n_1 \\0 & 1
\end{pmatrix} \gamma \begin{pmatrix}
	1 & n_2 \\0 & 1
\end{pmatrix} \right)\psi^{-1}(n_1+n_2)\ud n_1\ud n_2\\
&\int_{k\backslash\BA}\int_{k\backslash\BA}\sum_{\gamma \in \RT(k)}\sum_{\delta\in \RN(k)}\Phi\left(  \begin{pmatrix}
	1 & n_1 \\ 0& 1
\end{pmatrix} \delta \gamma \begin{pmatrix}
	1 & n_2 \\0 & 1
\end{pmatrix} \right)\psi^{-1}(n_1+n_2)\ud x\ud y\\
&=\int_{x\in\BA}\int_{y\in k\backslash\BA}\sum_{\gamma \in \RT(k)}\Phi\left(\begin{pmatrix}
	1 & n_1 \\0 & 1
\end{pmatrix}  \gamma \begin{pmatrix}
	1 & n_2 \\0 & 1
\end{pmatrix}   \gamma^{-1}\cdot \gamma \right)\psi^{-1}(n_1+n_2)\ud n_1\ud n_2.
\end{align*}
Note that we have for $\gamma=\begin{pmatrix}
	\gamma_1 & 0 \\0 & \gamma_2
\end{pmatrix}$,
\[
\begin{pmatrix}
	\gamma_1 & 0\\0 & \gamma_2
\end{pmatrix}\begin{pmatrix}   1 & n_2 \\0 & 1\end{pmatrix}   \begin{pmatrix}
	\gamma_1 & 0\\ 0& \gamma_2
\end{pmatrix}^{-1}  =  \begin{pmatrix}
	1 & \gamma_1\gamma_2^{-1}n_2 \\0 & 1
\end{pmatrix} ,
\]
and make a change of variable by $y\mapsto \gamma_1^{-1}\gamma_2 y$ and $x\mapsto x-y$, the above integral becomes
\[
\int_{x\in\BA}\int_{y\in k\backslash\BA}\Phi\left(  \begin{pmatrix}
	1 & n_1 \\0 & 1
\end{pmatrix}\gamma   \right)\psi^{-1}(n_1-n_2+\gamma_1^{-1}\gamma_2n_2)\ud n_1\ud n_2,
\]
with the integral over $y$ being non-zero only when $\gamma_1=\gamma_2$, in which case it is equal to
\[
  \sum_{\gamma\in \RZ(k)}\int_{\BA}\Phi\left(  \begin{pmatrix}
	1 & n_1\\0 & 1
\end{pmatrix} \gamma \right)\psi^{-1}(n_1) \ud n_1,
\]
where $Z$ is the center of $\GL_2$ consisting of scalar matrices. This inner integral admits an Euler product
\[
\prod_{\nu}\int_{k_{\nu}}\Phi_{\nu}\left(  \begin{pmatrix}
	1 & x \\ 0 &  1
\end{pmatrix} \gamma    \right)\psi_{\nu}^{-1}(x)\ud x,
\]
which according to Proposition \ref{prp_irr_2_loc-glo} is 
\[
\sum_{t_2\in k^{\times}}\CO_{u,2,\BA}(f)(t_2).
\]

Next, for the rank $1$ matrices, it turns out the only non-zero-sum is the following
\[
\int_{k\backslash \BA}\int_{k\backslash\BA}\sum_{0\neq \gamma=(\gamma_{ij}),\gamma_{12}\neq 0,\det\gamma=0}\Phi\left(  \begin{pmatrix}
	1 & n_1 \\ 0& 1
\end{pmatrix}   \gamma\begin{pmatrix}
	1 & n_2 \\ 0& 1
\end{pmatrix}  \right)\psi^{-1}(n_1+n_2)\ud n_1\ud n_2.
\]
We can simplify the above sum, which is
\begin{align*}
&\int_{k\backslash \BA}\int_{k\backslash\BA}\sum_{t_1 \in k^{\times}} \sum_{\delta_1\in \RN(k)}\sum_{\delta_2\in \RN(k)}\Phi\left(  \begin{pmatrix}
	1 & n_1 \\0 & 1
\end{pmatrix}   \delta_1 \begin{pmatrix}
	0&0 \\ t_1 & 0 
\end{pmatrix} \delta_2\begin{pmatrix}
	1 & n_2 \\ 0& 1
\end{pmatrix}  \right) \psi^{-1}(n_1+n_2) \ud n_1\ud n_2 \\
&=\sum_{t_1 \in k^{\times}} \int_{\BA}\int_{\BA}\Phi\left(   \begin{pmatrix}
	1 & n_1 \\ 0& 1
\end{pmatrix}  \begin{pmatrix}
	0 & 0  \\ t_1 & 0 
\end{pmatrix} \begin{pmatrix}
	1 & n_2 \\ 0 & 1
\end{pmatrix} \right)\psi^{-1}(n_1+n_2)\ud n_1\ud n_2,
\end{align*}
Also, note that this sum is nothing but
\[
\sum_{t_1\in k^{\times}}\CO_{u,1,\BA}(f)(t_1).
\]

Then, the integral we start with is just
\[
\sum_{t\in \RT(k)}\frac{f}{(\ud t)^{\frac{1}{2}}}(t)+\sum_{t_1\in k^{\times}}\CO_{u,1,\BA}(f)(t_1)+\sum_{t_2\in k^{\times}}\CO_{u,2,\BA}(f)(t_2)=\mathrm{KTF}_{L\left(\std,\frac{1}{2}\right)}(f).
\]

For fixed $n_1,n_2\in \RN(\BA)$, we have the usual Poisson summation formula
\[
\sum_{\gamma\in \Mat_2(k)}\Phi(n_1\gamma n_2)=\sum_{\gamma\in\Mat_2^{\vee}(k)}\widehat{\Phi}(n_2^{-1}\gamma n_1^{-1}),
\]
where $\widehat{\Phi}:=\otimes^{\prime}_{\nu}\widehat{\Phi_{\nu}}$ is the Fourier transform of $\Phi$. Then we have
\begin{align*}
\int_{k\backslash\BA}\int_{k\backslash\BA}\sum_{\gamma\in \Mat_{2}(k)}\Phi\left(  \begin{pmatrix}
	1 & n_1 \\ 0& 1
\end{pmatrix} \gamma \begin{pmatrix}
	1 & n_2 \\ 0& 1
\end{pmatrix} \right)\psi^{-1}(n_1+n_2)\ud n_1\ud n_2\\
=\int_{k\backslash\BA}\int_{k\backslash\BA}\sum_{\gamma\in\Mat^{\vee}(k)}\widehat{\Phi}\left(  \begin{pmatrix}
	1 & -n_1 \\0 & 1
\end{pmatrix} \gamma \begin{pmatrix}
	1 & -n_2 \\0 & 1
\end{pmatrix}   \right)\psi^{-1}(n_1+n_2)\ud n_1\ud n_2,
\end{align*}
which is 
\[
\int_{k\backslash\BA}\int_{k\backslash\BA}\sum_{\gamma\in\Mat^{\vee}(k)}\widehat{\Phi}\left(  \begin{pmatrix}
	1 & n_1 \\ 0 & 1
\end{pmatrix} \gamma \begin{pmatrix}
	1 & n_2 \\0 & 1
\end{pmatrix}   \right)\psi(n_1+n_2)\ud x\ud y
\]
by changing $n_1$ to $-n_1$, and $n_2$ to $-n_2$. Similarly, we have 
\[
\Mat_{2}^{\vee}(k)=\RB(k)w\RB(k)\bigsqcup \RB(k)\bigsqcup \{g\in \Mat_2^{\vee}(k)\mid  g\neq0,\det g=0\}\bigsqcup\{0\},
\]
and the same computation shows that
\begin{align*}
	\int_{k\backslash\BA}\int_{k\backslash\BA} \sum_{\gamma\in \RB(k)w\RB(k)\subset\Mat_2^{\vee}}\widehat{\Phi}\left(  \begin{pmatrix}
	1 & n_1 \\ 0 & 1
\end{pmatrix} \gamma \begin{pmatrix}
	1 & n_2 \\0 & 1
\end{pmatrix}   \right)\psi(n_1+n_2)\ud n_1\ud n_2 =\sum_{t\in \RT(k)}\frac{\CH_{\std,\BA}(f)}{(\ud t)^{\frac{1}{2}}}(t),
\end{align*}
\begin{align*}
	\int_{k\backslash\BA}\int_{k\backslash\BA} \sum_{\gamma\in \RB(k)\subset\Mat_2^{\vee}}\widehat{\Phi}\left(  \begin{pmatrix}
	1 & n_1 \\ 0 & 1
\end{pmatrix} \gamma \begin{pmatrix}
	1 & n_2 \\ 0 & 1
\end{pmatrix}   \right)&\psi(n_1+n_2)\ud n_1\ud n_2 \\
 &=\sum_{t_1\in k^{\times}}\CO_{u,2,\BA}^{\vee}(\CH_{\std,\BA}(f))(t_1),
\end{align*}
and
\begin{align*}
\int_{k\backslash \BA}\int_{k\backslash\BA}\sum_{0\neq \gamma=(\gamma_{ij}),\gamma_{12}\neq 0,\det\gamma=0}\widehat{\Phi}\left(  \begin{pmatrix}
	1 & n_1 \\ 0 & 1
\end{pmatrix}   \gamma\begin{pmatrix}
	1 & n_2 \\0 & 1
\end{pmatrix}  \right)&\psi(n_1+n_2)\ud n_1\ud n_2\\
&=\sum_{t_2\in k^{\times}} \CO_{u,1,\BA}^{\vee}(\CH_{\std,\BA}(f))(t_2).
\end{align*}
Therefore, we obtain
\[
\mathrm{KTF}_{L\left(\std,\frac{1}{2}\right)}(f)=\mathrm{KTF}_{L\left(\std^{\vee},\frac{1}{2}\right)}(\CH_{\std,\BA}(f)).
\]

\end{proof}

\subsection{Poisson Summation Formula: A Direct Proof}\label{ssec_dirtpf}

In this section, we will use the local properties we established in Section \ref{sec_local spaces} and a two-step classical one-dimensional Poisson Summation Formula to prove Theorem \ref{thm_psf} directly. This is compatible with the local Hankel transform in the sense that the local Hankel transform is a composition of two Fourier transforms with some mysterious correction factors. The point of this section is that we can give a proof of Theorem \ref{thm_psf} independent of the Poisson summation formula on the upstairs space, which has the potential to be generalized to other reductive groups and Hankel transforms. This may be used to establish analytic properties of general automorphic $L$-functions.

To begin with, let us assume $f=\otimes^{\prime}_{\nu} f_{\nu}$. For each $f_{\nu}$, recall that we have the intermediate half-densities and normalized intermediate functions defined in Definition \ref{dfn_int}:
    	\[
    	\varphi_{\nu}:=\psi_{\nu}(-e^{\alpha_1})\circ\CF^{\psi_{\nu}}_{-\epsilon_2^{\vee},\frac{1}{2}}(f_{\nu}),
    	\]
    	\[
    	\widetilde{\varphi_{\nu}}=\frac{\varphi_{\nu}(t_1,t_2)}{|\frac{t_1}{t_2}|_{\nu}^{\frac{1}{2}}(\ud t)^{\frac{1}{2}}},
    	\]
    	and write $\widetilde{\varphi_f}:=\otimes^{\prime}_{\nu}\widetilde{\varphi_{\nu}}$. 
    
    \begin{lem}\label{lem_globla1}
    	For $f\in\CCD^-_{L\left(\std,\frac{1}{2}\right)}(\RN(\BA),\psi\backslash \RG(\BA)/\RN(\BA),\psi)$, we have
    	\[
\sum_{t\in \RT(k)}\frac{\CH_{\std,\BA}(f)}{(\ud t)^{\frac{1}{2}}}(t)+\sum_{t_2\in k^{\times}}\CO_{u,1,\BA}^{\vee}(\CH_{\std,\BA}(f))(t_2)=\sum_{t_1\in k,t_2\in k^{\times}}\widetilde{\varphi_f}(t_1,t_2),
\]
where $t=\begin{pmatrix}
	t_1 & 0 \\ 0 & t_2
\end{pmatrix}$.
    \end{lem}
    
    \begin{proof}
        According to Remark \ref{rmk_han}, we have
        \begin{align*}
	\frac{\CH_{\std,\nu}(f_{\nu})}{(\ud t)^{\frac{1}{2}}}(t)=|t_1|^{-\frac{1}{2}} |t_2|^{\frac{1}{2}}\int_{k_{\nu}}\widetilde{\varphi_{\nu}}(t_1,t_2)\psi\left(\frac{x}{t_1}\right)\ud x,
\end{align*}
hence
\begin{align*}
	\sum_{t\in \RT(k)}\frac{\CH_{\std,\BA}(f)}{(\ud t )^{\frac{1}{2}}}(t)=\sum_{(t_1,t_2)\in k^{\times}\times k^{\times}}\BF_{\psi,1,\BA}(\widetilde{\varphi_f}(\cdot,t_2))\left(-\frac{1}{t_1}\right)
\end{align*}
since $|t_1|=|t_2|=1$ for $t_1,t_2\in k^{\times}$. Thanks to Lemma \ref{lem_irr1+han}, we have
\[
\CO_{u,1,\BA}^{\vee}(\CH_{\std,\BA}(f))=\sum_{t_2\in k^{\times}}|t_2|\BF_{\psi,1,\BA}(\widetilde{\varphi_f}(\cdot,t_2))(0)=\sum_{t_2\in k^{\times}}\BF_{\psi,1,\BA}(\widetilde{\varphi_f}(\cdot,t_2))(0).
\]
Combining them, we have
\begin{align*}
	&\sum_{t\in \RT(k) }\frac{\CH_{\std,\BA}(f)}{(\ud t)^{\frac{1}{2}}}(t)+ \sum_{t_2\in k^{\times}}\CO_{u,1,\BA}^{\vee}(\CH_{\std,\BA}(f))(t_2)\\
	&=\sum_{(t_1,t_2)\in k^{\times}\times k^{\times}}\BF_{\psi,1,\BA}(\widetilde{\varphi_f} (\cdot,t_2))\left(-\frac{1}{t_1}\right)+\sum_{t_2\in k^{\times}}\BF_{\psi,1,\BA}(\widetilde{\psi_f}(\cdot,t_2))(0)\\
	&=\sum_{t_2\in k^{\times}}\sum_{t_1\in k}\BF_{\psi,1,\BA}(\widetilde{\varphi_f}(\cdot,t_2))(t_1).
\end{align*}
Then, for each fixed $t_2\in k^{\times}$, since $\widetilde{\varphi_f}$ is a function of Schwartz type according to Proposition \ref{prp_mid}, we can apply the classical Poisson summation formula to obtain the above is
\[
\sum_{t_1\in k,t_2\in k^{\times}}\widetilde{\varphi_f}(t_1,t_2).
\]
    \end{proof}
    
\begin{lem}
	We have
	\begin{align*}
		\sum_{t_1\in k,t_2\in k^{\times}}\widetilde{\varphi_f}(t_1,t_2)-\sum_{t_2\in k^{\times}}\CO_{u,2,\BA}(f)(t_2)&+\sum_{t_1\in k^{\times}}\CO_{u,2,\BA}^{\vee}(\CH_{\std,\BA}(f))(t_1)\\
		&=\sum_{t\in \RT(k)}\frac{f}{(\ud t)^{\frac{1}{2}}}(t)+\sum_{t_1\in k^{\times}}\CO_{u,1,\BA}(f)(t_1).
	\end{align*}
\end{lem}
    \begin{proof}
    	According to Definition \ref{dfn_irr2}, we have
    	\[
    	\sum_{t_2\in k^{\times}}\CO_{u,2,\BA}(f)=\sum_{t_2\in k^{\times}}\widetilde{\varphi_f}(0,t_2).
    	\]
    	Write $f_{\nu}=|\det t|_{\nu}\delta^{\frac{1}{2}}(t)\CO_t(\Phi_{\nu})$ and $\CO_t(\Phi)=\otimes^{\prime}_{\nu}\CO_t(\Phi_{\nu})$, then according to Lemma \ref{lem_irr2+han}, we have
    	\[
    	\CO_{u,2,\BA}^{\vee}(\CH_{\std,\BA}(f))=\sum_{z\in k^{\times}}|z|^{3}\BF_{\psi,2,\BA}(\CO_{(-z,\cdot)}(\Phi))(0)=\sum_{z\in k^{\times}}\BF_{\psi,2,\BA}(\CO_{(-z,\cdot)}(\Phi))(0).
    	\]
    	Therefore, the left-hand-side of the equality in this lemma is
    	\begin{align*}
    		&\sum_{t_1\in k,t_2\in k^{\times}}\widetilde{\varphi_f}(t_1,t_2)-\sum_{t_2\in k^{\times}}\widetilde{\varphi_f}(0,t_2)+\sum_{z\in k^{\times}}\BF_{\psi,2,\BA}(\CO_{(-z,\cdot)}(\Phi))(0)\\
    		&=\sum_{t_1\in k^{\times},t_2\in k^{\times}}\widetilde{\varphi_f}(t_1,t_2)+\sum_{z\in k^{\times}}\BF_{\psi,2,\BA}(\CO_{(-z,\cdot)}(\Phi))(0).\\
    	\end{align*}
    	Use Lemma \ref{lem_mid=F}, we have the above is
    	\begin{align*}
    		&\sum_{t_1\in k^{\times},t_2\in k^{\times}}|t_1|\psi\left( -\frac{t_2}{t_1}  \right)\BF_{\psi,2,\BA}(\CO_{(t_1,\cdot)}(\Phi))\left(-\frac{1}{t_2}\right)+\sum_{z\in k^{\times}}\BF_{\psi,2,\BA}(\CO_{(-z,\cdot)}(\Phi))(0)\\
    		&=\sum_{t_1\in k^{\times},t_2\in k^{\times}}\BF_{\psi,2,\BA}(\CO_{(t_1,\cdot)}(\Phi))\left(-\frac{1}{t_2}\right)+\sum_{z\in k^{\times}}\BF_{\psi,2,\BA}(\CO_{(-z,\cdot)}(\Phi))(0)\\
    		&=\sum_{t_1\in k^{\times}}\sum_{t_2\in k}\BF_{\psi,2,\BA}(\CO_{(t_1,\cdot)}(\Phi))(t_2).
    	\end{align*}
    	According to Proposition \ref{prp_irr1}, for each $t_1\in k^{\times}$, $\CO_{(t_1,\cdot)}(\Phi)$ is a function of Schwartz type, we can apply the Poisson summation formula to it. Therefore we obtain
    	\begin{align*}
    		\sum_{t_1\in k^{\times}}\sum_{t_2\in k}\BF_{\psi,2,\BA}(\CO_{(t_1,\cdot)}(\Phi))(t_2)&=\sum_{t_1\in k^{\times}}\sum_{t_2\in k}\CO_{(t_1,t_2)}(\Phi)\\
    		&=\sum_{t_1\in k^{\times},t_2\in k^{\times}}\CO_{(t_1,t_2)}(\Phi)+\sum_{t_1\in k^{\times}}\CO_{(t_1,0)}(\Phi)\\
    		&=\sum_{t\in \RT(k)}\frac{f}{(\ud t)^{\frac{1}{2}}}(t)+\sum_{t_1\in k^{\times}}\CO_{u,1,\BA}(f)(t_1).
    	\end{align*}
    	    \end{proof}
    
Combining these two lemmas, we get Theorem \ref{thm_psf}.

\section{Kuznetsov Trace Formula with Godement-Jacquet Sections}\label{sec_tf}

In this section, we will first establish a Kuznetsov-type trace formula with test functions being the Schwartz functions on the space of $2\times 2$ matrices. This is one example of a trace formula with non-standard test functions. The geometric side will be related to Theorem \ref{thm_psf}, which allows us to apply the Poisson summation formula on the geometric side. The spectrum side has already been studied in \cite{Wu}. 

Let $\Phi$ be a Schwartz function on $\Mat_2(\BA)$. Then for $s\in\BC$ with $\Re(s)>1$, we have $g\mapsto \Phi(g)|\det g|^s\cdot|\det g|\in\CL^1(\GL_2)$. For any unitary idele class character $\omega$, consider the right regular action $\RR_{\omega}$ on $\CL^2(\GL_2(k)\backslash \GL_2(\BA),\omega)$. By calculating the Whittaker period of the kernel function of $\RR(\Phi(\cdot)|\det\cdot|^{s+1})$, we will prove the following 
\begin{thm}\label{thm_tf}
	For $\Phi\in\CF(\Mat_2(\BA))$, and $s\in\BC$ with $\Re(s)>1$,
	 we have \begin{align*}
	&\frac{1}{(\ud x)^{\frac{1}{2}}}\int_{[\RZ]}\mathrm{KTF}( \RR_z(\Phi(x)(\ud^+x)^{\frac{1}{2})})\omega(z)|z|^{2s}\ud^{\times}z \\
    &=\sum_{\pi\;\mathrm{cuspidal},\omega_{\pi}=\omega} \sum_{\varphi_1.\varphi_2\in\CB(\pi)}\CZ\left(s+\frac{1}{2},\Phi,\beta(\varphi_2,\varphi_1^{\vee})\right)\BW_{\varphi_1}^{\psi^{-1}}(\RI_2)\BW_{\varphi_2}^{\psi^{-1}}(\RI_2)\\
	&+\sum_{\chi\in\widehat{\BR_+k^{\times}\backslash\BA^{\times}}}\int_{-\infty}^{\infty} \sum_{e_1,e_2\in\CB(\chi,\omega\chi^{-1})}\CZ\left(s+\frac{1}{2},\Phi,\beta_{i\tau}(e_2,e_1^{\vee})\right)\BW^{\psi^{-1}}(i\tau,e_1)(1)\BW^{\psi}(-i\tau,e_2^{\vee})(1)  \frac{\ud \tau}{4\pi}\\
	&+\int_{[\RN]\times[\RN]} \RK^{\prime}(n_1,n_2)\psi^{-1}(n_1^{-1}n_2)\ud n_1\ud n_2,
	\end{align*}
	with the last term being the Whittaker function of the Eisenstein series. For the undefined terms, see the next two sections.
	\end{thm}

\subsection{The Spectrum Side of the Trace Formula}

First, let us set up notations for the spectrum side. We fix a section of the adelic norm map $|\cdot|:\BA^{\times}\rightarrow\BR_+$ and identify $\BR_+$ as a subgroup of $\BA^{\times}$. Let $\omega$ be a unitary idele class character that is also trivial on $\BR_+$. Denote $\CL^2(\GL_2,\omega)$ the space of functions $\varphi:\GL_2(k)\backslash\GL_2(\BA)\rightarrow\BC$ such that
\begin{itemize}
	\item [(1)] $\displaystyle{\varphi(gz)=\omega(z)\varphi(g)}$, $\forall g\in\GL_2(\BA)$ and $z\in \RZ(\BA)$,
	\item [(2)] $\displaystyle{\int_{Z(\BA)\GL_2(k)\backslash\GL_2(\BA)}|\varphi(g)|^2\ud g<\infty}$.
\end{itemize}
Then $\GL_2(\BA)$ acts on it via the right regular action, and we denote this action by $(\RR_{\omega},\CL^2(\GL_2,\omega))$. 

Let $(\pi,V_{\pi})$ be a cuspidal automorphic representation in $(\RR_{\omega},\CL^2(\GL_2,\omega))$, we will use $V_{\pi}^{\infty}$ to denote the subspace of smooth vectors. For $\varphi\in V_{\pi}^{\infty}$, we write
\[
\BW_{\varphi}^{\psi}(g):=\int_{[\RN]}\varphi(ng)\psi^{-1}(n)\ud n
\]
for the Whittaker function of $\varphi$.

Similarly let $\chi$ be a unitary character of $\BR_+k^{\times}\backslash\BA^{\times}$, we can associate the principal series $\pi_s:=\pi(\chi|\cdot|^s,\omega\chi^{-1}|\cdot|^{-s})$ for all $s\in\BC$, whose underlying Hilbert space is $V_{\chi,\omega\chi^{-1}}$ consisting all the functions $f:\RK\rightarrow\BC$ satisfying
\begin{itemize}
	\item [(1)] $\displaystyle{ f\left(  \begin{pmatrix}
		t_1 & u \\ 0 & t_2
	\end{pmatrix}g \right)=\chi(t_1)\omega\chi^{-1}(t_2)f(g) }$,
	\item [(2)] $  \displaystyle{ \int_{\RK}|f(k)|^2\ud k<\infty}$,
\end{itemize}
where $\RK$ is the standard maximal compact subgroup of $\GL_2(\BA)$, which is the tensor product of $\RK_{\nu}$, and
\[
\RK_{\nu} =     \left\{ \begin{array}{rcl}
         \SO_2(\BR) & \mbox{if}
         & k_{\nu}\cong\BR \\ \SU_2(\BC) & \mbox{if} & k_{\nu} \cong\BC \\
        \GL_2(\Fo_{\nu}) & \mbox{if}   &\nu<\infty             \end{array}\right..
\]
This space is common for all $s\in\BC$ and the subspace of smooth vectors, although the actions $\pi_s$ are different. In particular, $\pi_s$ is unitary if and only if $s\in i\BR$. To any smooth vector $e\in V^{\infty}_{\chi,\omega\chi^{-1}}$, we associate a flat section $e(s)$ in $\pi_s$, from which we construct an Eisenstein series
\begin{align*}
	\RE(s,e)(g):=\sum_{\gamma\in \RB(k)\backslash\GL_2(k)} e(s)(\gamma g),
\end{align*}
which converges for $\Re(s)\gg1$ and admits a meromorphic continuation to $s\in\BC$. We write
\[
\BW^{\psi}(s,e)(g):=\int_{[\RN]}\RE(s,e)(ng)\psi^{-1}(n)\ud n=\int_{\RN(\BA)}e(wng)\psi^{-1}(n)\ud n
\]
for the Whittaker function of $\RE(s,e)$.

For $V=V_{\pi}$ or $V_{\chi,\omega\chi^{-1}}$, the underlying inner product identifies $V$ with its dual $V^{\vee}$. Define $\displaystyle{\varphi^{\vee}:=  \frac{\overline{\varphi}}{\|\varphi\|^2} \in V^{\vee} }$ for $\varphi\neq 0$. The matrix coefficient $\beta(\varphi_2,\varphi_1^{\vee})$ and $\beta_s(e_2,e_1^{\vee})$ for a pair of non-zero vectors is defined to be
\[
\beta(\varphi_2,\varphi_1^{\vee})(g):=\frac{\langle\pi(g)\varphi_2,\varphi_1\rangle }{\|\varphi_1\|^2}\;\mathrm{and}\;\beta_s(e_2,e_1^{\vee})(g):=\frac{\langle \pi_s(g)e_2,e_1\rangle }{\|e_1\|^2},\;\forall g\in \GL_2(\BA).
\]
For a Schwartz function $\Phi$ on $\Mat_2(\BA)$, we have the Godement-Jacquet zeta integrals for $\beta\in\{  \beta(\varphi_2,\varphi_1^{\vee}),\beta_s(e_2,e_1^{\vee}) \}$
\[
\CZ(s^{\prime},\Phi,\beta):=\int_{\GL_2(\BA)}\Phi(g)\beta(g)|\det g|^{s^{\prime}+\frac{1}{2}}\ud g,
\]
which converges for $\Re(s)\gg 1$.

Let $\Phi$ be a Schwartz function on $\Mat_2(\BA)$. Then for $s\in\BC$ with $\Re(s)>1$, we have $g\mapsto \Phi(g)|\det g|^s\cdot|\det g|\in\CL^1(\GL_2)$ and it acts on $(\RR_{\omega},\CL^2(\GL_2,\omega))$. For any $\varphi\in\CL^2(\GL_2(k)\backslash\GL_2(\BA),\omega)$, we have for $x\in\GL_2(\BA)$,
\begin{align*}
	&\RR_{\omega}(\Phi(\cdot)|\det \cdot|^{s+1})\varphi(x)=\int_{\GL_2(\BA)}\Phi(y)|\det y|^{s+1}\varphi(xy)\ud y \\
	&=\int_{\GL_2(\BA)}\Phi(x^{-1}y)|\det x^{-1}y|^{s+1}\varphi(y)\ud y\\
	&=\int_{[\GL_2]} \sum_{\gamma\in\GL_2(k)}\Phi(x^{-1}\gamma y )|\det x^{-1}y|^{s+1}\varphi(y)\ud y\\
	&=\int_{\GL_2(k)\RZ(\BA)\backslash\GL_2(\BA)} \int_{[\RZ]} \sum_{\gamma\in\GL_2(k)}\Phi(x^{-1}\gamma y z)|z|^{2s+2}\omega(z)\ud^{\times}z |\det x^{-1}y|^{s+1} \varphi(y)\ud y.
\end{align*}
We denote the corresponding kernel function
\[
\RK(x,y):=\int_{[\RZ]} \sum_{\gamma\in\GL_2(k)}\Phi(x^{-1}\gamma y z)|z|^{2s+2}\omega(z)\ud^{\times}z \cdot |\det x^{-1}y|^{s+1}.
\]
According to \cite[Theorem 2.10]{Wu}, we have the $\CL^2$-spectrum decomposition
\begin{thm}
	When $\Re(s)>1$, the Fourier inversion of $\RK(x,y)$ with respect to $y$ in $\CL^2(\GL_2,\omega^{-1})$ converges normally for $(x,y)\in [\GL_2]\times [\GL_2]$, and takes the form
	\begin{align*}
	\RK(x,y)&=\sum_{\pi\;\mathrm{cuspidal},\omega_{\pi}=\omega} \sum_{\varphi_1.\varphi_2\in\CB(\pi)}\CZ\left(s+\frac{1}{2},\Phi,\beta(\varphi_2,\varphi_1^{\vee})\right)\varphi_1(x)\varphi_2^{\vee}(y)\\
	&+\sum_{\chi\in\widehat{\BR_+k^{\times}\backslash\BA^{\times}}}\int_{-\infty}^{\infty} \sum_{e_1,e_2\in\CB(\chi,\omega\chi^{-1})}\CZ\left(s+\frac{1}{2},\Phi,\beta_{i\tau}(e_2,e_1^{\vee})\right)\RE(i\tau,e_1)(x)\RE(-i\tau,e_2^{\vee})(y)\frac{\ud \tau}{4\pi}\\
	&+\frac{1}{\vol[\PGL_2]}\sum_{\eta\in\widehat{k^{\times}\backslash\BA^{\times}},\eta^2=\omega}\left(\Phi(g)\eta(\det g)|\det g|^{s+1} \ud g\right)\eta(\det x)\overline{\eta(\det y)},
\end{align*}
where $\CB(\pi)$ is an orthogonal basis of $\RK$-isotypic vectors in $\pi$, $\CB(\chi,\omega\chi^{-1})$ is an orthogonal basis of $K$-isotypic vectors in $V_{\chi,\omega\chi^{-1}}$.
\end{thm}

Integrating the above formula against the prescribed character over $[\RN]\times[\RN]$, we obtain
\begin{align*}
	&\int_{[\RN]\times[\RN]}\RK(n_1,n_2)\psi^{-1}(n_1^{-1}n_2)\ud n_1\ud n_2\\
	&=\sum_{\pi\;\mathrm{cuspidal},\omega_{\pi}=\omega} \sum_{\varphi_1.\varphi_2\in\CB(\pi)}\CZ\left(s+\frac{1}{2},\Phi,\beta(\varphi_2,\varphi_1^{\vee})\right)\BW_{\varphi_1}^{\psi^{-1}}(\RI_2)\BW_{\varphi_2^{\vee}}^{\psi^{}}(\RI_2)\\
	&+\sum_{\chi\in\widehat{\BR_+k^{\times}\backslash\BA^{\times}}}\int_{-\infty}^{\infty} \sum_{e_1,e_2\in\CB(\chi,\omega\chi^{-1})}\CZ\left(s+\frac{1}{2},\Phi,\beta_{i\tau}(e_2,e_1^{\vee})\right)\BW^{\psi^{-1}}(i\tau,e_1)(\RI_2)\BW^{\psi}(-i\tau,e_2^{\vee})(\RI_2)  \frac{\ud \tau}{4\pi}.
\end{align*}

\subsection{The Geometric Side of the Trace Formula}\label{ssec_geo}
In this section, we will calculate the integral of $\RK(x,y)$ with the prescribed character over $[\RN]\times[\RN]$ from the geometric aspect.

When $\Re(s)>1$, we have
\begin{align*}
	&\int_{[\RN]\times[\RN]}\RK(n_1,n_2)\psi^{-1}(n_1^{-1}n_2)\ud n_1\ud n_2\\
	&=\int_{[\RN]\times[\RN]}\int_{[\RZ]} \sum_{\gamma\in\GL_2(k)}\Phi(n_1^{-1}\gamma n_2 z)|z|^{2s+2}\omega(z)\ud^{\times}z \cdot \psi^{-1}(n_1^{-1}n_2)\ud n_1\ud n_2\\
	&=\int_{[\RZ]}  \int_{[\RN]\times[\RN]}\Phi(n_1\gamma n_2 z)\psi^{-1}(n_1n_2)\ud n_1\ud n_2\cdot \omega(z)|z|^{2s+2}\ud^{\times} z.
\end{align*}

In addition to the kernel function $\RK(x,y)$, we can also consider the following function given by a similar summation over rank $1$ rational matrices when $\Re(s)\gg 1$:
\begin{align*}
		\RK^{\prime}(x,y):&=\int_{[\RZ]}\sum_{\gamma \in\Mat_{2}(k),\rank\gamma=1}\Phi(x^{-1}\gamma yz)\omega(z)|z|^{2s+2}\ud^{\times}z\cdot|\det x^{-1}y|^{s+1}.\\
	\end{align*}
According to \cite[Section 2.2]{Wu}, when $\displaystyle{\Re(s)\gg 1}$, denote
\begin{align*}
		\RR(x,y):=&\int_{[\RZ]} \sum_{\gamma\in k^{\times}}\Phi\left( x^{-1}\begin{pmatrix}
			0 & \gamma \\ 0 & 0
		\end{pmatrix}  y z\right)\omega(z)|z|^{2s+2}\ud^{\times}z\cdot|\det x^{-1}y|^{s+1} \\
		&=\int_{\BA^{\times}}\Phi\left(x^{-1} \begin{pmatrix}
			0 & z \\ 0 & 0
		\end{pmatrix}  y\right)\omega(z)|z|^{2s+2}\ud^{\times}z \cdot |\det x^{-1}y|^{s+1}.
	\end{align*}
	It is easy to see that for fixed $x$, $y\mapsto \RR(x,y)$ belongs to the induced model of the principle series representation $\pi(|\cdot|^{s+\frac{1}{2}},\omega^{-1}|\cdot|^{-s-\frac{1}{2}})$, and for fixed $y$, $x\mapsto \RR(x,y)$ is in $\pi(\omega|\cdot|^{s+\frac{1}{2}},|\cdot|^{-s-\frac{1}{2}})$. Note that the action of $\GL_2(k)\times\GL_2(k)$ on the rank $1$ rational matrices admits a single orbit and the stabilizer group of the line consisting of the elements 
	\[
	\begin{pmatrix}
		0 & \gamma \\ 0 & 0
	\end{pmatrix},\gamma\in k^{\times}
	\]
	is $\RB(k)\times \RB(k)$, therefore,
	\[
	\RK^{\prime}(x,y)=\sum_{\gamma_1\in \RB(k)\backslash \GL_2(k)}\sum_{\gamma_2\in \RB(k)\backslash\GL_2(k)}\RR(\gamma_1x,\gamma_2y)
	\]
is an Eisenstein series in each variable. 

For notational simplicity, let us still write $\mathrm{KTF}$ for the composition of $\mathrm{KTF}$ and the twisted push-forward from $\CCD(\Mat_2(\BA))$ to $\CCD_{L\left(\std,\frac{1}{2}\right)}(\RN(\BA),\psi \backslash\RG /\RN(\BA),\psi)$. Then, the geometric side can be written as
\[
\int_{[\RZ]}\mathrm{KTF}(\RR_z(\Phi(x)(\ud^+x)^{\frac{1}{2}}  ))\omega(z)|z|^{2s}\ud^{\times}z-\int_{[\RN]\times[\RN]}\RK^{\prime}(n_1,n_2)\psi^{-1}(n_1^{-1}n_2)\ud n_1\ud n_2 ,\;\Re(s)\gg 1,
\]
and Theorem \ref{thm_tf} holds.

\section{Analytic Continuation}\label{sec_cont}
In this section, we will prove the analytic continuation of terms in Theorem \ref{thm_tf} except for the cuspidal part. It will be independent of the Godement-Jacquet theory and the Poisson summation formula on $\Mat_2(\BA)$. For simplicity, denote
\[
\CI(s):=\int_{[\RZ]} \mathrm{KTF}(\RR_z(\Phi(x)(\ud^+x)^{\frac{1}{2}}  ))\omega(z)|z|^{2s}\ud^{\times}z ,\;\Re(s)>1,
\]
\[
\CI_{\rb}(s):=\int_{[\RN]\times[\RN]} \RK^{\prime}(n_1,n_2)\psi^{-1}(n_1^{-1}n_2)\ud n_1\ud n_2,\;\Re(s)\gg 1,
\]
\[
\CI_{\cusp}(s):=\sum_{\pi\;\mathrm{cuspidal},\omega_{\pi}=\omega} \sum_{\varphi_1.\varphi_2\in\CB(\pi)}\CZ\left(s+\frac{1}{2},\Phi,\beta(\varphi_2,\varphi_1^{\vee})\right)\BW_{\varphi_1}^{\psi^{-1}}(\RI_2)\BW_{\varphi_2^{\vee}}^{\psi^{}}(\RI_2),
\]
and also when $\Re(s)>1$:
\begin{align*}
&\CI_{\mathrm{Eis}}(s)\\
&:=\sum_{\chi\in\widehat{\BR_+k^{\times}\backslash\BA^{\times}}}\int_{-\infty}^{\infty} \sum_{e_1,e_2\in\CB(\chi,\omega\chi^{-1})}\CZ\left(s+\frac{1}{2},\Phi,\beta_{i\tau}(e_2,e_1^{\vee})\right)\BW^{\psi^{-1}}(i\tau,e_1)(\RI_2)\BW^{\psi}(-i\tau,e_2^{\vee})(\RI_2)  \frac{\ud \tau}{4\pi}.
\end{align*}

We may choose $\Phi$ to be $\RK$-finite, hence the summation of $\chi,e_1,e_2$ in $\CI_{\Eis}$ is finite and we denote
\[
\CI_{\chi,e_1,e_2}(s):=\int_{-\infty}^{\infty}\CZ\left(s+\frac{1}{2},\Phi,\beta_{i\tau}(e_2,e_1^{\vee})\right)\BW^{\psi^{-1}}(i\tau,e_1)(\RI_2)\BW^{\psi}(-i\tau,e_2^{\vee})(\RI_2)  \frac{\ud \tau}{4\pi}.
\]

\subsection{Analytic Continuation of $\CI(s)$}

In this section, we will study the convergence of $\CI(s)$. According to the calculation in Section \ref{ssec_inproof}, write
\[
\BO_1=\bigsqcup_{\beta\in k^{\times},\gamma\in k}\RN(k)\begin{pmatrix}
	 0 & \gamma \\ \beta & 0 
\end{pmatrix}\RN(k),\;  \BO_2=\bigsqcup_{\alpha\in k^{\times}} \RN(k)\begin{pmatrix}
	\alpha  & 0 \\ 0 & \alpha
 \end{pmatrix}\RN(k),
\]
and
\[
\RK_i(z):=\int_{[\RN]\times[\RN]} \sum_{\gamma\in\BO_i} \Phi(n_1\gamma n_2z)\psi^{-1}(n_1+n_2)\ud n_2\ud n_2,
\]
then we have
\[
\CI(s)=\int_{[\RZ]}(\RK_1(z)+\RK_2(z))\omega(z)|z|^{2s+2}\ud^{\times}z.
\]
For the first term, we have
\begin{align*}
	\RK_1(z)&=\int_{\BA^2}\sum_{\beta\in k^{\times},\gamma \in k}\Phi\left( \begin{pmatrix}
		1 & n_1 \\ 0 & 1
	\end{pmatrix}  \begin{pmatrix}
		0 & \gamma \\ \beta & 0
	\end{pmatrix} \begin{pmatrix}
		1 & n_2 \\  0& 1
	\end{pmatrix}z \right)\psi^{-1}(n_1+n_2)\ud n_1\ud n_2\\
	&=\int_{\BA^2} \sum_{\beta\in k^{\times},\gamma \in k} \Phi \left(    \begin{pmatrix}
		\beta n_1z & \gamma z+\beta n_1 n_2 z \\ \beta z & \beta n_2 z
	\end{pmatrix}     \right)\psi^{-1}(n_1+n_2)\ud n_1 \ud n_2.
\end{align*}
Applying the Poisson summation formula, we have
\[
\sum_{\gamma \in k }\Phi \left(    \begin{pmatrix}
		\beta n_1z & \gamma z+\beta n_1 n_2 z \\ \beta z & \beta n_2 z
	\end{pmatrix}     \right)=|z|^{-1}\sum_{\gamma\in k }\BF_{\psi,2}\Phi \left(    \begin{pmatrix}
		\beta n_1 z & \gamma z^{-1} \\ \beta z & \beta n_2 z
	\end{pmatrix}   \right)\psi (\beta n_1 n_2 \gamma z^{-1}).
\]
Integral over $n_2\in \BA$, we have
\begin{align*}
	&\int_{n_2\in \BA}\sum_{\gamma \in k }\Phi \left(    \begin{pmatrix}
		\beta n_1z & \gamma z+\beta n_1 n_2 z \\ \beta z & \beta n_2 z
	\end{pmatrix}     \right)\psi^{-1}(n_2)\ud n_2 \\
	&=|z|^{-1} \int_{n_2\in\BA}\sum_{\gamma\in k }\BF_{\psi,2}\Phi \left(    \begin{pmatrix}
		\beta n_1 z & \gamma z^{-1} \\ \beta z & \beta n_2 z
	\end{pmatrix}   \right)\psi^{-1} (n_2-\beta n_1 n_2 \gamma z^{-1})\ud n_2 \\
	&=|z|^{-2}\sum_{\gamma \in k}\BF_{\psi,4}\BF_{\psi,2}\Phi\left(  \begin{pmatrix}
		\beta n_1 z & \gamma z^{-1} \\ \beta z & \beta^{-1}z^{-1}-n_1\gamma z^{-2}
	\end{pmatrix}   \right).
\end{align*}
We need the following lemma:
\begin{lem}\label{lem_gj}
	Let $\Phi\in\CF(\BA)$. Then we have for any $N>1$,
	\[
	\sum_{\alpha\in k^{\times}}|\Phi(\alpha z)|\ll_N \min\{|z|^{-1},|z|^{-N} \}.
	\]
\end{lem}
\begin{proof}
	This elementary estimation appeared in \cite{GJ72} without proof. For a detailed proof, see \cite{Wu19}.
\end{proof}
According to the above lemma, we have
\begin{align*}
	\RK_1(z)\ll_{N_1} |z|^{-3}\min\{|z|^{-1},|z|^{-N_1} \}(1+\min{|z|,|z|^{N_2}}).
\end{align*}
As for $\RK_2(z)$, we have
\begin{align*}
	\RK_2(z)=\int_{\BA}\sum_{\alpha\in k^{\times}} \Phi\left(  \begin{pmatrix}
		\alpha z & \alpha z n\\ 0 & \alpha z
	\end{pmatrix}  \right) \psi^{-1}(n)\ud n&=|z|^{-1}\sum_{\alpha\in k^{\times}}\BF_{\psi,2}\Phi\left( \begin{pmatrix} \alpha z & \alpha^{-1}z^{-1}\\  0& \alpha z \end{pmatrix} \right) \\
	&\ll_{N_1,N_2}|z|^{-1}\min\{|z|^{-1},|z|^{-N_1}\}\min\{|z|,|z|^{N_2}\}.
\end{align*}
Combining all these, we have
\begin{prp}\label{prp_acofI}
	\begin{align*}
		\int_{z\in [\RZ],|z|\geq 1}\mathrm{KTF}(\RR_z(\Phi(x)(\ud^+x)^{\frac{1}{2}}  ))\omega(z)|z|^{2s}\ud^{\times}z
	\end{align*}
	is convergent for all $s\in\BC$ and defines an entire function.
\end{prp}

When $\Re(s)>1$, by applying Theorem \ref{thm_psf}, we have \[\mathrm{KTF}(\RR_z(\Phi(x)(\ud x)^{\frac{1}{2}}))=\mathrm{KTF}(   \widehat{\RR_z(\Phi)}(x)(\ud x)^{\frac{1}{2}}   ),\] and note that
\[
\widehat{\RR_z(\Phi)}(x)=|z|^{-4}\widehat{\Phi}(z^{-1}x),\;\forall x\in\Mat_2(\BA),z\in Z(\BA).
\] 
Hence we have
\[
\widehat{\RR_z(\Phi)}(x)(\ud x)^{\frac{1}{2}} =\RR_z    (\widehat{\Phi}(x)(\ud^+x )^{\frac{1}{2}})|z|^{-2}.
\]
Therefore,
\[
\int_{|z|<1}\mathrm{KTF}(\RR_z(\Phi(x)(\ud^+x)^{\frac{1}{2}}  ))\omega(z)|z|^{2s}\ud^{\times}z=\int_{|z|<1}\mathrm{KTF}(\RR_z(\widehat{\Phi}(x)(\ud^+x)^{\frac{1}{2}}  ))\omega(z)|z|^{2s}\ud^{\times}z,
\]
which is absolutely convergent for all $s\in\BC$ and defines an entire function according to the above Proposition and the $\GL_2\times\GL_2$ action on $\Mat_2^{\vee}$ that we defined in Section \ref{sec_local}. Therefore we have
\begin{prp}
	$\CI(s)$ admits an analytic continuation to an entire function on $s\in\BC$.
\end{prp}

\subsection{Analytic Continuation of $\CI_{\Eis}$}
In this section we will prove the analytic continuation of $\CI_{\Eis}$, for which it suffices to show the analytic continuation of $\CI_{\chi,e_1,e_2}(s)$ for each triple $(\chi,e_1,e_2)$. Write $z$ for the coordinate of the inner integral of $\tau$, i.e.,
\begin{align*}
\CI_{\chi,e_1,e_2}(s)&=\int_{-\infty}^{\infty}\CZ\left(s+\frac{1}{2},\Phi,\beta_{i\tau}(e_2,e_1^{\vee})\right)\BW^{\psi^{-1}}(i\tau,e_1)(\RI_2)\BW^{\psi}(-i\tau,e_2^{\vee})(\RI_2)  \frac{\ud \tau}{4\pi}\\
&=\int_{\Re(z)=0}\CZ\left(s+\frac{1}{2},\Phi,\beta_{z}(e_2,e_1^{\vee})\right)\BW^{\psi^{-1}}(z,e_1)(\RI_2)\BW^{\psi}(-z,e_2^{\vee})(\RI_2)  \frac{\ud z}{4\pi i}.
\end{align*}
According to \cite[proof of Theorem 2.10]{Wu}, we have
\begin{align*}
	&\CZ\left(s+\frac{1}{2},\Phi,\beta_{z}(e_2,e_1^{\vee})\right)\\
	&=\int_{(\BA^{\times})^2}\BF_{\psi,2}(_{e_1^{\vee}}\Phi_{e_2})\left(\begin{pmatrix}
		t_1 &  0 \\0 & t_2
	\end{pmatrix} \right)|t_1|^{s+\frac{1}{2}+z}\chi(t_1)\omega\chi^{-1}(t_2)|t_2|^{s+\frac{1}{2}-z}\ud^{\times}t_1\ud^{\times}t_2,
\end{align*}
where
\begin{align*}
	_{e_1^{\vee}}\Phi_{e_2}(g):=\int_{\RK^2}e_1^{\vee}(k_1)\Phi(k_1^{-1}gk_2)e_2(k_2)\ud k_1\ud k_2.
\end{align*}
Then according to \cite{Tat67}, we know $\CZ\left(s+\frac{1}{2},\Phi,\beta_z(e_2,e_1^{\vee}) \right)$ is a meromorphic function by itself. Then, we need to consider the case that the zeta integrals have poles on the line $\Re(z)=0$, which could occur only when $\chi=1$ and $\chi=\omega$. 

For any $0<\sigma<\frac{1}{2}$, and consider the case when $\frac{1}{2}<\Re(s)<\frac{1}{2}+\sigma$. According to Stirling's formula, we know $\displaystyle{\CZ\left(s+  \frac{1}{2} ,\Phi,\beta_z(e_2,e_1^{\vee}) \right)}$ is of rapid decay when $\Re(z)\rightarrow\infty$ for $0<\Re(z)<\sigma$. According to \cite[Lemma 1, Page8]{JZ87}, the Whittaker functions at value $\RI_2$ is uniformly bounded on vertical strips, we can move the $z$-integral from $\Re(z)=0$ to $\Re(z)=\sigma$ and pick up the poles, we obtain
\begin{align*}
	&\int_{\Re(z)=0}\CZ\left( s+\frac{1}{2},\Phi,\beta_{z}(e_2,e_1^{\vee})   \right)\BW^{\psi^{-1}}(z)(\RI_2)\BW^{\psi}(-z)(\RI_2)\frac{\ud z}{4\pi i}\\
	&=\int_{\Re(z)=\sigma}\CZ\left( s+\frac{1}{2},\Phi,\beta_{z}(e_2,e_1^{\vee})   \right)\BW^{\psi^{-1}}(z,e_1)(\RI_2)\BW^{\psi}(-z,e_2^{\vee})(\RI_2)\frac{\ud z}{4\pi i}\\
	&-\frac{1}{2}\sum_{0<\Re(z)<\sigma,s+ \frac{1}{2}\pm z=0,1}\Res\CZ\left( s+\frac{1}{2},\Phi,\beta_{z}(e_2,e_1^{\vee})   \right)\BW^{\psi^{-1}}(z,e_1)(\RI_2)\BW^{\psi}(-z,e_2^{\vee})(\RI_2).
\end{align*}
Note that the only possible residue in this region occurs when $s+\frac{1}{2}-z=1$, which implies $0<\Re(z)<\sigma$ automatically, and hence the residue part is just
\[
-\frac{1}{2} \Res_{z=s-\frac{1}{2}} \CZ\left(s+\frac{1}{2},\Phi,\beta_{z}(e_2,e_1^{\vee})\right)\cdot \BW^{\psi^{-1}}\left(s-\frac{1}{2},e_1\right)(\RI_2)\BW^{\psi}\left(\frac{1}{2}-s,e_2^{\vee} \right)(\RI_2).
\]
Since there is no pole on the line $\Re(z)=\sigma$ when $\frac{1}{2}-\sigma<\Re(s)<\frac{1}{2}+\sigma$, and the integral is absolutely convergent. So we can extend $\CI_{\chi,e_1,e_2}(s)$ to $\Re(s)>0$. 

When $\frac{1}{2}-\sigma<\Re(s)<\frac{1}{2}$, let us then move the $z$-integral to $\Re(z)=0$ again and pick up the poles, we have
\begin{align*}
	&\int_{\Re(z)=\sigma}\CZ\left( s+\frac{1}{2},\Phi,\beta_{z}(e_2,e_1^{\vee})   \right)\BW^{\psi^{-1}}(z,e_1)(\RI_2)\BW^{\psi}(-z,e_2^{\vee})(\RI_2)\frac{\ud z}{4\pi i}\\
	&-\Res_{z=s-\frac{1}{2}} \CZ\left(s+\frac{1}{2},\Phi,\beta_{z}(e_2,e_1^{\vee})\right)\BW^{\psi^{-1}}\left(s-\frac{1}{2},e_1\right)(\RI_2)\BW^{\psi}\left(\frac{1}{2}-s,e_2^{\vee} \right)(\RI_2)\\
	&=\int_{\Re(z)=0}\CZ\left( s+\frac{1}{2},\Phi,\beta_{z}(e_2,e_1^{\vee})   \right)\BW^{\psi^{-1}}(z,e_1)(\RI_2)\BW^{\psi}(-z,e_2^{\vee})(\RI_2)\frac{\ud z}{4\pi i}\\
	&+\frac{1}{2}\Res_{z=\frac{1}{2}-s}\CZ\left( s+\frac{1}{2},\Phi,\beta_{z}(e_2,e_1^{\vee})   \right)\BW^{\psi^{-1}}\left(\frac{1}{2}-s,e_1 \right)(\RI_2)\BW^{\psi}\left(s-\frac{1}{2},e_2^{\vee}\right)(\RI_2) \\
	&-\frac{1}{2}\Res_{z=s-\frac{1}{2}} \CZ\left(s+\frac{1}{2},\Phi,\beta_{z}(e_2,e_1^{\vee})\right) \BW^{\psi^{-1}}\left(s-\frac{1}{2},e_1\right)(\RI_2)\BW^{\psi}\left(\frac{1}{2}-s,e_2^{\vee} \right)(\RI_2).\\
\end{align*}
Note that this can be analytically extended to $     -\frac{1}{2}<\Re(s)<\frac{1}{2}$ since there is pole on the line $\Re(z)=0$ in this region. 

Next consider the case that $\Re(s)<-\frac{1}{2}+\sigma$, let us move the integral from $\Re(z)=0$ to $\Re(z)=\sigma$ again, we have
\begin{align*}
	&\int_{\Re(z)=0}\CZ\left( s+\frac{1}{2},\Phi,\beta_{z}(e_2,e_1^{\vee})   \right)\BW^{\psi^{-1}}(z,e_1)(\RI_2)\BW^{\psi}(-z,e_2^{\vee})(\RI_2)\frac{\ud z}{4\pi i}\\ 
	&+\frac{1}{2}\Res_{z=\frac{1}{2}-s}\CZ\left( s+\frac{1}{2},\Phi,\beta_{z}(e_2,e_1^{\vee})   \right)\BW^{\psi^{-1}}\left(\frac{1}{2}-s,e_1\right)(\RI_2)\BW^{\psi}\left(s-\frac{1}{2},e_2^{\vee}\right)(\RI_2) \\
	&-\frac{1}{2}\Res_{z=s-\frac{1}{2}} \CZ\left(s+\frac{1}{2},\Phi,\beta_{z}(e_2,e_1^{\vee})\right) \BW^{\psi^{-1}}\left(s-\frac{1}{2},e_1\right)(\RI_2)\BW^{\psi}\left(\frac{1}{2}-s,e_2^{\vee} \right)(\RI_2)\\ 
	&=\int_{\Re(z)=\sigma}\CZ\left( s+\frac{1}{2},\Phi,\beta_{z}(e_2,e_1^{\vee})   \right)\BW^{\psi^{-1}}(z,e_1)(\RI_2)\BW^{\psi}(-z,e_2^{\vee})(\RI_2)\frac{\ud z}{4\pi i}\\ 
	&-\frac{1}{2}\Res_{z=s+\frac{1}{2}}\CZ\left( s+\frac{1}{2},\Phi,\beta_{z}(e_2,e_1^{\vee})   \right)\BW^{\psi^{-1}}\left(s+\frac{1}{2},e_1\right)(\RI_2)\BW^{\psi}\left(-s-\frac{1}{2},e_2^{\vee}\right)(\RI_2)\\
	&+\frac{1}{2}\Res_{z=\frac{1}{2}-s}\CZ\left( s+\frac{1}{2},\Phi,\beta_{z}(e_2,e_1^{\vee})   \right)\BW^{\psi^{-1}}\left(\frac{1}{2}-s,e_1)(\RI_2)\BW^{\psi}(s-\frac{1}{2},e_2^{\vee}\right)(\RI_2) \\
	&-\frac{1}{2}\Res_{z=s-\frac{1}{2}} \CZ\left(s+\frac{1}{2},\Phi,\beta_{z}(e_2,e_1^{\vee})\right) \BW^{\psi^{-1}}\left(s-\frac{1}{2},e_1\right)(\RI_2)\BW^{\psi}\left(\frac{1}{2}-s,e_2^{\vee} \right)(\RI_2),
\end{align*}
which can be analytically continued to $-\sigma-\frac{1}{2}<\Re(s)<-\frac{1}{2}+\sigma$. When $-\sigma-\frac{1}{2}<\Re(s)<-\frac{1}{2}$, we could move the integral to $\Re(z)=0$ again, and we have finally
\begin{align*}
	&\int_{\Re(z)=\sigma}\CZ\left( s+\frac{1}{2},\Phi,\beta_{z}(e_2,e_1^{\vee})   \right)\BW^{\psi^{-1}}(z,e_1)(\RI_2)\BW^{\psi}(-z,e_2^{\vee})(\RI_2)\frac{\ud z}{4\pi i}\\ 
	&-\frac{1}{2}\Res_{z=s+\frac{1}{2}}\CZ\left( s+\frac{1}{2},\Phi,\beta_{z}(e_2,e_1^{\vee})   \right)\BW^{\psi^{-1}}\left(s+\frac{1}{2},e_1\right)(\RI_2)\BW^{\psi}\left(-s-\frac{1}{2},e_2^{\vee}\right)(\RI_2)\\
	&+\frac{1}{2}\Res_{z=\frac{1}{2}-s}\CZ\left( s+\frac{1}{2},\Phi,\beta_{z}(e_2,e_1^{\vee})   \right)\BW^{\psi^{-1}}\left(\frac{1}{2}-s,e_1\right)(\RI_2)\BW^{\psi}\left(s-\frac{1}{2},e_2^{\vee} \right)(\RI_2) \\
	&-\frac{1}{2}\Res_{z=s-\frac{1}{2}} \CZ\left(s+\frac{1}{2},\Phi,\beta_{z}(e_2,e_1^{\vee})\right)\BW^{\psi^{-1}}\left(s-\frac{1}{2},e_1\right)(\RI_2)\BW^{\psi}\left(\frac{1}{2}-s,e_2^{\vee} \right)(\RI_2) \\
	&=\int_{\Re(z)=0}\CZ\left( s+\frac{1}{2},\Phi,\beta_{z}(e_2,e_1^{\vee})   \right)\BW^{\psi^{-1}}(z,e_1)(\RI_2)\BW^{\psi}(-z,e_2^{\vee})(\RI_2)\frac{\ud z}{4\pi i}\\ 
	&+\frac{1}{2}\Res_{z=-s-\frac{1}{2}}\CZ\left( s+\frac{1}{2},\Phi,\beta_{z}(e_2,e_1^{\vee})   \right)\BW^{\psi^{-1}}\left(-s-\frac{1}{2},e_1\right)(\RI_2)\BW^{\psi}\left(s+\frac{1}{2},e_2^{\vee}\right)(\RI_2)\\
	&-\frac{1}{2}\Res_{z=s+\frac{1}{2}}\CZ\left( s+\frac{1}{2},\Phi,\beta_{z}(e_2,e_1^{\vee})   \right)\BW^{\psi^{-1}}\left(s+\frac{1}{2},e_1\right)(\RI_2)\BW^{\psi}\left(-s-\frac{1}{2},e_2^{\vee}\right)(\RI_2)\\
	&+\frac{1}{2}\Res_{z=\frac{1}{2}-s}\CZ\left( s+\frac{1}{2},\Phi,\beta_{z}(e_2,e_1^{\vee})   \right)\BW^{\psi^{-1}}\left(\frac{1}{2}-s,e_1\right)(\RI_2)\BW^{\psi}\left(s-\frac{1}{2},e_2^{\vee}\right)(\RI_2) \\
	&-\frac{1}{2}\Res_{z=s-\frac{1}{2}} \CZ\left(s+\frac{1}{2},\Phi,\beta_{z}(e_2,e_1^{\vee})\right)\BW^{\psi^{-1}}\left(s-\frac{1}{2},e_1\right)(\RI_2)\BW^{\psi}\left(\frac{1}{2}-s,e_2^{\vee} \right)(\RI_2). \\\end{align*}
And this can be extended to all $\Re(s)<-\frac{1}{2}$. To summarize, we have
\begin{prp}
	$\CI_{\Eis}(s)$ can be analytically extended to all $s\in\BC$. Moreover, for each $\chi,e_2,e_1$, the analytic continuation of 
	\[
	\CI_{\chi,e_1,e_2}(s)=\int_{\Re(z)=0}\CZ\left(s+\frac{1}{2},\Phi,\beta_{z}(e_2,e_1^{\vee})\right)\BW^{\psi^{-1}}(z,e_1)(1)\BW^{\psi}(-z,e_2^{\vee})(1)  \frac{\ud z}{4\pi i}, \;\Re(s)>\frac{1}{2}
	\]
	has the following expressions:
	\begin{itemize}
		\item $\displaystyle{0<\Re(s)\leq \frac{1}{2}}$:
		\begin{align*}
			&\CI_{\chi,e_1,e_2}=\int_{\Re(z)=\sigma}\CZ\left( s+\frac{1}{2},\Phi,\beta_{z}(e_2,e_1^{\vee})   \right)\BW^{\psi^{-1}}(z,e_1)(\RI_2)\BW^{\psi}(-z,e_2^{\vee})(\RI_2)\frac{\ud z}{4\pi i}\\
	&-\frac{1}{2}\Res_{z=s-\frac{1}{2}} \CZ\left(s+\frac{1}{2},\Phi,\beta_{z}(e_2,e_1^{\vee})\right) \BW^{\psi^{-1}}\left(s-\frac{1}{2},e_1\right)(\RI_2)\BW^{\psi}\left(\frac{1}{2}-s,e_2^{\vee} \right)(\RI_2),
		\end{align*}
		where $\displaystyle{0<\sigma<\frac{1}{2}}$ is such that  $\displaystyle{\frac{1}{2}-\sigma <\Re(s)}$.
		\item $\displaystyle{-\frac{1}{2}<\Re(s)\leq 0}$:
		\begin{align*}
			&\CI_{\chi,e_1,e_2}=\int_{\Re(z)=0}\CZ\left( s+\frac{1}{2},\Phi,\beta_{z}(e_2,e_1^{\vee})   \right)\BW^{\psi^{-1}}(z,e_1)(\RI_2)\BW^{\psi}(-z,e_2^{\vee})(\RI_2)\frac{\ud z}{4\pi i}\\
	&+\frac{1}{2}\Res_{z=\frac{1}{2}-s}\CZ\left( s+\frac{1}{2},\Phi,\beta_{z}(e_2,e_1^{\vee})   \right)\BW^{\psi^{-1}}\left(\frac{1}{2}-s,e_1 \right)(\RI_2)\BW^{\psi}\left(s-\frac{1}{2},e_2^{\vee}\right)(\RI_2) \\
	&-\frac{1}{2}\Res_{z=s-\frac{1}{2}} \CZ\left(s+\frac{1}{2},\Phi,\beta_{z}(e_2,e_1^{\vee})\right) \BW^{\psi^{-1}}\left(s-\frac{1}{2},e_1\right)(\RI_2)\BW^{\psi}\left(\frac{1}{2}-s,e_2^{\vee} \right)(\RI_2).\\
		\end{align*}
		\item $\displaystyle{-1<\Re(s)\leq -\frac{1}{2}}$:
		\begin{align*}
			&\CI_{\chi,e_1,e_2}=\int_{\Re(z)=\sigma}\CZ\left( s+\frac{1}{2},\Phi,\beta_{z}(e_2,e_1^{\vee})   \right)\BW^{\psi^{-1}}(z,e_1)(\RI_2)\BW^{\psi}(-z,e_2^{\vee})(\RI_2)\frac{\ud z}{4\pi i}\\ 
	&-\frac{1}{2}\Res_{z=s+\frac{1}{2}}\CZ\left( s+\frac{1}{2},\Phi,\beta_{z}(e_2,e_1^{\vee})   \right)\BW^{\psi^{-1}}\left(s+\frac{1}{2},e_1\right)(\RI_2)\BW^{\psi}\left(-s-\frac{1}{2},e_2^{\vee}\right)(\RI_2)\\
	&+\frac{1}{2}\Res_{z=\frac{1}{2}-s}\CZ\left( s+\frac{1}{2},\Phi,\beta_{z}(e_2,e_1^{\vee})   \right)\BW^{\psi^{-1}}\left(\frac{1}{2}-s,e_1)(\RI_2)\BW^{\psi}(s-\frac{1}{2},e_2^{\vee}\right)(\RI_2) \\
	&-\frac{1}{2}\Res_{z=s-\frac{1}{2}} \CZ\left(s+\frac{1}{2},\Phi,\beta_{z}(e_2,e_1^{\vee})\right) \BW^{\psi^{-1}}\left(s-\frac{1}{2},e_1\right)(\RI_2)\BW^{\psi}\left(\frac{1}{2}-s,e_2^{\vee} \right)(\RI_2),	\end{align*}
	where $\displaystyle{0<\sigma<\frac{1}{2}}$ is such that $\displaystyle{-\frac{1}{2}-\sigma}<\Re(s)$.
		\item $\Re(s)\leq -1$:
		\begin{align*}
			&\CI_{\chi,e_1,e_2}(s)=\int_{\Re(z)=0}\CZ\left( s+\frac{1}{2},\Phi,\beta_{z}(e_2,e_1^{\vee})   \right)\BW^{\psi^{-1}}(z,e_1)(\RI_2)\BW^{\psi}(-z,e_2^{\vee})(\RI_2)\frac{\ud z}{4\pi i}\\ 
	&+\frac{1}{2}\Res_{z=-s-\frac{1}{2}}\CZ\left( s+\frac{1}{2},\Phi,\beta_{z}(e_2,e_1^{\vee})   \right)\BW^{\psi^{-1}}\left(-s-\frac{1}{2},e_1\right)(\RI_2)\BW^{\psi}\left(s+\frac{1}{2},e_2^{\vee}\right)(\RI_2)\\
	&-\frac{1}{2}\Res_{z=s+\frac{1}{2}}\CZ\left( s+\frac{1}{2},\Phi,\beta_{z}(e_2,e_1^{\vee})   \right)\BW^{\psi^{-1}}\left(s+\frac{1}{2},e_1\right)(\RI_2)\BW^{\psi}\left(-s-\frac{1}{2},e_2^{\vee}\right)(\RI_2)\\
	&+\frac{1}{2}\Res_{z=\frac{1}{2}-s}\CZ\left( s+\frac{1}{2},\Phi,\beta_{z}(e_2,e_1^{\vee})   \right)\BW^{\psi^{-1}}\left(\frac{1}{2}-s,e_1\right)(\RI_2)\BW^{\psi}\left(s-\frac{1}{2},e_2^{\vee}\right)(\RI_2) \\
	&-\frac{1}{2}\Res_{z=s-\frac{1}{2}} \CZ\left(s+\frac{1}{2},\Phi,\beta_{z}(e_2,e_1^{\vee})\right)\BW^{\psi^{-1}}\left(s-\frac{1}{2},e_1\right)(\RI_2)\BW^{\psi}\left(\frac{1}{2}-s,e_2^{\vee} \right)(\RI_2). \end{align*}
	\end{itemize}
\end{prp}
We will conclude this section with a more careful study of the residues occurring in the above analytic continuation of $\CI_{\chi,e_1,e_2}(s)$. Let us start with the following definitions from \cite[Section 4]{Wu}:
\begin{dfn}
	For any Schwartz function $\Phi\in\CF(\Mat_2(\BA))$, and a unitary character $\chi$ of $\BR_+k^{\times}\backslash\BA^{\times}$, we define for $g_1,g_2\in\GL_2(\BA)$ and $\Re(s)\gg1$,
	\[
	f_{\Phi}(g_1,g_2;\chi|\cdot|^{2s+\frac{3}{2}},*):=\int_{\BA^{\times}}\BF_{\psi,2}\RL(g_1)\RR(g_2)\Phi\left(\begin{pmatrix}
		t & 0 \\ 0  & 0
	\end{pmatrix} \right)\chi(t)|t|^{2s+1}\ud^{\times}t\cdot |\det g_1^{-1}g_2|^{s+1},
	\]
	and
	\[
	f_{\Phi}(g_1,g_2;\chi|\cdot|^{2s+\frac{3}{2}},\star):=\int_{\BA^{\times}}\BF_{\psi,2}\RL(g_1)\RR(g_2)\Phi\left(  \begin{pmatrix}
		0 & 0 \\ 0 & t
	\end{pmatrix}  \right)\chi(t)|t|^{2s+1}\ud^{\times}t\cdot|\det g_1^{-1}g_2|^{s+1}.
	\]
	Both of them can be analytically extended to $s\in\BC$. Note that $f_{\Phi}(g_1,g_2;\chi|\cdot|^{2s+\frac{3}{2}},*)$ defines a vector in the tensor product $\pi(\chi|\cdot|^{s+\frac{1}{2}},|\cdot|^{-s-\frac{1}{2}})\otimes\pi(\chi^{-1}|\cdot|^{-s-\frac{1}{2}},|\cdot|^{s+\frac{1}{2}})$. We denote by $(\CM\times f)_{\Phi}(g_1,g_2;\chi|\cdot|^{2s+\frac{3}{2}},*)$ the image under the intertwining operator with respect to the first variable. Similarly, $f_{\Phi}(g_1,g_2;\chi|\cdot|^{2s+\frac{3}{2}},\star)\in\pi(|\cdot|^{-s-\frac{1}{2}},\chi|\cdot|^{s+\frac{1}{2}})\otimes\pi(|\cdot|^{s+\frac{1}{2}},\chi^{-1}|\cdot|^{-s-\frac{1}{2}})$, and we denote $(f\times\CM)_{\Phi}(g_1,g_2;\chi|\cdot|^{2s+\frac{3}{2}},\star)$ its image under the intertwining operator with respect to the second variable.
	\end{dfn}

\begin{lem}\label{lem_mf=fm}
	We have
	\[
	(\CM\times f)_{\Phi}(g_1,g_2;\chi|\cdot|^{2s+\frac{3}{2}},*)=(f\times\CM)_{\Phi}(g_1,g_2;\chi|\cdot|^{2s+\frac{3}{2}},\star).
	\]
\end{lem}
\begin{proof}
	It suffices to prove the equation for $g_1=g_2=\RI_2$ since the general case follows from this special case by taking $\RL(g_1)\RR(g_2)\Phi$ instead of $\Phi$. By definition, when $\Re(s)\gg1$, we have
	\begin{align*}
		&(\CM\times f)_{\Phi}(\RI_2,\RI_2;\chi|\cdot|^{2s+\frac{3}{2}},*)\\
		&=\int_{\BA}\int_{\BA^{\times}} \int_{\BA}\Phi\left(  \begin{pmatrix}
			1 & -n \\ 0 & 1
		\end{pmatrix} w^{-1} \begin{pmatrix}
			t & x \\ 0 & 0 
		\end{pmatrix} \right)\ud x \chi(t)|t|^{2s+1}\ud^{\times}t\ud n\\
		&=\int_{\BA}\int_{\BA^{\times}}\int_{\BA}\Phi\left(   \begin{pmatrix}
			0 & -nt \\ 0 & t
		\end{pmatrix} w\begin{pmatrix}
			1 & xt^{-1} \\ 0 & 1
		\end{pmatrix}    \right)\ud x \chi(t)|t|^{2s+1}\ud^{\times}t\ud n\\
		&=(f\times\CM)_{\Phi}(\RI_2,\RI_2;\chi|\cdot|^{2s+\frac{3}{2}},\star).
	\end{align*}
\end{proof}

\begin{prp}\label{prp_res}
	When $\Re(s)\ll1$, denote
	\[
	\widehat{\CI}_{\rb}(s):=\int_{\BA^{\times}}\int_{\BA^2}\widehat{\Phi}\left( \begin{pmatrix}
	1 & n_1 \\ 0 & 1
\end{pmatrix} \begin{pmatrix}
	0  & 0 \\z & 0 
\end{pmatrix}\begin{pmatrix}
	1 & n_2 \\ 0 & 1 
\end{pmatrix}  \right)\psi(n_1+n_2)\omega(z)|z|^{2-2s} \ud^{\times}z
	\]
	dual to $\CI(s)$, then we have
	\begin{itemize}
		\item [(1)]When $\Re(s)\gg 1$ and $\chi=\omega$, \begin{align*}
		\sum_{e_1,e_2\in\CB(\omega,1)} \Res_{z=s+\frac{1}{2}}   \CZ\left(s+\frac{1}{2},\Phi,\beta_z(e_2,e_1^{\vee}) \right)\BW^{\psi^{-1}}\left(s+\frac{1}{2},e_1\right)(\RI_2)\BW^{\psi}&\left( -s-\frac{1}{2},e_2^{\vee} \right)(\RI_2)\\
&=\CI_{\rb}(s).
	\end{align*}
	\item [(2)] When $\Re(s)\ll1$ and $\chi=\omega$, \begin{align*}
		\sum_{e_1,e_2\in\CB(\omega,1)}\Res_{z=s-\frac{1}{2}}\CZ\left(s+\frac{1}{2},\Phi,\beta_z(e_2,e_1^{\vee})\right)\BW^{\psi^{-1}}\left(s-\frac{1}{2},e_1\right)(\RI_2)\BW^{\psi}&\left(-s+\frac{1}{2},e_2^{\vee} \right)(\RI_2)\\
		&=-\widehat{\CI}_{\rb}(s).
	\end{align*}
	\item [(3)]
When $\Re(s)\gg1 $ and $\chi=1$, we have
\begin{align*}
	\sum_{e_1,e_2\in\CB(1,\omega^{-1})}\Res_{z=-s-\frac{1}{2}}\CZ\left( s+\frac{1}{2},\Phi,\beta_z(e_2,e_1^{\vee})  \right)\BW^{\psi^{-1}}\left(-s-\frac{1}{2} \right)(\RI_2)\BW^{\psi}&\left(s+\frac{1}{2} \right)(\RI_2)\\
	&=-\CI_{\rb}(s).
\end{align*}
\item [(4)] When $\Re(s)\ll 1$ and $\chi=1$, 
\begin{align*}
		\sum_{e_1,e_2\in\CB(\omega,1)}\Res_{z=-s+\frac{1}{2}}\CZ\left(s+\frac{1}{2},\Phi,\beta_z(e_2,e_1^{\vee})\right)\BW^{\psi^{-1}}\left(-s+\frac{1}{2},e_1\right)(\RI_2)\BW^{\psi}&\left(s-\frac{1}{2},e_2^{\vee} \right)(\RI_2)\\
		&=\widehat{\CI}_{\rb}(s).
		\end{align*}
		\end{itemize}
\end{prp}

\begin{proof}
 We start with $(1)$. When $\Re(s)\gg1 $, from \cite{Tat67} we know
\begin{align*}
	\Res_{z=s+\frac{1}{2}}\CZ\left( s+\frac{1}{2},\Phi,\beta_z(e_2,e_1^{\vee})    \right)&=\int_{\BA^{\times}} \BF_{\psi,2}(_{e_1^{\vee}}\Phi_{e_2})\left(\begin{pmatrix}
		t & 0\\ 0 & 0 
	\end{pmatrix} \right)\omega(t)|t|^{2s+1}\ud^{\times}t\\
	&=f_{_{e_1^{\vee}}\Phi_{e_2}}(\RI_2,\RI_2;\omega|\cdot|^{2s+\frac{3}{2}},*).
\end{align*}
Consider the decomposition
\begin{align*}
	&f_{\Phi}(g_1,g_2;\omega|\cdot|^{2s+\frac{3}{2}},*)\\
	&=\sum_{e_1\in\CB(\chi,1)}\sum_{e_2\in\CB(\chi,1)}\int_{\RK^2}f_{\Phi}(k_1,k_2)\ud k_1\ud k_2 e_1^{\vee}(k_1)e_2(k_2)e_1\left(s+\frac{1}{2}\right)(g_1)e_2^{\vee}\left(-s-\frac{1}{2}\right)(g_2).
\end{align*}
Note that
\begin{align*}
	&\int_{\RK^2}f_{\Phi}(k_1,k_2)e_1^{\vee}(k_1)e_2(k_2)\ud k_1\ud k_2\\
	&=\int_{\RK^2}\int_{\BA^{\times}}\BF_{\psi,2}\RL(k_1)\RL(k_2)\Phi\left(  \begin{pmatrix}
		t & 0 \\ 0  & 0
	\end{pmatrix}  \right)\omega(t)|t|^{2s+1}\ud^{\times}t e_1^{\vee}(k_1)e_2(k_2)   \ud k_1\ud k_2\\
	&=\int_{\BA^{\times}}\int_{\RK^2}\BF_{\psi,2}\RL(k_1)\RR(k_2)\Phi\left(\begin{pmatrix}
		t & 0 \\ 0  & 0
	\end{pmatrix}\right)e_1^{\vee}(k_1)e_2(k_2)\ud k_1\ud k_2\omega(t)|t|^{2s+1} \ud^{\times}t\\
	&=\int_{\BA^{\times}}\BF_{\psi,2}(_{e_1^{\vee}}\Phi_{e_2})\left(\begin{pmatrix}
		t & 0 \\ 0  & 0
	\end{pmatrix} \right)\omega(t)|t|^{2s+1}\ud^{\times}t,
\end{align*}
we have the Whittaker function of $f_{\Phi}(g_1,g_2;\omega|\cdot|^{2s+\frac{3}{2}},*)$ at $(\RI_2,\RI_2)$ is
\begin{align*}
	\sum_{e_1,e_2\in\CB(\chi,1)}f_{_{e_1^{\vee}}\Phi_{e_2}}(\RI_2,\RI_2;\omega|\cdot|^{2s+\frac{3}{2}},*)\BW^{\psi^{-1}}\left(s+\frac{1}{2},e_1\right)(\RI_2)\BW^{\psi}\left(-s-\frac{1}{2},e_2^{\vee} \right)(\RI_2).
\end{align*}
On the other hand, according to \cite[Lemma 4.9]{Wu}, this Whittaker function is given by the integral
\begin{align*}
	\int_{\BA^{\times}}\int_{\BA^2}\Phi\left(   \begin{pmatrix}
		1 & n_1 \\  0& 1
	\end{pmatrix} \begin{pmatrix}
		 0  & 0
 \\z &  0	\end{pmatrix} \begin{pmatrix}
 	1 &n_2 \\ 0& 1
 \end{pmatrix} \right)\psi^{-1}(n_1+n_2)\ud n_1\ud n_2\omega(z)|z|^{2s+2}\ud^{\times}z,
\end{align*}
which is exactly $\CI_{\rb}(s)$.

For $(3)$,the residue is given by the following integral when $\Re(s)\gg 1$
\begin{align*}
	\Res_{z=-s-\frac{1}{2}}\CZ\left( s+\frac{1}{2},\Phi,\beta_z(e_2,e_1^{
	\vee}) \right)&=-\int_{\BA^{\times}} \BF_{\psi,2}( {_{e_1^{\vee}\Phi_{e_2}})}\left(  \begin{pmatrix}
		 0 & 0 \\ 0& t
	\end{pmatrix}  \right)|t|^{2s+1}\omega(t)\ud^{\times}t\\
	&=-f_{ {_{e_1^{\vee}}\Phi_{e_2}}   }(\RI_2,\RI_2;\omega|\cdot|^{2s+\frac{3}{2}},\star).
\end{align*}
Similarly, we have
\begin{align*}
	\sum_{e_1,e_2\in\CB(1,\omega^{-1})} f_{\Phi}(\RI_2,\RI_2;\omega|\cdot|^{2s+\frac{3}{2}},\star)\BW^{\psi^{-1}}\left(-s-\frac{1}{2}\right)(\RI_2)\BW^{\psi}\left(s+\frac{1}{2}\right)(\RI_2)
\end{align*}
is the Whittaker function of $f_{\Phi}(g_1,g_2;\omega|\cdot|^{2s+\frac{3}{2}},\star)$ at $(\RI_2,\RI_2)$. But according to Lemma \ref{lem_mf=fm} and \cite{GJ77}, the intertwining operator does not change the Eisenstein series and the Whittaker functions. So the value of the Whittaker function of $f_{\Phi}(g_1,g_2;\omega|\cdot|^{2s+\frac{3}{2}},\star)$ at $(\RI_2,\RI_2)$ is also the value of the Whittaker function of $f_{\Phi}(g_1,g_2;\omega|\cdot|^{2s+\frac{3}{2}},*)$, which is $\CI_{\rb}(s)$ according to part $(1)$.

For $(2)$, when $\Re(s)\ll 1$, using the functional equation of $\GL_1$ zeta integrals, we have
\begin{align*}
	&\Res_{z=s-\frac{1}{2}}\CZ\left( s+\frac{1}{2},\Phi,\beta_z(e_2,e_1^{\vee}    \right)\\
	&=-\int_{\BA^{\times}}\BF_{\psi,1}\BF_{\psi,4}\BF_{\psi,2}({ _{e_1^{\vee}}\Phi_{e_2}  })\left(\begin{pmatrix}
		t & 0 \\ 0 & 0 
	\end{pmatrix} \right)\omega(t)|t|^{1-2s}\ud^{\times}t.
\end{align*}
Note that for any $g_1,g_2\in\GL_2(\BA)$, we have
\[
\BF_{\psi,2}\RL(g_1)\RR(g_2)\widehat{\Phi}=\BF_{\psi,2}\widehat{\RL(g_2)\RR(g_1)\Phi}|\det g_1^{-1}g_2|^2.
\]
Using the Fourier inversion formula, we have
\[
\BF_{\psi,2}\widehat{\RL(g_2)\RR(g_1)\Phi}\left( \begin{pmatrix}
	x_{11} & x_{12} \\ x_{21} & x_{22}
\end{pmatrix} \right)=\BF_{\psi,2}\BF_{\psi,1}\BF_{\psi,4}\RL(g_2)\RR(g_1)\Phi\left(\begin{pmatrix}
	x_{11} & x_{12} \\ -x_{21} & x_{22} 
\end{pmatrix} \right).
\]
Then, according to a similar argument, we have
\begin{align*}
&\sum_{e_1,e_2\in\CB(\omega,1)}\Res_{z=s-\frac{1}{2}}\CZ\left(s+\frac{1}{2},\Phi,\beta_z(e_2,e_1^{\vee})\right)\BW^{\psi^{-1}}\left(s-\frac{1}{2},e_1\right)(\RI_2)\BW^{\psi}\left(-s+\frac{1}{2},e_2^{\vee} \right)(\RI_2)\\
&=\int_{\BA^{\times}}\int_{\BA^2}\widehat{\Phi}\left( \begin{pmatrix}
	1 & n_1 \\ 0 & 1
\end{pmatrix} \begin{pmatrix}
	 0 & 0 \\z & 0 
\end{pmatrix}\begin{pmatrix}
	1 & n_2 \\ 0 & 1 
\end{pmatrix}  \right)\psi(n_1+n_2)\omega(z)|z|^{2-2s} \ud^{\times}z.
\end{align*}
Finally $(4)$ is a combination of the above $(1)$, $(2)$ and $(3)$, and we leave the verification to the readers.
\end{proof}

\begin{rmk}
	From the above Proposition, the residues are noting but the Whittaker function of the kernel function given by the summation over the rank $1$ rational matrices. In this sense, the boundary term appears as the residue of the Eisenstein series in the spectrum decomposition of the kernel function via analytic continuation.
\end{rmk}

\subsection{Analytic Continuation of $\CI_{\rb}(s)$}

In this section, we will study the analytic continuation of the boundary term $\CI_{\rb}(s)$. A first observation is that when $\Re(s)\gg 1$,
\begin{align*}
	&\int_{[\RN]}\RK^{\prime}(x,n)\psi^{-1}(n)\ud n=\int_{\RN(\BA)} \sum_{\gamma\in \RB(k)\backslash \GL_2(k)}\RR(\gamma x,wn)\psi^{-1}(n)\ud n\\
	&=\int_{\RN(\BA)} \sum_{\gamma\in \RB(k)\backslash\GL_2(k)} \int_{\BA^{\times}}\Phi\left(x^{-1}\gamma^{-1}\begin{pmatrix}
		0 & z \\0 & 0 
	\end{pmatrix}w n\right)\psi^{-1}(n) \omega(z)|z|^{2s+2}\ud^{\times}z |\det x^{-1}|^{s+1}   \ud n \\
	&=\sum_{\gamma\in \RB(k)\backslash\GL_2(k)}\int_{\BA} \int_{\BA^{\times}} \RL_{\gamma x}\Phi \left(  \begin{pmatrix}
		z & zn \\ 0 & 0
	\end{pmatrix} \right) \omega(z)|z|^{2s+2}\ud^{\times}z|\det x^{-1}|^{s+1}   \psi^{-1}(n)\ud n\\
	&=\sum_{\gamma\in \RB(k)\backslash\GL_2(k)} \int_{\BA^{\times}}  \BF_{\psi,2}\RL_{\gamma x}\Phi\left(  \begin{pmatrix}
		z & z^{-1} \\ 0 & 0
	\end{pmatrix}  \right)\omega(z)|z|^{2s+1}\ud^{\times }z|\det x^{-1}|^{s+1}.
\end{align*}
Note that the inner integral is convergent for all $s\in\BC$. Now consider the section
\[
\GL_2(\BA)\times\BC\rightarrow\BC:(g,s)\mapsto \Ff_s(g):=\int_{\BA^{\times}}  \BF_{\psi,2}\RL_{g}\Phi\left(  \begin{pmatrix}
		z & z^{-1} \\ 0 & 0
	\end{pmatrix}  \right)\omega(z)|z|^{2s+1}\ud^{\times }z|\det g^{-1}|^{s+1}.
\]
This is a section to the induced principal series for all $s\in\BC$, and let
\[
\RE(\Ff_s)(g):=\sum_{\gamma\in \RB(k)\backslash\GL_2(k)}\Ff_s(\gamma g)
\]
be the corresponding Eisenstein series. When $\Re(s)\gg1$, we have 
\[
\int_{[\RN]}\RK^{\prime}(n_1,n_2)\psi^{-1}(n_1^{-1}n_2)\ud n_1\ud n_2=\int_{[\RN]}\RE(\Ff_s)(n)\psi(n)\ud n=\int_{N(\BA)}\Ff_s(wn)\psi(n)\ud n.
\]
Then, the analytic continuation comes from the general theory of the Eisenstein series.

Note that we have
\begin{align*}
	\BF_{\psi,2}\RL_{x}\Phi \left( \begin{pmatrix}
		x_{11} & x_{12} \\ x_{21} & x_{22}
	\end{pmatrix} \right)&=\int_{\BA}\RL_{x}\Phi\left( \begin{pmatrix}
		x_{11} & y \\ x_{21} & x_{22}
	\end{pmatrix}\right)\psi^{-1}(yx_{22})\ud y.
\end{align*}
Then for $x= w\begin{pmatrix}
	1 & n \\ 0 & 1
\end{pmatrix}\in\GL_2(\BA)$, we have
\begin{align*}
	\BF_{\psi,2}\RL_{x}\Phi\left(\begin{pmatrix}
		z & z^{-1} \\0 & 0
	\end{pmatrix} \right)=\int_{\BA}\Phi\left( \begin{pmatrix}
		-nz & -ny \\ z & y
	\end{pmatrix}  \right)\psi^{-1}(yz^{-1})\ud y.
\end{align*}
Then integral over $n$ and make a change of variable by $y\mapsto zn_1, n\mapsto -n_2$, we have
\begin{align*}
		&\int_{[\RN]\times[\RN]}\RK^{\prime}(n_1,n_2)\psi^{-1}(n_1^{-1}n_2)\ud n_1\ud n_2\\
		&=\int_{\BA^2}\int_{\BA^{\times}}\Phi\left( \begin{pmatrix}
			zn_2 & zn_1n_2 \\ z & zn_1
		\end{pmatrix}  \right)\psi^{-1}(n_1+n_2)\omega(z)|z|^{2s+2}\ud^{\times }z\ud n_1\ud n_2.
\end{align*}
This also coincides with the calculation in Section \ref{ssec_inproof}.

Note that from here we can see that $\CI_{\rb}(s)$ admits an Euler product
\begin{align*}
	\prod_{\nu}\int_{k_{\nu}^{2}}\int_{k_{\nu}^{\times}}\Phi_{\nu}\left( \begin{pmatrix}
			zn_2 & zn_1n_2 \\ z & zn_1
		\end{pmatrix}  \right)\psi_{\nu}^{-1}(n_1+n_2)\omega_{\nu}(z)|z|^{2s+2}\ud^{\times }z\ud n_1\ud n_2.
\end{align*}
Let us finish this section with the unramified calculation.

\begin{prp}
	Let $\Phi_{\nu}=1_{\Mat_2(\Fo_{\nu})}$, and $\psi_{\nu}$ is unramified, then we have
	\begin{align*}
	\int_{k_{\nu}^{2}}\int_{k_{\nu}^{\times}}\Phi_{\nu}\left( \begin{pmatrix}
			zn_2 & zn_1n_2 \\ z & zn_1
		\end{pmatrix}  \right)\psi_{\nu}^{-1}(n_1+n_2)\omega_{\nu}(z)|z|^{2s+2}\ud^{\times }z\ud n_1\ud n_2=1-\omega_{\nu}(\varpi_{\nu})q^{-(2s+2)}
\end{align*}
if we normalize the measure $\ud^{\times}z$ such that $\vol(\Fo_{\nu}^{\times},\ud^{\times}z)=1$. So let $S$ be a finite set of places of $k$ containing all the Archimedean places and such that for all $v\notin S$ we have $\Phi_{\nu}=1_{\Mat_2(\Fo_{\nu})}$, $\psi_{\nu}$ is unramified and $\vol(\Fo_{\nu}^{\times},\ud^{\times}z)=1$, then we have
\[
\prod_{\nu\notin S}\int_{k_{\nu}^{2}}\int_{k_{\nu}^{\times}}\Phi_{\nu}\left( \begin{pmatrix}
			zn_2 & zn_1n_2 \\ z & zn_1
		\end{pmatrix}  \right)\psi_{\nu}^{-1}(n_1+n_2)\omega_{\nu}(z)|z|^{2s+2}\ud^{\times }z\ud n_1\ud n_2=\frac{1}{L^S(\omega,2s+2)},
\]
where $L^S(\omega,s)$ is the partial Dirichlet $L$-function.
\end{prp}

\begin{proof}
	A first observation is that 
	\begin{align*}
		&\int_{k_{\nu}^{2}}\int_{k_{\nu}^{\times}}\Phi_{\nu}\left( \begin{pmatrix}
			zn_2 & zn_1n_2 \\ z & zn_1
		\end{pmatrix}  \right)\psi^{-1}(n_1+n_2)\omega(z)|z|^{2s+2}\ud^{\times }z\ud n_1\ud n_2\\
		&=\sum_{m=0}^{\infty}\int_{\varpi_{\nu}^m\Fo_{\nu}^{\times}}\omega_{\nu}(\varpi_{\nu})^m q_{\nu}^{-m(2s+2)}\int_{|zn_1|\leq 1,|zn_2|\leq 1,|zn_1n_2|\leq 1}\psi_{\nu}^{-1}(n_1+n_2)\ud n_1\ud n_2 \ud^{\times}z.
	\end{align*}
We first fix $z$ such that $z\in\varpi_{\nu}^m\Fo_{\nu}^{\times}$, and then $n_1\in\varpi^{m_1}_{\nu}\Fo_{\nu}^{\times}$ with $m+m_1\geq 0$. Then the condition for $n_2$ is
\[
|n_2|\leq \min\{   q^{m},q^{m+m_1}  \}.
\]
Therefore, the inner integral becomes
\begin{align*}
	&\int_{|zn_1|\leq 1,|zn_2|\leq 1,|zn_1n_2|\leq 1}\psi_{\nu}^{-1}(n_1+n_2)\ud n_1\ud n_2\\
	&=\sum_{m_1=-m}^{0} \int_{n_1\in \varpi_{\nu}^{m_1}\Fo_{\nu}^{\times} }\psi_{\nu}^{-1}(n_1)\ud n_1 \int_{|n_2|\leq q^{m+m_1} }\psi_{\nu}^{-1}(n_2)\ud n_2 \\
	&+\sum_{m_1=1}^{\infty} \int_{n_1\in\varpi^{m_1}_{\nu}\Fo_{\nu}^{\times}}\psi^{-1}_{\nu}(n_1)\ud n_1 \int_{|n_2|\leq q^m}\psi_{\nu}^{-1}(n_2)\ud n_2\\
	&=\int_{n_1\in \varpi_{\nu}^{-m}\Fo_{\nu}^{\times} }\psi_{\nu}^{-1}(n_1)\ud n_1 + q^{-1}\int_{|n_2|\leq q^m}\psi_{\nu}^{-1}(n_2)\ud n_2\\
	&= \left\{ \begin{array}{rcl}
       1 & \mbox{if}
        & m=0 \\  -1 & \mbox{if} & m=1  \\
       0 & \mbox{if} & m\geq 2
                \end{array}\right..
\end{align*}
Hence
\begin{align*}
	&\int_{k_{\nu}^{2}}\int_{k_{\nu}^{\times}}\Phi_{\nu}\left( \begin{pmatrix}
			zn_2 & zn_1n_2 \\ z & zn_1
		\end{pmatrix}  \right)\psi^{-1}(n_1+n_2)\omega(z)|z|^{2s+2}\ud^{\times }z\ud n_1\ud n_2\\
		&=1-\omega_{\nu}(\varpi_{\nu})q^{-(2s+2)}.
\end{align*}
\end{proof}

\begin{cor}
	The cuspidal part $\CI_{\cusp}(s)$ admits an analytic continuation to all $s\in\BC$.
\end{cor}

\section{Functional Equations of in the Trace Formula}\label{sec_feq}
In this section, we will explain that the Theorem \ref{thm_psf} is responsible for the functional equation of standard $L$-functions of $\GL_2$. Let $\Phi$ be a Schwartz function on $\Mat_2(\BA)$ as before and $\widehat{\Phi}$ its Fourier transform. Denote for $x,y\in\Mat_2^{\vee}(\BA)$
\[
\widehat{\RK}(x, y):= \int_{[\RZ]}  \sum_{\gamma\in\GL_2(k)}\widehat{\Phi}(zy^{-1}\gamma x)\omega^{-1}(z)|z|^{-2s+2}\ud^{\times }z\cdot|\det y^{-1}x|^{-s+1},
\]
and
\begin{align*}
		\widehat{\RK}^{\prime}(x,y):=\int_{[\RZ]}\sum_{\gamma \in\Mat_{2}(k),\rank \gamma=1}\widehat{\Phi}(zy^{-1}\gamma x)\omega^{-1}(z)|z|^{-2s+2}\ud^{\times}z\cdot|\det y^{-1}x|^{-s+1}.
\end{align*}

 Let 
\[
\widehat{\CI}(s):=\int_{[\RZ]} \mathrm{KTF}(\RR_z(\widehat{\Phi}(x)(\ud^+x)^{\frac{1}{2}}  ))\omega(z)|z|^{2s}\ud^{\times}z,
\]
and
\[
\widehat{\CI}_{\cusp}(s):=\sum_{\pi\;\mathrm{cuspidal},\omega_{\pi}=\omega} \sum_{\varphi_1.\varphi_2\in\CB(\pi)}\CZ\left(-s+\frac{1}{2},\widehat{\Phi},\beta^{\vee}(\varphi_2,\varphi_1^{\vee})\right)\BW_{\varphi_1}^{\psi^{-1}}(\RI_2)\BW_{\varphi_2^{\vee}}^{\psi}(\RI_2).
\]
It is also easy to see that $\widehat{\CI}_{\rb}(s)$ defined in Proposition \ref{prp_res} is
\[
\widehat{\CI}_{\rb}(s)=\int_{[\RN]\times[\RN]}\widehat{\RK}^{\prime}(n_1,n_2)\psi^{-1}(n_1^{-1}n_2)\ud n_1\ud n_2.
\]

By identifying $\Mat_2^{\vee}$ with $\Mat_2$ and replace $\Phi$ by $\widehat{\Phi}$ in Proposition \ref{prp_acofI}, we see $\widehat{\CI}(s), \widehat{\CI}_{\Eis}(s)$ and $\widehat{\CI}_{\cusp}(s)$ also admit meromorphic continuations to all $s\in\BC$. Then according to Theorem \ref{thm_psf} again, we have
\begin{prp}\label{prp_fqI}
	$\CI(s)=\widehat{\CI}(s)$ for all $s\in\BC$.
\end{prp}

\begin{proof}
	Because both of them are
	\begin{align*}
		\int_{z\in [\RZ],|z|\geq 1}\mathrm{KTF}(\RR_z(\Phi(x)(\ud^+x)^{\frac{1}{2}}  ))\omega(z)|z|^{2s}\ud^{\times}z+\int_{|z|<1}\mathrm{KTF}(\RR_z(\widehat{\Phi}(x)(\ud^+x)^{\frac{1}{2}}  ))\omega(z)|z|^{2s}\ud^{\times}z
	\end{align*}
	as in the proof of Proposition \ref{prp_acofI}.
\end{proof}

According to \cite[Theorem 2.10]{Wu} by considering the action of $g\mapsto \widehat{\Phi}(g^{-1})|\det g|^{s-1}$ on $(\RR_{\omega},\CL^2(\GL_2,\omega))$, we have the spectrum decomposition of $\widehat{\RK}(x,y)$:

\begin{align*}
	\widehat{\RK}(x,y)&=\sum_{\pi\;\mathrm{cuspidal},\omega_{\pi}=\omega} \sum_{\varphi_1.\varphi_2\in\CB(\pi)}\CZ\left(-s+\frac{1}{2},\widehat{\Phi},\beta^{\vee}(\varphi_2,\varphi_1^{\vee})\right)\varphi_1(x)\varphi_2^{\vee}(y)\\
	&+\sum_{\chi\in\widehat{\BR_+k^{\times}\backslash\BA^{\times}}}\int_{-\infty}^{\infty} \sum_{e_1,e_2\in\CB(\chi,\omega\chi^{-1})}\CZ\left(-s+\frac{1}{2},\widehat{\Phi},\beta^{\vee}_{i\tau}(e_2,e_1^{\vee})\right)\RE(i\tau,e_1)(x)\RE(-i\tau,e_2^{\vee})(y)\frac{\ud \tau}{4\pi}\\
	&+\frac{1}{\vol[\PGL_2]}\sum_{\eta\in\widehat{k^{\times}\backslash\BA^{\times}},\eta^2=\omega}\left(\Phi(g)\eta(\det g)|\det g|^{-s+1} \ud g\right)\overline{\eta(\det x)}\eta(\det y),
\end{align*}
and integrate it against the same character over $[\RN]\times[\RN]$, we obtain
\begin{align*}
	&\int_{[\RN]\times[\RN]}\widehat{\RK}(n_1,n_2)\psi^{-1}(n_1^{-1}n_2)\ud n_1\ud n_2\\
	&=\sum_{\pi\;\mathrm{cuspidal},\omega_{\pi}=\omega} \sum_{\varphi_1.\varphi_2\in\CB(\pi)}\CZ\left(-s+\frac{1}{2},\widehat{\Phi},\beta^{\vee}(\varphi_2,\varphi_1^{\vee})\right)\BW_{\varphi_1}^{\psi^{-1}}(\RI_2)\BW_{\varphi_2^{\vee}}^{\psi}(\RI_2)\\
	&+\sum_{\chi\in\widehat{\BR_+k^{\times}\backslash\BA^{\times}}}\int_{-\infty}^{\infty} \sum_{e_1,e_2\in\CB(\chi,\omega\chi^{-1})}\CZ\left(-s+\frac{1}{2},\widehat{\Phi},\beta_{i\tau}(e_2,e_1^{\vee})\right)\BW^{\psi^{-1}}(i\tau,e_1)(\RI_2)\BW^{\psi}(-i\tau,e_2^{\vee})(\RI_2)  \frac{\ud \tau}{4\pi}.
\end{align*}

Then when $\Re(s)<-1$ we have
\begin{align*}
	&\widehat{\CI}(s)=\int_{Z(k)\backslash Z(\BA)}\mathrm{KTF}(\RR_z(\widehat{\Phi}(x)(\ud^+x)^{\frac{1}{2}}  ))\omega(z)|z|^{2s}\ud^{\times}z \\
	&=\int_{[\RN]\times[\RN]}\left(\widehat{\RK}(n_1,n_2)+\widehat{\RK}^{\prime}(n_1,n_2)\right)\psi^{-1}(n_1^{-1}n_2)\ud n_2\ud n_2\\
    &=\sum_{\pi\;\mathrm{cuspidal},\omega_{\pi}=\omega} \sum_{\varphi_1.\varphi_2\in\CB(\pi)}\CZ\left(-s+\frac{1}{2},\widehat{\Phi},\beta^{\vee}(\varphi_2,\varphi_1^{\vee})\right)\BW_{\varphi_1}^{\psi^{-1}}(\RI_2)\BW_{\varphi_2^{\vee}}^{\psi}(\RI_2)\\
	&+\sum_{\chi\in\widehat{\BR_+k^{\times}\backslash\BA^{\times}}}\int_{-\infty}^{\infty} \sum_{e_1,e_2\in\CB(\chi,\omega\chi^{-1})}\CZ\left(-s+\frac{1}{2},\widehat{\Phi},\beta^{\vee}_{i\tau}(e_2,e_1^{\vee})\right)\BW^{\psi^{-1}}(i\tau,e_1)(\RI_2)\BW^{\psi}(-i\tau,e_2^{\vee})(\RI_2)  \frac{\ud \tau}{4\pi}\\
	&+\int_{[\RN]\times[\RN]} \widehat{\RK}^{\prime}(n_1,n_2)\psi^{-1}(n_1^{-1}n_2)\ud n_1\ud n_2.
	\end{align*}
Because
\begin{align*}
	&\CZ\left(s+\frac{1}{2} ,\Phi,\beta_{i\tau}(e_2,e_1^{\vee}) \right)\\
	&=\int_{(\BA^{\times})^2}\BF_{\psi,2     }({  _{e_1^{\vee}}}\Phi_{e_2})\begin{pmatrix}
		t_1 & 0 \\0 & t_2
	\end{pmatrix}\chi(t_1)|t_1|^{s+\frac{1}{2}+i\tau}\chi^{-1}(t_2)|t_2|^{s+\frac{1}{2}-i\tau}\ud^{\times}t_1\ud^{\times}t_2,
\end{align*}
and 
\begin{align*}
	\BF_{\psi,2}( ({  _{e_1^{\vee}}}\widehat{\Phi}_{e_2}) )=\BF_{\psi,2}(\widehat{ ({  _{e_2}}\Phi_{e_1^{\vee}})})=\BF_{\psi,1}\BF_{\psi,4}\BF_{\psi_,2}({_{e_2}\Phi_{e_1^{\vee}}})\begin{pmatrix}
		\cdot & \cdot \\ -\cdot & \cdot
	\end{pmatrix},
\end{align*}
we can apply the functional equation of $\GL_1$ zeta integrals to obtain
\begin{align*}
	&\CZ\left(-s+\frac{1}{2},\widehat{\Phi},\beta^{\vee}_{i\tau}(e_2,e_1^{\vee})    \right)\\
	&=\int_{(\BA^{\times})^2}\BF_{\psi,1}\BF_{\psi,4}  \BF_{\psi,2     }({  _{e_2}}\Phi_{e_1^{\vee}})\begin{pmatrix}
		t_1 & 0 \\ 0 & t_2
	\end{pmatrix}\chi(t_1)|t_1|^{-s+\frac{1}{2}+i\tau}\chi^{-1}(t_2)|t_2|^{-s+\frac{1}{2}-i\tau}\ud^{\times}t_1\ud^{\times}t_2\\
	&=\int_{(\BA^{\times})^2} \BF_{\psi,2     }({  _{e_2}}\Phi_{e_1^{\vee}})\begin{pmatrix}
		t_1 & 0 \\ 0 & t_2
	\end{pmatrix}\chi^{-1}(t_1)|t_1|^{s+\frac{1}{2}-i\tau}\chi(t_2)|t_2|^{s+\frac{1}{2}+i\tau}\ud^{\times}t_1\ud^{\times}t_2\\
	&=\CZ\left(s+\frac{1}{2} ,\Phi,   \beta_{-i\tau}(e_1^{\vee},e_2) \right).
\end{align*}
Therefore
\begin{align*}
	&\sum_{\chi\in\widehat{\BR_+k^{\times}\backslash\BA^{\times}}}\int_{-\infty}^{\infty} \sum_{e_1,e_2\in\CB(\chi,\omega\chi^{-1})}\CZ\left(-s+\frac{1}{2},\widehat{\Phi},\beta^{\vee}_{i\tau}(e_2,e_1^{\vee})\right)\BW^{\psi^{-1}}(i\tau,e_1)(\RI_2)\BW^{\psi}(-i\tau,e_2^{\vee})(\RI_2)  \frac{\ud \tau}{4\pi}\\
	&=\sum_{\chi\in\widehat{\BR_+k^{\times}\backslash\BA^{\times}}}\int_{-\infty}^{\infty} \sum_{e_1,e_2\in\CB(\chi,\omega\chi^{-1})}\CZ\left(s+\frac{1}{2},\Phi,\beta_{-i\tau}(e_1^{\vee},e_2)\right)\BW^{\psi^{-1}}(i\tau,e_1)(\RI_2)\BW^{\psi}(-i\tau,e_2^{\vee})(\RI_2)  \frac{\ud \tau}{4\pi}.
\end{align*}
By making $\tau\mapsto-\tau$ and change the role of $e$ with $e^{\vee}$, we have the above is 
\begin{align*}
	&\sum_{\chi\in\widehat{\BR_+k^{\times}\backslash\BA^{\times}}}\int_{-\infty}^{\infty} \sum_{e_1,e_2\in\CB(\chi,\omega\chi^{-1})}\CZ\left(s+\frac{1}{2},\Phi,\beta_{i\tau}(e_2,e_1^{\vee})\right)\BW^{\psi^{-1}}(i\tau,e_1)(\RI_2)\BW^{\psi}(-i\tau,e_2^{\vee})(\RI_2)  \frac{\ud \tau}{4\pi}\\
	&=\CI(s)-\left(  -\frac{1}{2}\CI_{\rb}(s) \right)+\left( \frac{1}{2}\CI_{\rb}(s) \right)-\frac{1}{2}\widehat{\CI}_{\rb}(s)+\left(-\frac{1}{2}\widehat{\CI}_{\rb}(s)\right),
	\end{align*}
	i.e.,
	\[
	\widehat{\CI}_{\mathrm{Eis}}(s)=\CI_{\mathrm{Eis}}(s)+\CI_{\mathrm{b}}(s)-\widehat{\CI}_{\mathrm{b}}(s).
	\]
Combining with Proposition \ref{prp_fqI} and Theorem \ref{thm_tf}, we obtain the functional equation of the cuspidal parts:
\begin{cor}
	\begin{align*}
		\CI_{\cusp}(s)=\widehat{\CI}_{\cusp}(s).
	\end{align*}
\end{cor}

\section{Isolation of the Spectrum}\label{sec_iso}

Let $\RS$ be a finite set of places of $k$ containing all the Archimedean places. Denote by $\CF(\Mat_2(\BA))_{\RS}$ the subspace of Schwartz functions on $\Mat_2(\BA)$ consisting of vectors of the form
\begin{align*}
	\Phi_{\RS}\otimes_{\nu\notin\RS} 1_{\Mat_2(\Fo_{\nu})}. 
\end{align*}
For simplicity, denote $\Phi^{\circ,\RS}:=\otimes_{\nu\notin\RS} 1_{\Mat_2(\Fo_{\nu})}$ and write $\Phi_{\RS}$ as elements in $\CF(\Mat_2(\BA))_{\RS}$ by tensoring this basic vector. Let $\CH^{\RS}$ be the unramified Hecke algebra outside $\RS$
\[
\CH^{\RS}:=\otimes^{\prime}_{\nu\notin\RS}\CH(\RG(k_{\nu}),\RK_{\nu}).
\]
Then we get a family of functionals first when $\Re(s)>1$
\begin{align*}
	\mathrm{KTF}^{\RS}(s):\CH^{\RS}\otimes\CF(\Mat_2(\BA))_{\RS}&\rightarrow\BC\\
	h\otimes\Phi_{\RS}&\mapsto \int_{[\RZ]}\mathrm{KTF} (\RR_z(\Phi_{\RS}\otimes (h\star\Phi^{\circ,\RS})(x)(\ud^+x)^{\frac{1}{2}}  ))\omega(z)|z|^{2s}  \ud z,
	\end{align*}
 By analytic continuation from Section \ref{sec_cont}, we obtain a family $\mathrm{KTF}^{\RS}(s)$ for $s\in\BC$.

By Satake isomorphism, we have for each non-Archimedean place $\nu$,
\[
\CH(\RG(k_{\nu}),\RK_{\nu})\cong\BC[\RG^{\vee}\sslash\RG^{\vee}],
\]
where the action is by conjugation. Hence 
\[
\CH^{\RS}\cong\BC\left[\prod_{\nu\notin\RS}(\RG^{\vee}\sslash\RG^{\vee})   \right],
\]
where regular functions are restricted tensor products of regular functions on finite factors with respect to the constant function $1$. This isomorphism is denoted by $h\mapsto \widehat{h}$. Since every automorphic representation that is unramified outside of $\RS$ determines a point on $\displaystyle{\prod_{\nu\notin\RS}\RG^{\vee}\sslash\RG^{\vee}}$ and is determined by it due to strong multiplicity one. Let 
\[
\RU^{\RS}\subset\prod_{\nu\notin\RS}\RG^{\vee}\sslash\RG^{\vee}
\]
be the subset corresponding to the unitary ones with given central character $\omega$, which is compact with the product Hausdorff topology. Let \[
\RR=\{[\pi(\omega|\cdot|^{s+\frac{1}{2}},|\cdot|^{-s-\frac{1}{2}})]: s\in\BC    \}.
\]
We write
\[
\BC[\RU^{\RS}\cup \RR]=\BC\left[\prod_{\nu\notin\RS}(\RG^{\vee}\sslash\RG^{\vee})   \right],
\]
by restriction. Then we have
\begin{lem}
	$\BC[\RU^{\RS}\cup\RR]$ is dense on the space $\BC(\RU^{\RS}\cup\RR)$ of continuous functions on $\RU^{\RS}\cup\RR$, where the latter is equipped with the inductive limit topology of uniform convergence on compact subsets.
\end{lem}

\begin{proof}
	According to the choice of the topology here, it suffices to show for any compact subset $\RC$, the restriction of $\BC[\RU^{\RS}\cup\RR]$ to $\RC$ is dense in $\BC(\RC)$, and it suffices to show polynomial functions in $\BC[\RU^{\RS}\cup\RR]$ are closed under complex conjugation, which is \cite[Lemma 5.1.1]{Sak19a}.
\end{proof}

Fix an element $\Phi_{\RS}\in\CF(\Mat_2(\BA))_{\RS}$, then we can view $\mathrm{KTF}^{\RS}(s)$ as a functional on $\CH^{\RS}$, which can be extended to a continuous functional on $\BC(\RU^{\RS}\cup\RR)$, hence a finite measure with compact support:

\begin{lem}
	$\mathrm{KTF}^{\RS}(s)$ can be extended to a finite measure with compact support on $\BC(\RU^{\RS}\cup\RR)$.
\end{lem}

\begin{proof}
    To avoid confusion, for a Schwartz function $\Phi$ on $\Mat_2(\BA)$, denote $\RK_{\Phi}(x,y)$ to be the kernel function associated to $\Phi$ defined by
    \[
    \RK_{\Phi}(x,y):=\int_{[\RZ]}\sum_{\gamma\in\GL_2(k)}\Phi(x^{-1}\gamma yz)\omega(z)|z|^{2s+2}\ud^{\times}z\cdot |\det x^{-1}y|^{s+1}
    \]
    as in the previous sections. And similarly for $\RK^{\prime}_{\Phi}(x,y)$, $\widehat{\RK}_{\widehat{\Phi}}(x,y)$ and $\widehat{\RK}^{\prime}_{\widehat{\Phi}}(x,y)$.
    
	When $\Re(s)>1$, according to \cite[Lemma 3.25]{Wu17}, we can choose some compact subset $\Omega$ of $\RG(\BA)$ containing a fundamental domian of $[\RN]$, then we have for $y\in\Omega$
	\begin{align}\label{eq_k}
	|\RK_{\Phi_{\RS}\otimes h\star\Phi^{\circ,\RS}}(x,y)|\ll \left( \|(1+\Delta_{\infty})^{N+A}\RK_{\Phi}(x,\cdot )\|_{\CL^2(\GL_2,\omega)}+\sum_{\chi^2=\omega}\|\RF_{\chi\circ\det }(\RK(x,\cdot)) \| \right)\cdot \|\widehat{h}\|_{\CL^\infty(\RU^{\RS})},
	\end{align}
	where $\Phi=\Phi_{\RS}\otimes\Phi^{\circ,\RS}$, $N\in\BN$ depends only on the field $k$, $A$ is some real number large enough, $\Delta_{\infty}$ is the Laplacian operator defined as in \cite[Section 3.5]{Wu17}, and $\RF_{\chi\circ\det}$ is the spectrum projectors defined as in \cite[Section 2.2]{Wu17}. Note that the right-hand-side of (\ref{eq_k}) is a continuous function for the $x$-variable, then after we integrate it over $[\RN]$, we see
	\[
	 \int_{[\RN]\times[\RN]} \RK(n_1,n_2)\psi^{-1}(n_1^{-1}n_2)  \ll \|\widehat{h} \|_{\CL^{\infty}(\RU^{\RS})},
	\]
	where the implict constant depends on $\Phi_{\RS}$, $s$ and the underground field $k$. Also note that according to \cite[Theorem 2.10]{Wu}, we have \[\displaystyle{ (1+\Delta_{\infty})^{N+A}\RK_{\Phi}(x, y)\ll \mathrm{Ht}(y)^{1-s}   },
	\]
	where the hight function $\mathrm{Ht}$ is as defined in \cite[Section 1.2]{Wu}. Then, when $\Re(s)$ is fixed, the implicit constant can be chosen to be independent of $s$.
	
	According to Section \ref{ssec_geo}, as the kernel function $\RK_{ \Phi_{\RS}\otimes(h\star\Phi^{\circ,\RS}) }^{\prime}(x,y)$ is an Eisenstein series and the action of $h$ is given by the value of the delta distribution of $\widehat{h}$ on the principal series $\pi(\omega|\cdot|^{s+\frac{1}{2}},|\cdot|^{-s-\frac{1}{2}})$, we get the desired extension when $\Re(s)>1$.
	
	When $\Re(s)<-1$, combine Proposition \ref{prp_fqI} and the fundamental lemma Proposition \ref{prp_hankel basic}, we have
	\begin{align}\label{eq_kk}
		\mathrm{KTF}^{\RS}(s)(h)\ll \| \widehat{h}\|_{\CL^{\infty}(\RU^{\RS})}+|\widehat{h}(\pi(\omega|\cdot|^{s-\frac{1}{2}},|\cdot|^{-s+\frac{1}{2}}))|,
	\end{align}
	which means $\mathrm{KTF}^{\RS}(s)$ is also a finite measure of compact support.
	
	When $-1\leq\Re(s)\leq 1$, we will use the Phragmen-Lindelof principle to bound $\mathrm{KTF}^{\RS}(s)$. First let us fix $\epsilon>0$ and consider the case when $\Re(s)=1+\epsilon$. As $\Re(s)=1+\epsilon$ is fixed, there is some compact set $\RC^{\prime}$ of 
	\[
	\prod_{\nu\notin\RS} (\RG^{\vee}\sslash\RG^{\vee})
	\]
	such that $\RC^{\prime}$ contains $\displaystyle{ \{[\pi(\omega|\cdot|^{s+\frac{1}{2}},|\cdot|^{-s-\frac{1}{2}})] : s\in\BC, \Re(s)=1+\epsilon \} }$, hence
	\[
	\| \widehat{h}\|_{\CL^{\infty}(\RU^{\RS})}+|\widehat{h}(\pi(\omega|\cdot|^{s-\frac{1}{2}},|\cdot|^{-s+\frac{1}{2}}))|\leq \|\widehat{h}\|_{\CL^{\infty}(\RC)},\;\Re(s)=1+\epsilon,
	\]
	where $\RC:=\RU^{\RS}\cup\RC^{\prime}$ is compact. Moreover, when $\Re(s)=1+\epsilon$, the growth of 
	\[
	(1+\Delta_{\infty})^{N+A}\RK_{\Phi}(x,\cdot)
	\]
	on any Seigel domain is bounded by $\mathrm{Ht}(y)^{-\epsilon}$, so are $\|(1+\Delta_{\infty})^{N+A}\RK_{\Phi}(x,\cdot )\|_{\CL^2(\GL_2,\omega)}
$ and $\|\RF_{\chi\circ\det }(\RK(x,\cdot))\| $. Combine them together, we have
	\[
	|\mathrm{KTF}^{\RS}(s)(h)|\ll\|\widehat{h}\|_{\CL^{\infty}(\RC)},
	\]
	when $\Re(s)=1+\epsilon$, with the implicit constant indepdent of $s$, i.e., there is some constant $\RB^{\prime}$ such that $|\mathrm{KTF}^{\RS}(s)(h)|\leq \RB^{\prime}\|\widehat{h}\|_{\CL^{\infty}(\RC)}$ for any $s$ with $\Re(s)=1+\epsilon$. A similar argument applied to the case that $\Re(s)=-1-\epsilon$ due to Theorem \ref{thm_psf}, with replacing $\Phi$ by $\widehat{\Phi}$. Then we obtain that there is some constant $\RB>0$ such that
	\[
	|\mathrm{KTF}^{\RS}(s)(h)|\leq \RB\|\widehat{h}\|_{\CL^{\infty}(\RC)},\;\Re(s)=1+\epsilon\;\mathrm{or}\;-1-\epsilon.
	\]
	
	Next let us show that $\mathrm{KTF}^{\RS}(s)(h)$ is bounded if $h$ fixed when $-1-\epsilon\leq \Re(s)\leq 1+\epsilon$. According to the Lemma \ref{lem_es}, when $-1-\epsilon\leq \Re(s) \leq 1+\epsilon$, we have
	\begin{align*}
		&\int_{z\in [\RZ],|z|\geq 1}\left| \mathrm{KTF}( \RR_z(  \Phi_{\RS}\otimes h\star\Phi^{\circ,\RS})(\ud x)^{\frac{1}{2}} )\omega(z)|z|^{2s}\right|\ud^{\times}z\\
		&\leq \int_{z\in [\RZ],|z|\geq 1} \int_{\RN(\BA)\times\RN(\BA)}\sum_{t_1\in k^{\times},t_2\in k} \left| \BF_{\psi,2}(   \Phi_{\RS}\otimes(h\star\Phi^{\circ,\RS}) )\left(  \begin{pmatrix}
			t_1n_1z & \frac{t_2}{z} \\ t_1z & t_1n_2z
		\end{pmatrix}   \right)   \right| |z|^{2s+2}\ud n_1\ud n_2 \ud^{\times}z\\
		&+\int_{z\in \RZ(\BA),|z|\geq 1}\int_{\RN(\BA)}\left| \Phi\left( \begin{pmatrix}
			z & zn \\ 0 & z
		\end{pmatrix}    \right)   \right|   |z|^{2s+2}\ud^{\times}z .\\
	\end{align*}
According to Lemma \ref{lem_gj}, there is some function $\widetilde{\Phi}^{\prime}$ of Schwrtz type such that
\[
\sum_{,t_2\in k} \left| \BF_{\psi,2}(   \Phi_{\RS}\otimes(h\star\Phi^{\circ,\RS}) )\left(  \begin{pmatrix}
			t_1n_1z & \frac{t_2}{z} \\ t_1z & t_1n_2z
		\end{pmatrix}   \right)   \right|\ll (1+|z|)\widetilde{\Phi}^{\prime}(t_1z),
\]
hence there is another Schwartz function $\widetilde{\Phi}$ such that
\begin{align*}
&\int_{z\in [\RZ],|z|\geq 1}\left| \mathrm{KTF}( \RR_z(  \Phi_{\RS}\otimes h\star\Phi^{\circ,\RS})(\ud x)^{\frac{1}{2}} )\omega(z)|z|^{2s}\right|\ud^{\times}z \\
&\ll \int_{z\in\BZ(\BA),|z|\geq 1}(1+|z|)\widetilde{\Phi}(z)(|z|^{2s}+|z|^{2s+1})\ud^{\times}z\ll1,\;-1-\epsilon\leq \Re(s)\leq 1+\epsilon.
\end{align*}
Apply the same argument to $\widetilde{\Phi}$, and make use of Proposition \ref{prp_fqI}, we have
\begin{align*}
   & \mathrm{KTF}^{\RS}(s)(h)=\int_{z\in [\RZ],|z|\geq 1}\mathrm{KTF}(\RR_z((\Phi_{\RS}\otimes h\star\Phi^{\circ,\RS})(x)(\ud^+x)^{\frac{1}{2}}  ))\omega(z)|z|^{2s}\ud^{\times}z \\
	&+ \int_{|z|<1}\mathrm{KTF}(\RR_z(\widehat{\Phi_{\RS}\otimes h\star\Phi^{\circ,\RS}}(x)(\ud^+x)^{\frac{1}{2}}  ))\omega(z)|z|^{2s}\ud^{\times}z \ll 1,\;-1-\epsilon\leq \Re(s)\leq 1+\epsilon.
\end{align*}

Finally, fix $h$ and apply the Phragmen-Lindelof principle (\cite[Theorem 12.8]{Rud87}) to the holomorphic function $s\mapsto \mathrm{KTF}^{\RS}(s)(h)$, we have
	\[
	|\mathrm{KTF}^{\RS}(s)(h)|\leq \RB\|\widehat{h}\|_{\CL^{\infty}(\RC)},\;-1-\epsilon\leq \Re(s)\leq 1+\epsilon,
	\]
	which means $\mathrm{KTF}^{\RS}(s)$ is also a finite measure when $-1\leq \Re(s)\leq 1$.
	
	\end{proof}

We denote $\widehat{\RG}^{\mathrm{aut}}$, the set of automorphic representations appearing in the Plancherel formula for $\CL^2(\GL_2,\omega)$. We consider $\widehat{\RG}^{\mathrm{aut}}$ as a subset of
\[
\lim_{\rightarrow}\RU^{\RS}
\]
in the obvious sense, and we freely talk about $\widehat{\RG}^{\mathrm{aut}}$ as a subset of $\RU^{\RS}$, meaning its intersection with $\RU^{\RS}$. We denote $\widehat{\RG}^{\cusp}$ the subset of $\widehat{\RG}^{\mathrm{aut}}$ consisting of cuspidal representations and $\widehat{\RG}^{\Eis}$ the subset of principle series. Then $\widehat{\RG}^{\cusp}$ is a countable number points and $\widehat{\RG}^{\Eis}$ is a countable union of lines. Both of them are measurable sets with respect to the standard Borel structure on $\RU^{\RS}$.

Then, combined with the spectral expansion, we have
\begin{thm}
	For fixed $\Phi_{\RS}$ and all $s\in\BC$, The functional $\mathrm{KTF}^{\RS}(s)$ is a finite complex measure $\nu_{\Phi_{\RS}}$ on $\widehat{\RG}^{\cusp}\cup\widehat{\RG}^{\Eis}\cup\RR$. The measure $\nu_{\Phi_{\RS}}$ is equal to 
	\[
	\sum_{\varphi_1.\varphi_2\in\CB(\pi)}\CZ\left(s+\frac{1}{2},\Phi,\beta(\varphi_2,\varphi_1^{\vee})\right)\BW_{\varphi_1}^{\psi^{-1}}(\RI_2)\BW_{\varphi_2^{\vee}}^{\psi}(\RI_2)\ud \pi
	\]
	on $\widehat{\RG}^{\mathrm{cusp}}$ when $\Re(s)>1$. It admits an analytic continuation to all $s\in\BC$ and satisfies the functional equation
	\begin{align*}
		&\sum_{\varphi_1.\varphi_2\in\CB(\pi)}\CZ\left(s+\frac{1}{2},\Phi,\beta(\varphi_2,\varphi_1^{\vee})\right)\BW_{\varphi_1}^{\psi^{-1}}(\RI_2)\BW_{\varphi_2^{\vee}}^{\psi}(\RI_2)\\
		&=\sum_{\varphi_1.\varphi_2\in\CB(\pi)}\CZ\left(-s+\frac{1}{2},\widehat{\Phi},\beta^{\vee}(\varphi_2,\varphi_1^{\vee})\right)\BW_{\varphi_1}^{\psi^{-1}}(\RI_2)\BW_{\varphi_2^{\vee}}^{\psi}(\RI_2).
	\end{align*}
	In particular, we get the analytic continuation of $L(s,\pi,\st)$. Moreover, 
	\[
	L\left(s+\frac{1}{2},\pi,\st\right)=\epsilon\left(s+\frac{1}{2},\pi \right)L\left(-s+\frac{1}{2},\pi,\st^{\vee}\right).
	\]
	where 
	$L\left(s+\frac{1}{2},\pi,\st\right)$ is the complete $L$-function.
	
\end{thm}

\begin{proof}
	We only need to prove the analytic continuation and functional equation of $L$-functions. Write $\pi=\otimes^{\prime}_{\nu}\pi_{\nu}$ according to \cite{Fla79}. For $\varphi\in\CB(\pi)$, write 
	\[
	\varphi=\otimes^{\prime}_{\nu\in|k|}\varphi_{\nu}.
	\]
	Then we have an Euler product
	\[
	\CZ\left( s+\frac{1}{2},\Phi,\beta(\varphi_2,\varphi_1^{\vee}) \right)=\prod_{\nu\in|k|}\CZ_{\nu}\left(s+\frac{1}{2},\Phi_{\nu},\beta(\varphi_{2,\nu},\varphi_{1,\nu}^{\vee}) \right),
	\]
	where
	\[
	\CZ_{\nu}\left(s+\frac{1}{2},\Phi_{\nu},\beta(\varphi_{2,\nu},\varphi_{1,\nu}^{\vee}) \right):=\int_{\RG(k_{\nu})}\Phi_{\nu}(g)\beta(\varphi_{2,\nu},\varphi_{1,\nu}^{\vee})(g)|\det g|_{\nu}^{s+1}\ud g.
	\]
	 For non-Archimedean $\nu$, let $\varphi_{\nu}^{\circ}$ be the unique newvector of $\pi_{\nu}$ as in \cite[Theorem (5)]{JPSS81} up to scalar. For Archimedean $\nu$, let $\varphi^{\circ}_{\nu}$ be the unique highest weight vector of ${\pi}_{\nu}$. Denote $\varphi^{\circ}=\otimes^{\prime}\varphi_{\nu}^{\circ}$. Then choose $\Phi^{\circ}=\otimes^{\prime}_{\nu}\Phi^{\circ}_{\nu}$ such that $\Phi^{\circ}_{\nu}$ is as in \cite[Theorem 1.2]{Hum21} and as in \cite[Theorem 1.1]{Lin18}. Then, since the $\RK$-type of $\Phi^{\circ}$ is fixed, according to the uniqueness of the newvector in the non-Archimedean case and the highest weight vector in the non-Archimedean case, we have
	\[
	\CZ\left( s+\frac{1}{2},\Phi^{\circ},\beta(\varphi_2,\varphi_1^{\vee}) \right)\neq 0
	\]
	if and only if $\varphi_1=\varphi_2=\varphi^{\circ}$. Moreover, in this case
	\[
	\CZ\left( s+\frac{1}{2},\Phi^{\circ},\beta(\varphi^{\circ},\varphi^{\circ}) \right)=L\left(s+\frac{1}{2},\pi,\st\right).
	\]
	Then we get the analytic continuation of $L(s,\pi,\st)$. 
	
	As for the functional equation, according to \cite[Theorem 3.3, Proposition 6.12]{GJ72}, there is an epsilon factor $\epsilon(s,\pi)$ such that
		\begin{align*}
	\sum_{\varphi_1.\varphi_2\in\CB(\pi)}\CZ\left(-s+\frac{1}{2},\widehat{\Phi^{\circ}},\beta^{\vee}(\varphi_2,\varphi_1^{\vee})\right)&\BW_{\varphi_1}^{\psi^{-1}}(\RI_2)\BW_{\varphi_2^{\vee}}^{\psi}(\RI_2)\\
	&=\BW_{\varphi^{\circ}}^{\psi^{-1}}(\RI_2)\BW_{\varphi^{\circ,\vee}}^{\psi}(\RI_2)\epsilon\left(s+\frac{1}{2},\pi \right) L\left(-s+\frac{1}{2},\pi^{\vee},\st\right),
	\end{align*}
	we may get 
	\[
   	L\left(s+\frac{1}{2},\pi,\st\right)=\epsilon\left(s+\frac{1}{2},\pi \right)L\left(-s+\frac{1}{2},\pi,\st^{\vee}\right).
	\]
	\end{proof}

\bibliographystyle{alpha}
	\bibliography{references}
\end{document}